\documentclass[11pt,reqno]{amsart}
\usepackage{a4wide}
\usepackage{amsmath,amsfonts,amssymb,bbm}
\usepackage{verbatim}
\usepackage[shortlabels]{enumitem}
\usepackage[T1]{fontenc}
\usepackage[unicode,hypertexnames=false,colorlinks=true,linkcolor=blue,citecolor=blue]{hyperref}
\usepackage{esint}
\usepackage{hyperref}
\usepackage[capitalise, nameinlink, noabbrev]{cleveref} 
\usepackage{mathrsfs}
\usepackage{enumitem}
\hypersetup{colorlinks,breaklinks,
	linkcolor=[rgb]{0,0,0},
	citecolor=[rgb]{0,0,0},
	urlcolor=[rgb]{0,0,0}}

\usepackage{xfrac,xcolor,color,graphicx}
\newcommand{\be}{\begin{equation}}
	\newcommand{\ee}{\end{equation}}
\theoremstyle{plain}

\newtheorem{theorem}{Theorem}[section]
\newtheorem{proposition}[theorem]{Proposition}
\newtheorem{corollary}[theorem]{Corollary}
\newtheorem{lemma}[theorem]{Lemma}
\newtheorem{definition}[theorem]{Definition}
\newtheorem{condition}{Condition}
\theoremstyle{remark}
\newtheorem{remark}[theorem]{Remark}

\newcommand{\C}{\mathbb{C}}
\newcommand{\R}{\mathbb{R}}

\newcommand{\N}{\mathbb{N}}

\newcommand{\eps}{\varepsilon}

\newcommand{\mathbbmm}[1]{\text{\usefont{U}{bbm}{m}{n}#1}}
\newcommand{\ind}{\mathbbmm{1}}
\newcommand{\HH}{\mathcal{H}}
\newcommand{\loc}{{{\tiny{\text{loc}}}}}

\def\O{\Omega}

\newcommand{\bea}{\begin{equation*}\begin{aligned}}
		\newcommand{\eea}{\end{aligned}\end{equation*}}
\usepackage{scalerel}


\crefname{subsection}{Subsection}{Subsections}

\numberwithin{equation}{section}

\numberwithin{equation}{section}

\title{Uniqueness of the blow-up for some Alt-Phillips cones}

\author[M. Carducci]{Matteo Carducci}\thanks{}
\address {Matteo Carducci \newline \indent
	Classe di Scienze, Scuola Normale Superiore \newline \indent
	Piazza dei Cavalieri 7, 56126 Pisa - ITALY}
\email{\href{mailto:matteo.carducci@sns.it}{matteo.carducci@sns.it}}
\author[G. Tortone]{Giorgio Tortone}\thanks{}
\address {Giorgio Tortone \newline \indent
	Dipartimento di Matematica ``Giuseppe Peano'', Università di Torino \newline \indent
	Via Carlo Alberto 10, 10123 Torino - ITALY}
\email{\href{mailto:giorgio.tortone@unito.it}{giorgio.tortone@unito.it}}

\begin{document}
\begin{abstract}
We establish uniqueness of blow-ups, with sharp quantitative convergence, for several classes of singular minimizing cones in the Alt-Phillips problem, in the range $\gamma\in(0,2)$. As a consequence, we obtain uniqueness at every free boundary point in dimensions $d=2,3,4$ for $\gamma\in(1,2)$, and in dimensions $d\geq 5$ for $\gamma\in\left(1,\frac32\right)$.

The proof of uniqueness is based on three new logarithmic epiperimetric inequalities. The sharp distinction between polynomial and logarithmic convergence is governed by a finite-dimensional integrability condition (sub-integrability) for the spherical linearized problem. 

We prove this sharpness for the radial cone and its cylindrical extensions through an explicit integrability and bifurcation analysis, showing in particular that logarithmic convergence may be sharp even in dimension two.
In contrast, we show that the one-dimensional cone is exceptional: although the integrability condition fails, the convergence is polynomial. 

Finally, we characterize the minimality of the radial cone in terms of $d$ and $\gamma$ by means of a one-dimensional calibration argument, exhibiting in dimension $d\geq6$ a nontrivial regime in which the radial cone is stable but not minimizing.
\end{abstract}
\maketitle
	\tableofcontents
	\section{Introduction}
The uniqueness of blow-up limits at singular points is one of the central questions in the regularity theory of minimal surfaces and free boundaries.
A unique blow-up, together with a rate of convergence, gives quantitative control of the original object as a deformation of its tangent cone, providing a key ingredient in the analysis of the singular set \cite{AA81,Simon,SimonCylindric,Caffarelli:obstacle-problem-revisited,coldingminicozzi1,coldingminicozzi2,gp09,csv18,fs19,csv20,esv}. 

In this paper, we establish for the first time uniqueness of the blow-up for some singular
minimizing cones of the Alt-Phillips functional
\be\label{eq:AP}
\mathcal{J}_\gamma(u):=\int_{B_1}\Big(|\nabla u|^2+u^\gamma\ind_{\{u>0\}}\Big)\,dx,
\quad \text{where }\gamma\in(-2,2).
\ee
We also identify a sharp criterion for polynomial convergence in terms of a new notion of \emph{sub-integrability}. Perhaps more surprisingly, we also exhibit a non-integrable singular cone for which polynomial convergence still holds.

\medskip

The functional \eqref{eq:AP} was first introduced by Phillips \cite{phillips} and Alt-Phillips \cite{Alt-Phillips}, and has been investigated extensively in recent years, both for positive exponents \cite{DeSilvaSavinPositivePower,fk24,generic-alt-phillips,aks25,thomas,ros-restrepo,savinyu-altphillips-stable,savinyu-altphillips-concentration,Fernandez-classification} and for negative ones \cite{DeSilvaSavinNegativePower,DeSilvaSavinNegativePowerUniformDensity,DeSilvaSavinNegativePowerCompactness,CT1}. 
    As particular cases, we find the obstacle problem $\gamma=1$, the Alt-Caffarelli problem $\gamma=0$, and we recover minimal surfaces when $\gamma\to{-2}$ \cite{DeSilvaSavinNegativePowerUniformDensity,DeSilvaSavinNegativePowerCompactness}.
    In this paper, we only focus on the case $\gamma\in(0,2)$.
    
    Let $u\in H^1(B_1)$ be a minimizer of \eqref{eq:AP}, and suppose that $0\in\partial\Omega_u$ is a free boundary point, where $ \O_u:=\{u>0\}$. We consider the rescalings 
$$
u_r(x):=\frac{u(rx)}{r^\beta},\quad\text{where}\quad\beta:=\frac{2}{2-\gamma}.
$$ 
By Weiss' monotonicity formula \cite{weiss-alt-phillips-monotonicity-formula}, the sequence $u_r$ converges along a subsequence to some blow-up $b$, which is a $\beta$-homogeneous minimizer of the Alt-Phillips functional. When $b$ is the flat solution $c_\gamma (x\cdot \nu)_+^\beta$, for $\nu\in\mathbb{S}^{d-1}$, everything is known: the blow-up is unique and the free boundary is locally given by the graph of a smooth function \cite{Alt-Phillips,DeSilvaSavinPositivePower,ros-restrepo,CT1}.
Such points are called regular, and we denote by $\text{Reg}(u)$ the set of regular points. 

\medskip

Regarding the singular set $\mathrm{Sing}(u):=\partial \Omega_u\setminus\text{Reg}(u)$, the situation is much more delicate. By \cite{Alt-Phillips,weiss-alt-phillips-monotonicity-formula} it is known that, for $\gamma \in (0,1)$, the singular part of the free boundary is a closed set of Hausdorff dimension at most $d-3$, whereas for $\gamma \in (1,2)$ singularities may arise even in two dimensions. In the latter case, by \cite{bonorino}, the singular set $\mathrm{Sing}(u)$ is rectifiable.

The first classification of singular blow-ups was carried out in dimension two and for $\gamma \in (1,2)$ in \cite{bonorinoetall} using ODE techniques. Very recently, the singularities have been studied from different perspectives: in terms of generic regularity \cite{generic-alt-phillips}, through the construction of singular minimizing cones \cite{savinyu-altphillips-stable,savinyu-altphillips-concentration}, and through rigidity results in low dimensions \cite{Fernandez-classification}. 

Nevertheless, the question of whether or not blow-ups at singular points are unique has remained completely open to date. 
Indeed, a priori, it is possible that around a singular point, the free boundary asymptotically approaches different cones at different scales.
\smallskip

\subsection{Uniqueness of the blow-up} We proceed by proving the first uniqueness result for the following class of blow-ups. Let $\mathcal{B}$ be the set of singular minimizing cones, and denote by
 \bea\label{eq:stratification-Bell}\mathcal{B}_\ell:=\left\{b\in\mathcal{B}:\ b(y,z)=B(y),\ (y,z)\in\R^{d-\ell}\times \R^\ell,\ B>0 \text{ in }\R^{d-\ell}\setminus \{0\}\right\},
\eea
the class of cylindrical extensions, for $\ell=0,\dots,d-1$. Then, we consider blow-ups satisfying the following condition.
\begin{condition}\label{condition:assumption-uniqueness}
        We suppose that $b$ is a minimizing cone belonging to one of the following classes.
        \begin{itemize}
        \item[(i)] Positive cones: $b \in \mathcal{B}_0$, namely $b>0$ in $\R^{d}\setminus\{0\}$.
        \item[(ii)] One-dimensional cone: $b\in \mathcal{B}_{d-1}$, namely $b$ is a rotation of the one-dimensional cone $$
        b_{\text{one}}(x):=c_{\text{one}}|x_d|^\beta\quad\text{where}\quad c_{\text{one}}^{\gamma-2}= \beta^2.
        $$
        \item[(iii)] Translational cylindrical cones for $\gamma \in (1,2)$: $b\in \mathcal{B}_{\ell}$, for some $\ell=1,\ldots,d-2$, is a cylindrical extension of a translational cone $B$. Namely, every $(\beta-1)$-homogeneous Jacobi field of $B$ is generated by translations of $B$, see \cref{def:beta-1translational}.
        \end{itemize}
\end{condition}
We refer to \cref{sub:parabola} and \cref{sub:bifurcation-intro} for examples of cones satisfying \cref{condition:assumption-uniqueness}, and we refer to (i) in \cref{subsub:role-condition1} for a discussion about \cref{condition:assumption-uniqueness}.

\smallskip

Our main result is the following uniqueness theorem for singular blow-ups, with a quantitative logarithmic-type convergence. 
	\begin{theorem}\label{t:uniqueness}
		Let $u\in H^1(B_1)$ be a minimizer of the Alt-Phillips problem, and $b$ be a blow-up of $u$ at $0\in\partial\Omega_u$ satisfying \cref{condition:assumption-uniqueness}. Then, $b$ is the unique blow-up and, for some $r_0>0$, \be\label{eq:log-conv}\|u_{r}-b\|_{L^\infty(B_1)}\le \frac{C}{|\log r|^\alpha}\quad\text{for every }r\in(0,r_0),\ee for some $C>0$ and $\alpha\in(0,1)$ depending only on $d$, $\gamma$ and $b$.
        \end{theorem}
For $\gamma\in(1,2)$, by \cite{Alt-Phillips,bonorino}, every minimizing cone can be written as a cylindrical extension of some positive cone, and so $\mathcal{B}=\cup_{\ell=0}^{d-1}\mathcal{B}_\ell$. In this regime, the only additional hypothesis in \cref{condition:assumption-uniqueness} is the translational assumption.

In certain ranges of the parameters $d$ and $\gamma$, the translationality can be proved for all possible singular minimizing cones (see \cref{prop:trans0}), and thus our main result \cref{t:uniqueness} applies to every cone. 
The following corollary gives, in these cases, a complete uniqueness statement, in particular in low dimensions, and a complete regularity description for free boundaries.

\begin{corollary}\label{t:ap-low-dimension}
        Let $u\in H^1(B_1)$ be a minimizer of the Alt-Phillips problem and suppose that \be\label{eq:assumption-dimension}\text{either}\quad d=2,3,4\text{ and } \gamma \in (1,2)\quad\text{or}\quad d\ge 5\,\text{ and }\,\gamma\in\left(1,\frac32+\frac{1}{2(d-1)}\right).\ee Then, at every free boundary point of $u$, the blow-up is unique. 
        Moreover,  if we set
        $$
        \Sigma_{\ell}(u) := \left\{x_0 \in \mathrm{Sing}(u) : \text{the blow-up of $u$ at $x_0$ belongs to }  \mathcal{B}_\ell\right\},\quad \text{for }\ell=0,\dots,d-1,
        $$
        then we have
        $$
        \mathrm{Sing}(u)=
        \bigcup_{\ell=0}^{d-1} \Sigma_\ell(u),$$ where 
        $\Sigma_0(u)$ is locally discrete, $\Sigma_{d-1}(u)$ is locally covered by a $C^{1,\alpha}$ $(d-1)$-dimensional manifold, and $\Sigma_\ell(u)$ is locally covered by a $C^{1,\log}$ $\ell$-dimensional manifold, for every $\ell=1,\ldots,d-2$.
        \end{corollary}
It is important to emphasize that the uniqueness results above are not based on an explicit classification of singular minimizing cones. This point is particularly relevant, since blow-up profiles are not known explicitly in general. 

We also point out that the $C^{1,\alpha}$ regularity result for $\Sigma_{d-1}(u)$ follows from point $(i)$ of \cref{thm:parabola-cones} below.

\subsection{Sharp convergence}      
        A natural question arising from \cref{t:uniqueness} is whether the logarithmic convergence in \eqref{eq:log-conv} is optimal, or whether it can be improved to a polynomial rate. 
        In our setting, the sharp dichotomy between these two rates is governed by the \emph{sub-integrability} condition, a property weaker than \emph{integrability}. This notion is introduced in our companion paper \cite{CT3} in the context of nonlinear evolution equations.

        More precisely, we say that $b$ is \emph{integrable} if every $\beta$-homogeneous Jacobi field is generated by a one-parameter family of $\beta$-homogeneous solutions (see \cref{def:integrability}).   
By \cite{adamsimon}, after a Lyapunov-Schmidt reduction, the equation for $\beta$-homogeneous solutions near $b$ is reduced to a finite-dimensional equation on the kernel of the spherical linearized operator. In this setting, the integrability condition is equivalent to the vanishing of the corresponding reduced functional, in a neighborhood of the origin. If the reduced functional is non-positive in a neighborhood of $0$, we say that $b$ is \emph{sub-integrable} (see \cref{def:sub-integrability}).
 
\smallskip

The following theorem establishes a sharp dichotomy between logarithmic and polynomial convergence rates. We stress that the one-dimensional cone $b_{\text{one}}$ is treated separately in \cref{thm:parabola-cones}, since $b_{\text{one}}$ is not included in the definition of sub-integrability in \cref{def:sub-integrability}.
       \begin{theorem} \label{t:rate-convergence}
        Let $b\in \mathcal{B}_\ell$ be a cone, for $\ell=0,\ldots,d-2$, satisfying \cref{condition:assumption-uniqueness}. Then:
        \begin{itemize}
        \item[(i)] if $b$ is sub-integrable, then \eqref{eq:log-conv} can be improved to $r^\alpha$;
                \item[(ii)] if $b$ is not sub-integrable, then there exists a weak solution $u$ of the Alt-Phillips problem (see \cref{def:weak-sol}) for which the logarithmic convergence in \eqref{eq:log-conv} is sharp.
        \end{itemize}
        \end{theorem}
We point out that, since the Alt-Phillips functional is convex for $\gamma \in (1,2)$, the construction in point $(ii)$ of \cref{t:rate-convergence} can be extended to minimizers. 
The construction of weak solutions that exhibits sharp logarithmic convergence to the blow-up limit is a consequence of the study of slow converging solutions to nonlinear evolution equations in \cite{CT3} (see \cref{lemma:adamsimon} and \cref{remark:intermediate-cones}).

On the other hand, in point $(i)$ of \cref{t:rate-convergence}, for intermediate cylindrical extensions, the fundamental ingredient is a \emph{partial} Lyapunov-Schmidt reduction defined in \cref{lemma:Lyapunov-Schmidt-partial}, namely a Lyapunov-Schmidt reduction performed only on the base cone $B$. For more details, we refer to \cref{sub:proofs}.
        
\subsection{The parabola cones}\label{sub:parabola}
        
        Natural examples of singular minimizing cones satisfying \cref{condition:assumption-uniqueness}, and therefore satisfying the uniqueness result in \cref{t:uniqueness}, are the radial cone 
$$
b_{\text{rad}}(x):=c_{\text{rad}}|x|^{\beta}\in\mathcal{B}_0,\quad\text{where}\quad c_{\text{rad}}^{\gamma-2}=\frac2\gamma\lambda(\beta),
$$ 
and its cylindrical extensions. 
More precisely, we consider the family of \emph{parabola cones}, for $\ell=0,\dots,d-1$ and $(y,z)\in\R^{d-\ell}\times \R^\ell$, defined as
\be\label{eq:parabolas}b_{\ell}(y,z):=c_{\ell}|y|^\beta\in\mathcal{B}_\ell, \quad\text{where}\quad c_\ell^{\gamma-2}=\frac2\gamma\lambda_{d-\ell}(\beta),\ee 
	where $\lambda_n(\beta):=\beta(\beta+n-2)$ and $\lambda(\beta):=\lambda_d(\beta)$.

    The validity of the translation condition for the family $b_{\ell}$ is proved in \cref{prop:radial-cone-is-int-traslations} and \cref{corollary:kernel-cylindrical}. Set 
        \be\label{eq:def-gammakd}\gamma_{k,d}:=2-\frac{2}{\beta_{k,d}}\in(1,2),\quad\text{where}\quad \beta_{k,d}:=\frac{k(k+d-2)}2-d+2,\quad k\in\N_{\ge3}.\ee

The following is the main result on uniqueness and rate of convergence for parabola cones. We recall that a cone $b$ is \emph{integrable through rotations} if every $\beta$-homogeneous Jacobi field is generated by rotations of the cone; see \cref{def:integrability-through-rotations} for the precise definition. 
        \begin{theorem}\label{thm:parabola-cones}
        Let $u\in H^1(B_1)$ be a minimizer of the Alt-Phillips problem. Then, at every free boundary point of $u$ admitting a parabola cone \eqref{eq:parabolas} as a blow-up, the blow-up is unique. 
        Moreover:
        \begin{itemize}
                \item[(i)] $b_{\text{one}}$ is not integrable, nevertheless \eqref{eq:log-conv} can be improved to $r^\alpha$;
                \item[(ii)] $b_{\text{rad}}$ is integrable through rotations for $\gamma\in(0,1)$, and \eqref{eq:log-conv} can be improved to $r^\alpha$;
                \item[(iii)] $b_\ell$ for $\ell=0,\ldots,d-2$ and  $\gamma\in(1,2)$ is translational, and:
                \begin{itemize}[leftmargin=*]
                    \item if $\gamma\not=\gamma_{k,d-\ell}$ for every $k\in\N_{\ge3}$, then $b_\ell$ is integrable through rotations and \eqref{eq:log-conv} can be improved to $r^\alpha$;
                    \item if $\gamma=\gamma_{k,d-\ell}$ for some $k\in\N_{\ge3}$, then $b_\ell$ is not sub-integrable, and the logarithmic convergence in \eqref{eq:log-conv} is sharp.
                \end{itemize}  
            \end{itemize}
        
        \end{theorem}
Notice that the convergence to $b_{\text{rad}}$ is of logarithmic-type also in dimension $d=2$, when $\gamma=\gamma_{k,2}$ (see also \cref{rem:caso2special}). To the best of our knowledge, this is the first regularity result for scalar free boundary problems for which a sharp logarithmic convergence holds even in dimension $d=2$. 

\medskip

We emphasize that the one-dimensional cone $b_{\text{one}}$ is exceptional and is treated separately from cylindrical extensions. 
Indeed, by \cref{remark:non-translation-one-dim-cone} and \cref{remark:non-int-one-dim-cone}, it is neither integrable nor translational. Nevertheless, and somewhat surprisingly, we can still prove a polynomial rate of convergence.
To the best of our knowledge, this provides the first example of a non-integrable cone for which such a polynomial convergence rate is available. 
As an application, this stronger asymptotic should play an important role in the study of generic regularity results, see \cite{frs20,generic-alt-phillips}.

Heuristically, the one-dimensional cone plays a role analogous to a singular object of multiplicity-two in geometric problems. Indeed, it consists of two symmetric flat solutions meeting along the same spine. 
In this sense, $b_{\mathrm{one}}$ is not a typical cylindrical extension, but it is better understood as a two-sheeted configuration, in analogy with multiplicity-two planes in minimal surface theory \cite{surveycamillo} and two-phase configurations in free boundary problems \cite{regularity-two-phase}.
This precise structure is one of the reasons why it is possible to prove a polynomial rate of convergence. We refer to \cref{sub:proofs} for more details.

\medskip

Regarding the parabola cones $b_{\ell}$, for $\ell=0,\ldots,d-2$, including the radial cone $b_0=b_{\text{rad}}$, the main obstruction to polynomial convergence is the non-trivial kernel of the spherical linearized operator, which coincides with the cylindrical extension of spherical harmonics $\mathcal{H}_k(\mathbb{S}^{d-\ell-1})$, when $\gamma = \gamma_{k,d-\ell}$. 
In particular, in the case $\ell=0$, the radial cone $b_{\mathrm{rad}}$ provides an example of a minimizing cone with an isolated singularity for which the integrability condition fails.
On the other hand, for $\gamma \not= \gamma_{k,d-\ell}$, the radial cone has trivial kernel and, for $\ell=1,\ldots,d-2$, the parabola cones are integrable through rotations, thus the convergence in \eqref{eq:log-conv} can be improved to a power-type rate. See \cref{prop:optimality-log} and \cref{corollary:kernel-cylindrical} for the precise statements.

\subsection{The minimality of the radial cone}
For $\gamma\in[1,2)$, the convexity of the Alt-Phillips functional allows one to construct minimizing cones by considering $\beta$-homogeneous critical points. The case $\gamma\in(0,1)$ is substantially more delicate and requires a finer analysis, at least in dimension $d\ge3$. Indeed, by \cite{Alt-Phillips}, there are no singular minimizing cones in dimension $d=2$, for $\gamma\in(0,1)$. The only known examples of singular minimizing cones for $\gamma\in(0,1)$ have been obtained only recently, by Savin-Yu \cite{savinyu-altphillips-stable} (see also \cite{savinyu-altphillips-concentration}).

\smallskip

In this paper, we also investigate in this direction, giving a complete characterization of the minimality of the radial cone $b_{\text{rad}}$, for every $\gamma\in(0,1)$. More precisely, following the notation in \cite{savinyu-altphillips-stable}, for $d=3,4,5,6$, we set 
\be\label{e:Delta-d-gamma}
\Delta(d,\gamma):=(d-2)^2-4(1-\gamma)\lambda(\beta),
\ee
    and let $\gamma_\Delta(d)\in(0,1)$ be the only root of $\Delta(d,\gamma)=0$ in $(0,1)$. Naturally, we extend $\gamma_\Delta(d)$ to be zero in the cases $d\geq 7$, i.e., 
    \be\label{eq:def-gamma-delta}\gamma_{\Delta}(d):=\begin{cases}
        \frac{2(2\sqrt{d-1}-d+2)}{2\sqrt{d-1}-d+4}&\quad\text{if }d=3,4,5,6,\\
        0&\quad\text{if $d\ge7$}.
    \end{cases}\ee
The numerical values are $$\gamma_\Delta(3)\approx 0.9552,\quad \gamma_\Delta(4)\approx 0.8453,\quad \gamma_\Delta(5)=\frac23, \quad \gamma_\Delta(6)\approx 0.3820.$$
By \cite{savinyu-altphillips-stable}, the radial cone $b_{\text{rad}}$ is stable if and only if $\gamma\ge \gamma_\Delta(d)$, and moreover it is minimizing for $\gamma>1-\frac{(d-2)^2}{64d^2}$.

The latter result is obtained through a delicate construction of upper and lower foliations around the radial cone. In particular, the argument in \cite[Proposition 4.1]{savinyu-altphillips-stable} reduces the minimality of $b_{\text{rad}}$ to its one-sided minimality from below (see \cref{def:minimo-alto-basso}). 

\smallskip

Using a different one-dimensional calibration argument (see \cref{lemma:minimality1} and \cref{lemma:minimality2}) and constructing explicit competitors (see \cref{lemma:minimality3}), we are able to characterize the minimality of the radial cone as follows. 
\begin{theorem}\label{thm:minimality-of-radial-cone}
		Let $d\ge3$ and $\gamma\in(0,2)$. Then the following hold:
		\begin{itemize}
			\item[(i)] if $d=3,4,5$, then $b_{\text{rad}}$ is minimizing 
            if and only if $\gamma\ge \gamma_{\Delta}(d)$;\vspace{0.1cm} 
			\item[(ii)] if $d\ge6$, then $b_{\text{rad}}$ is minimizing if and only if $\gamma\ge\frac{2}{d-2}$; in particular, for $\gamma \in [\gamma_\Delta(d),\frac{2}{d-2})$, the radial cone is stable but not minimizing.
		\end{itemize}
        \end{theorem}

The first part of \cref{thm:minimality-of-radial-cone} implies that 
$$
d^*(\gamma) = 3 \qquad \text{for $\gamma \in[\gamma_\Delta (3),1)$},
$$
where $d^*(\gamma)$ is the first dimension in which a minimizing cone for the Alt-Phillips problem exhibits singularities.
By Weiss' formula \cite{weiss-alt-phillips-monotonicity-formula} and a Federer's reduction principle, the dimensional threshold $d^*(\gamma)$  gives a sharp estimate on the Hausdorff dimension of the singular set. 
By the very recent contribution of Fernández-Real \cite{Fernandez-classification}, it is known that $d^*(\gamma)\geq 4$ for $\gamma \in (0,2/3]$ and that, for $d\leq 6$ and $\gamma \in (0,\gamma_{\Delta}(d))$, the set of positive cones $\mathcal{B}_0$ is empty. 

\smallskip
      
We also stress that the second part of \cref{thm:minimality-of-radial-cone} reveals a new phenomenon for singular cones of the Alt-Phillips problem, complementing the results of Savin-Yu \cite{savinyu-altphillips-concentration} in a rather symmetric way.

Indeed, in their work, they show that, for $\gamma$ sufficiently close to $1$, certain axially symmetric cones with contact set of positive density are minimizing in dimension $d\ge4$, in analogy with the cones in the Alt-Caffarelli problem.
Our result goes in the opposite direction: a natural singular cone of the obstacle-problem endpoint $\gamma=1$, namely $b_{\text{rad}}$, remains minimizing for $\gamma$ close to $0$, provided the dimension is sufficiently large. Thus singular minimizing cones associated with one endpoint of $\gamma \in (0,1)$ persist deep into the regime governed by the opposite endpoint, once the dimension is large enough.

\subsection{Bifurcations from the radial cone}\label{sub:bifurcation-intro}
In \cref{thm:parabola-cones}, the failure of integrability of the radial cone at the resonant parameters in \eqref{eq:def-gammakd} is not only a spectral phenomenon, but is also reflected in the local structure of the space of homogeneous solutions near the radial branch. 

Indeed, at $\gamma=\gamma_{k,d}$, a finite-dimensional kernel appears, consisting of spherical harmonics $\mathcal{H}_k(\partial B_1)$ of degree $k$ (see \cref{prop:optimality-log}). This degeneracy is responsible for the failure of integrability in \cref{thm:parabola-cones}, and suggests the possible emergence of non-radial branches of homogeneous solutions. We make this picture precise by performing a parameter-dependent Lyapunov-Schmidt reduction (see \cref{lemma:Lyapunov-Schmidt-bifurcation}) and by classifying the corresponding local bifurcations in suitable symmetry classes. 

\smallskip

Let $\gamma \in(1,2)$. For the sake of readability we denote by $b_{\text{rad},\gamma}$ the radial cone corresponding to the exponent $\gamma$. The following is the main result concerning existence of new branches of singular cones in $\mathcal{B}_0$. 

        \begin{theorem}\label{thm:bifurcation-radial}
			Let $d\ge 2$, we define $$
            k_d:=+\infty \text{ if }d\le7,\quad k_8:=11,\quad k_9:=7,\quad k_{10}=k_{11}=k_{12}:=5 \quad\text{and}\quad k_d:=3 \text{ if }d\ge 13.
            $$
           For $k\in\N_{\geq 3}$, local bifurcations from the radial branch $b_{\text{rad},\gamma}$ at  $\gamma_{k,d}$ are described as follows.\begin{itemize}
				\item[(i)] If $d=2$, a pitchfork bifurcation occurs, and only for $\gamma>\gamma_{k,d}$.\vspace{0.1cm} 
				
				\item[(ii)] If $d\ge3$ and $k$ is odd, then:
                \begin{itemize}[leftmargin=*]
                    \item in the sectorial symmetry class, a pitchfork bifurcation occurs, and only for $\gamma>\gamma_{k,d}$;
                    \item  in the zonal symmetry class, a pitchfork bifurcation occurs if $k<k_d$, and only for $\gamma>\gamma_{k,d}$.
                \end{itemize}\vspace{0.1cm}
				
				\item[(iii)] If $d\ge3$ and $k$ is even, 
                 then a transcritical bifurcation occurs in the zonal symmetry class, whereas no nontrivial branch has a sectorial spherical harmonic as tangent direction.
			\end{itemize}
		\end{theorem}
It is worth mentioning that, in dimension $d=2$, \cref{thm:bifurcation-radial} can be sharpened by excluding the occurrence of secondary bifurcations. Indeed, by \cite[Theorem~5.1]{bonorinoetall}, the number of nontrivial singular $\beta$-homogeneous solutions in $\mathcal B$, counted up to rotations, is
$\lfloor \sqrt{2\beta}\rfloor$, where $\lfloor x\rfloor$ denotes the greatest integer less than $x$.
The resonant values $\gamma_{k,2}$ coincide precisely with the threshold parameters at which this integer increases. Hence every new branch of homogeneous solutions appears directly from the radial branch, and no further bifurcations occur away from these resonant values. Indeed, the absence of secondary bifurcation is consistent with the integrability through rotations of the other positive cones (see \cref{prop:dim2-integrability}).

\smallskip

For $d\geq 3$ and $k$ odd, the picture appears to be richer. Numerical computations suggest that, in the zonal symmetry class, a pitchfork bifurcation also occurs for $k\geq k_d$, but with the nontrivial branches lying on the opposite side $\gamma<\gamma_{k,d}$. We have verified this rigorously with computer algebra (we used Mathematica) for the pairs
$$
(d,k)\quad \text{where}\quad 8\leq d \leq 10^3,\,\, k_d\leq k\leq 10^2+1.
$$
At present, we do not have a proof of the corresponding statement for every $d\geq 8$ and $k\geq k_d$.

\smallskip

We stress that the new singular cones arising in \cref{thm:bifurcation-radial}, with their cylindrical extensions, satisfy \cref{condition:assumption-uniqueness}, i.e., they satisfy the translational hypothesis (see \cref{rem:translational-bif}). In particular, in these cases the blow-up is unique by \cref{t:uniqueness}.

\subsection{Strategy of the proof: the three epiperimetric inequalities}\label{sub:proofs}
The uniqueness of blow-ups in \cref{t:uniqueness} is a direct consequence of the following logarithmic epiperimetric inequality for the Weiss' energy $$W(u):=\int_{B_1}\Big(|\nabla u|^2+u^\gamma\ind_{\{u>0\}}\Big)\,dx-\beta \int_{\partial B_1}u^2\,d\HH^{d-1}.$$ 
   
    \begin{theorem}\label{thm:epiperimetric-combined}
				Let $b$ be a cone satisfying \cref{condition:assumption-uniqueness}. Then there are constants $\eps>0$, $\delta>0$ and $\sigma\in[0,1)$ depending only on $d$, $\gamma$ and $b$, such that the following holds.	
                
              For every non-negative trace $c\in H^1(\partial B_1)$ such that \bea\label{eq:closness-assumption}\|z-b\|_{H^1(B_1)}\le \delta,\quad \|c-b\|_{L^\infty(\partial B_1)}\le \delta\quad\text{and}\quad |W(z)-W(b)|\le \delta,\eea there is a non-negative function $h\in H^1(B_1)$, with $h=c$ on $\partial B_1$, such that \be\label{eq:general-epi}W(h)-W(b)\le (1-\eps|W(z)-W(b)|^\sigma) (W(z)-W(b)),\ee where $z$ is the $\beta$-homogeneous extension of $c$ in $B_1$.
				
                Moreover, if either $b=b_{\text{one}}$ or $b\in\mathcal{B}_\ell$, for $\ell=0,\ldots,d-2$, is sub-integrable, then we can take $\sigma=0$.
			\end{theorem}
        
Epiperimetric inequalities have been widely used as a powerful tool to establish regularity results in both minimal surface theory and free boundary problems. They can be divided into two classes, according to the strategy used in their proof.

The first class consists of the \emph{epiperimetric inequalities by contradiction} \cite{taylor1,taylor2,Weiss-improvement,Weiss-alt-phillips-epiperimetric,gps16,fs16}, which are typically based on linearization techniques.

In the second class, we find the \emph{epiperimetric inequalities by construction}, which are obtained through the explicit construction of the competitor \cite{reifenberg,white,sv-epiperimetric2d,csv18,csv20,esv,cv24}. Typically, this is done either by decomposing the given trace into Fourier modes, or by following a gradient flow on the sphere. 

\medskip

\cref{thm:epiperimetric-combined} is a combination of three different epiperimetric inequalities, each adapted to the geometry of the limiting cones in \cref{condition:assumption-uniqueness}.

The case of \emph{positive cones} $b\in\mathcal{B}_0$ is technically involved and is based on a constructive approach for logarithmic epiperimetric inequalities. Notice that in this class, the set of singularities is isolated. We refer to \cite{Simon,esv,esv2,pariseepiperimetric} for other examples of uniqueness results in the context of isolated singularities.

The genuinely new difficulties arise for the one-dimensional cone and for intermediate cylindrical cones, where the cylindrical structure of the singular set makes the analysis more involved.  

For the \emph{one-dimensional cone} $b_{\text{one}}$, we use a very delicate contradiction argument, in the spirit of Weiss \cite{Weiss-improvement,Weiss-alt-phillips-epiperimetric}, which yields a classical epiperimetric inequality (i.e., \eqref{eq:general-epi} with $\sigma=0$) despite the failure of integrability. This result came as a complete surprise to us, since usually the failure of integrability gives a sharp logarithmic convergence.

For \emph{intermediate cylindrical extensions} $\mathcal{B}_\ell$, with $\ell=1,\dots,d-2$, 
we develop a new combined approach for epiperimetric inequalities based on both the constructive and the contradiction techniques.
More precisely, the contradiction argument removes nonlinear and admissibility errors, while the constructive part supplies the logarithmic improvement through a finite-dimensional reduction.

\medskip

The previous discussion summarizes the main ideas behind the three epiperimetric inequalities. In what follows, we refine the analysis by highlighting the main novelties and challenges.

\subsubsection{The first epiperimetric inequality}
For positive cones in $\mathcal{B}_0$, we use a constructive argument based on a Lyapunov-Schmidt reduction and a gradient flow on the finite-dimensional kernel. The main idea is to decompose the trace $c$ with respect to the eigenfunctions of the spherical linearized operator $L_b:=-\Delta_\theta -\lambda(\beta)+\frac\gamma2(\gamma-1)b^{\gamma-2},$ and to construct the competitor $h$ through a gradient flow on the sphere.

The main difficulty of this strategy is the presence of the kernel $K:=\ker(L_b)$, which may lead to the non-integrability of $b$. 
To overcome this difficulty, in the spirit of \cite{adamsimon}, we apply a Lyapunov-Schmidt reduction for critical points of $\mathcal{G}(\phi):=\mathcal{F}(b+\phi)-\mathcal{F}(b),$ where $\mathcal{F}$ is the spherical Weiss' energy \eqref{def:F} obtained by a slicing lemma. 
This finite-dimensional reduction allows us to decompose $c$ in terms of its component on $K$ and $K^\perp$ and to rewrite the functional $\mathcal{G}$ in terms of a finite-dimensional reduced functional $G$. The logarithmic improvement is then obtained using the Łojasiewicz inequality, in the spirit of Simon \cite{Simon}.

\subsubsection{The second epiperimetric inequality}
For the one-dimensional cone $b_{\text{one}}$, we argue by contradiction. First, in \cref{lemma:ker-1d}, we proceed by characterizing the kernel of the linearized operator $L_{b_{\text{one}}}$, which is given by 
$$
\ker(L_{b_{\text{one}}})=\text{span}\left\{|x_d|^{\beta-1}x_i,\text{sgn}(x_d)|x_d|^{\beta-1}x_i\right\}_{i=1}^{d-1}.
$$
We point out that the modes $\text{sgn}(x_d)|x_d|^{\beta-1}x_i$ are generated by rotations, and thus can be removed by suitably rotating $b_{\text{one}}$. By contrast, the modes $|x_d|^{\beta-1}x_i$ are not induced by rotations, and they are responsible for the non-integrability of the kernel (see \cref{remark:non-int-one-dim-cone}) and require a different argument. Indeed, the Lyapunov-Schmidt reduction does not seem to be the appropriate tool since perturbations of $b_{\text{one}}$ do not preserve the non-negativity condition.

To overcome this difficulty, we generate the modes $|x_d|^{\beta-1}x_i$ by rotating the two half-plane solutions of $b_{\text{one}}$ in an antipodal way. This is done by considering the family of functions 
$$
b_{\xi}:=c_{\text{one}}\left(x\cdot \frac{(\xi,1)}{\sqrt{1+|\xi|^2}}\right)_+^\beta+c_{\text{one}}\left(x\cdot \frac{(\xi,-1)}{\sqrt{1+|\xi|^2}}\right)_+^\beta,\quad\text{for }\xi\in \R^{d-1}.
$$ 
The contradiction argument is set up starting from a linearization around the family $b_\xi$. Nevertheless, the functions $b_\xi$ are not solutions, unless $\xi =0$: unlike the first epiperimetric inequality, this creates an error term in the first variation, which needs to be estimated directly.

Then, the proof is based on the following key ingredients.
First, by exploiting the two-sheeted nature of $b_{\text{one}}$, we show that the error term concentrates on $\{x_d=0\}$. This fact is used to prove weak convergence of the linearized sequence to $0$, since the equation can only be tested away from $\{x_d=0\}$ (see \cref{lemma:ker-1d}).

Secondly, the energy of $b_{\xi}$ decreases in a quantitative way with respect to $b_{\text{one}}$, namely $W(b_{\xi})-W(b_{\text{one}})\approx -\kappa_\beta |\xi|^{2\beta-1},$ where $\kappa_\beta>0$ (see \cref{lemma:negative-b}). 
In this way, we can control the behavior of the first variation near $\{x_d=0\}$, and deduce strong $H^1$ convergence.

\smallskip

We point out that the one-dimensional cone $b_{\text{one}}$ is not included in the definition of sub-integrability in \cref{def:sub-integrability}, since the Lyapunov-Schmidt reduction does not apply in this case.
Roughly speaking, the family $b_\xi$ plays the role of an explicit Lyapunov-Schmidt reduction, since the corresponding first variation vanishes in the limit. 
In this analogy, the reduced functional $G$ is replaced by $W(b_{\xi})-W(b_{\text{one}})$; the fact that this quantity is non-positive suggests that $b_{\text{one}}$ behaves as a sub-integrable cone, which is consistent with the polynomial convergence obtained in \cref{t:rate-convergence}.

\subsubsection{The third epiperimetric inequality}
For intermediate cylindrical cones $\mathcal{B}_\ell$, with $\ell=1,\dots,d-2$, 
we introduce a new hybrid argument, combining a constructive approach with a contradiction argument, two strategies that are usually kept separate.

Let us recall that the cone $b$ can be written as $b(y,z)=B(y)$, with $(y,z)\in\R^{d-\ell}\times \R^\ell$ and $B>0$ in $\R^{d-\ell}\setminus\{0\}$. Set $K_b:=\ker(L_b)$. One of the key ingredients in the proof is the following oblique decomposition, which encodes the cylindrical nature of $b$
$$
H^1(\partial B_1) = K_b^z \oplus K_B \oplus N_B \oplus \mathcal{O}_b,
$$
where $K_b= K_b^z\oplus K_B$ (see \cref{prop:kernel-cylindric}). 
The four subspaces above are defined as follows.
\begin{itemize}
    \item $K_b^z$ comes from $(\beta-1)$-homogeneous Jacobi fields associated with $B$, i.e., its elements can be written as $z_j\psi(y)$, for $j=1,\ldots,\ell$, where $\psi(y)$ is a $(\beta-1)$-homogeneous Jacobi field associated with $B$ (see \cref{prop:kernel-cylindric}). Notice that \cref{condition:assumption-uniqueness} ensures that $K_b^z$ is generated by mixed rotations.
    \item $K_B:=\ker(L_{B})$ is the kernel of the linearized operator of the base cone $B$, i.e., the elements in $K_B$ are the $\beta$-homogeneous Jacobi fields associated with $B$.
    \item $N_B$ is the orthogonal complement of $K_B$ in $X_B$, i.e., $X_B=K_B\oplus N_B$, where $X_B$ denotes the space of $\beta$-homogeneous functions depending only on the $y$-variable. This decomposition allows us to construct a \emph{partial} Lyapunov-Schmidt reduction 
    (see \cref{lemma:Lyapunov-Schmidt-partial}), namely involving only the $y$-variables.
    \item $\mathcal{O}_b$ is a subset of $\ker\Pi$, where $\Pi$ is a suitable projection (see \cref{def:map-Pi}) such that, whenever the first variation of $\mathcal{G}$ is computed at an element of $X_B$ in the direction $\psi$, it depends only on the projected component $\Pi\psi$ (see \cref{lemma:variation-Pi}). In particular, for every element of $X_B$, the first variation of $\mathcal G$ vanishes along all directions in $\mathcal{O}_b$. 
\end{itemize}
The \emph{partial} Lyapunov-Schmidt reduction is a crucial point of the argument. Indeed, it preserves the analyticity of the Lyapunov-Schmidt map $Y\colon K_B \to N_B$ and it keeps the relevant perturbations compatible with the vanishing profile of the cone, which is of order $|y|^\beta$. We stress that the analyticity is crucial to apply the Łojasiewicz inequality.

At this point, the assumption in \cref{condition:assumption-uniqueness} becomes decisive. Indeed, in \cref{corollary:corollary1} we show that the base cone $B$ is translational if and only if $K_b^z$ is generated by mixed rotations. Hence, the corresponding modes can be removed by a suitable rotation of the cone, and the remaining finite-dimensional reduction can be carried out only in the $y$-variables. If the $y$- and $z$-variables were treated simultaneously, neither of these properties would be available in a useful form. In fact, a competitor depending also on the $z$-variables would not be expected to lead to analyticity of the Lyapunov-Schmidt map and to satisfy the required non-negativity condition (see also (ii) in \cref{subsub:role-condition1}).

Finally, we argue by contradiction, linearizing around the projections of the traces on the space $\{b+\phi+Y(\phi):\phi\in K_B\}$. By the previous discussion, we have the key identity $$\delta \mathcal{G}(\phi+Y(\phi))[\psi]=0\quad\text{for every }\psi\in \mathcal{M}_b:=N_B\oplus \mathcal{O}_b,$$ where $\mathcal{M}_b$ is a complement of the kernel $K_b=K_B \oplus K_b^z$.  

By the finite-dimensional reduction, we can rewrite the restriction of $\mathcal{G}$ to $X_B$, which is analytic, in terms of a finite-dimensional reduced functional $G$, and then we can construct a competitor for the projection by the associated gradient flow. 
Finally, by applying a Łojasiewicz inequality, we deduce a logarithmic improvement, as in the proof of the first epiperimetric. 

Unlike in the proof of the second epiperimetric inequality, the main difficulty in the present contradiction argument is to prove that the linearized sequence converges to an element of the kernel $K_b$. We overcome this point by introducing, in the Weiss' contradiction argument, a different competitor built using the one coming from the gradient flow together with two cut-off functions (see Step 2 in the proof of \cref{t:epi-cylindrical}). 
\subsubsection{Extra comments}\label{subsub:role-condition1} In what follows, we explain the role of the translational assumption in \cref{condition:assumption-uniqueness}, and we highlight the advantages of the hybrid approach in the proof of the third epiperimetric inequality.

\medskip

\emph{(i) The role of translationality in \cref{condition:assumption-uniqueness}.}
The translational hypothesis in \cref{condition:assumption-uniqueness} is natural in the analysis of cylindrical cones. In particular, analogous assumptions appear as Condition $\ddagger$(b) in Simon's work \cite{SimonCylindric} (see also \cite{gabor,FTW-uniqueness}), and as condition (2) in the strong integrability assumption of \cite[Definition 2.5]{EngResZih}.

Let us point out, however, that our assumption is weaker than the one used by Simon. 
More precisely, in \cite{SimonCylindric}, it is assumed that every Jacobi field of degrees $1$ and $0$ of the base cone is generated, respectively, by rotations and translations. Together with its Condition $\ddagger$(c), these ensure that the $1$-homogeneous Jacobi fields of the cylindrical extension are generated by rotations. 

In contrast, we do not need to assume the rotational integrability of the $\beta$-homogeneous Jacobi fields, which corresponds to the $1$-homogeneous Jacobi field in the minimal surface setting. Indeed, in our framework, it is enough to rule out $(\beta-1)$-homogeneous Jacobi fields not generated by translations. 
The remaining, non-rotational part of the kernel is then treated through a Lyapunov-Schmidt reduction for the base cone, and the energy is improved by means of a Łojasiewicz inequality.

It would be interesting to understand whether a similar idea could be used to remove the assumption of rotational integrability in the minimal surface framework.

\smallskip

We are not aware of examples of singular minimizing cones for which such a translational condition fails, except the one-dimensional cone $b_{\text{one}}$, which, however, is treated separately. In all known examples of singular minimizing cones in $\mathcal{B}_\ell$, for $\ell=0,\ldots,d-2$, the translational condition can be verified directly (see \cref{prop:radial-cone-is-int-traslations}, \cref{corollary:kernel-cylindrical} and \cref{rem:translational-bif}). 
Moreover, the condition can be established in certain regimes of the parameters $d$ and $\gamma$ (see \cref{prop:trans0}).

Nevertheless, whether the translational hypothesis holds for every cone in $\mathcal{B}_\ell$ remains an open question, and by \cref{corollary:corollary1}, it suffices to prove it for cones in $\mathcal{B}_0$. If this were true, \cref{t:uniqueness} would yield uniqueness of blow-ups for all minimizing cones, for every $\gamma\in(1,2)$ and in every dimension $d\ge2$.

\medskip

\emph{(ii) The advantage of the hybrid approach.}
In our setting, the hybrid method developed for the third epiperimetric inequality seems to combine the advantages of both the constructive and contradiction approaches, as we now explain.

First, the constructive part of the argument retains some advantages of the constructive approach: it does not require integrability of the cone and yields logarithmic improvements. 

This feature allows us to overcome some of the limitations of purely contradiction-based epiperimetric inequalities. Indeed, to date, contradiction arguments at singular points have been available only under additional assumptions, such as integrability of the limiting cone. When applicable, they yield a classical epiperimetric inequality (i.e., \eqref{eq:general-epi} with $\sigma=0$), hence a polynomial rate of convergence.

Such a polynomial rate, however, is not expected to be true in general. In many important situations the sharp rate is indeed logarithmic. This phenomenon appears, for instance  in geometric variational problems \cite{nagura,adamsimon,CH-Thom,CT3}, obstacle-type problems \cite{csv18,fs19,csv20}, and the Alt-Caffarelli functional \cite{esv}. As we show in this paper, the same phenomenon occurs for the Alt-Phillips problem.

Second, the contradiction part of the argument removes nonlinear error terms that are not visible at the linearized level and is essential to ensure the admissibility of the competitor. 

The error terms arise from the fact that the trace $c$ does not necessarily have the same geometric structure as the blow-up (e.g., the same zero set or the same order of vanishing). Moreover, they also appear in the expansion of the semilinear term around $b$, as well as in higher-order terms coming from the trace, which in general does not satisfy the linearized equation. In our contradiction part, these errors vanish after passing to the limit, whereas, in a fully constructive approach, they would have to be estimated directly.

Regarding the question of admissibility (non-negativity), in our proof it is sufficient to establish only the admissibility of the projection of the competitor onto the space
$\{b+\phi+Y(\phi):\phi \in K_B\}$, whereas a constructive approach would require the admissibility of the whole object.
A priori, by writing the cylindrical expansion of $b$ in $(y,z)$,  
we would have to control both the perturbations in the $y$- and $z$-variables. Our contradiction part avoids this obstruction, since the projection onto the space $\{b+\phi+Y(\phi):\phi \in K_B\}\subset X_B$ depends only on the $y$-variables.

\subsection{Structure of the paper}
In \cref{section:preliminaries}, we collect some preliminaries on the Alt-Phillips problem. Moreover, we present the linearized operators and the notions of integrability, sub-integrability and translationality. 
\cref{section-positive} is devoted to the proof of the first epiperimetric inequality for positive cones $b \in \mathcal{B}_0$. 
Moreover, we show the optimality of the logarithmic rate of convergence in the case of not sub-integrable positive cones. 
In \cref{section:integrable-bifurcation}, we deepen the analysis of the radial cone $b_{\text{rad}}$ through an explicit integrability and bifurcation analysis. This section contains the proof of \cref{thm:bifurcation-radial}.
In \cref{section:minimality-radial-cone} we introduce the one-dimensional calibrations and we characterize the minimality of the radial cone in terms of $\gamma$ and $d$, proving \cref{thm:minimality-of-radial-cone}.
\cref{section-1d} is devoted to the exceptional case of the one-dimensional cone $b_{\text{one}}$. We prove the second epiperimetric inequality and we show the non-integrability and non-translationality of the one-dimensional cone.
In \cref{section-translational} we start the analysis of intermediate cylindrical extensions, and we study the interplay between the notions of integrability and translationality in the class $\mathcal{B}_\ell$ for $\ell=0,\dots,d-2$. Then, in \cref{section:intermediate-epiperimetric}, we show the third epiperimetric inequality for translational cylindrical extensions.
Finally, in \cref{section:final} we prove \cref{t:uniqueness}, \cref{t:ap-low-dimension}, \cref{t:rate-convergence}, \cref{thm:parabola-cones} and \cref{thm:epiperimetric-combined}.

\subsection{Acknowledgements}
We are grateful to Bozhidar Velichkov for some useful discussions. We also thank Luca Spolaor for valuable insights on the Adams-Simon paper \cite{adamsimon}. M.C. acknowledges support from the European Research Council (ERC) via the project ERC FiRM - {\em Fine structure and regularity of stationary and moving free boundaries} (grant agreement No. 101230705). M.C. and G.T. are members of INdAM-GNAMPA. M.C. and G.T. acknowledge the support from the GNAMPA project E53C25002010001 - {\em Struttura fine e regolarità in problemi variazionali non-lineari}.

		\section{Preliminaries}\label{section:preliminaries}
		\subsection{Notations and conventions}
        \begin{itemize}
              \item With a slight abuse of notation, we identify $\beta$-homogeneous functions with their trace on the sphere $\partial B_1$. For example, for a blow-up $b$, we write $b(r,\theta)=r^\beta b(\theta)$. We use the same convention for the kernel, identifying its elements with their homogeneous extensions when needed.
            \item Let $V=V_1\oplus V_2$ be a vector space. We denote by $P_{V_i}$ the projection onto $V_i$.
            \item Given $k\in \N_{\geq 0}$ and $\ell \in \N_{\geq 2}$, we denote by $\mathcal{H}_k(\mathbb{S}^{\ell-1})$ the set of spherical harmonics of degree $k$ on $\mathbb{S}^{\ell-1}$.
            \item For $n\in\N_{\ge1}$ and $\alpha\in\R$, we denote by $\lambda_{n}(\alpha):=\alpha(\alpha+n-2)$. When $n=d$, we drop the dependence on $n$, namely $\lambda(\alpha):=\alpha(\alpha+d-2)$.
            \item Given a rotation $R \in SO(d)$ and a function $f$, we denote by $Rf(x):=f(R^{-1}x)$. 
        \end{itemize}
         \subsection{Minimizers and solutions of the Alt-Phillips problem}
        In this subsection, we recall some well-known results about the Alt-Phillips problem. For a ball $B$ in $\R^d$, we consider the functional $$\mathcal{J}_\gamma(u,B):=\int_{B}\Big(|\nabla u|^2+u^\gamma\ind_{\{u>0\}}\Big)\,dx.$$
        We use the following definitions of minimizers and solutions of the Alt-Phillips problem.
        \begin{definition}\label{def:minimo}
            Let $u\in H^1(B)$, we say that $u$ is a minimizer of the Alt-Phillips problem if \be\label{eq:comparison}\mathcal{J}_\gamma(u,B)\le \mathcal{J}_\gamma(v,B)\quad\text{for every }v\in H^1(B),\ u=v\text{ on }\partial B.\ee
            We say that $u\in H^1_{\text{loc}}(\R^d)$ is a global minimizer of the Alt-Phillips problem if $u$ is a minimizer of $\mathcal{J}_\gamma(\cdot,B)$, for every ball $B\subset \R^d$.
            \end{definition}
            The following definition, introduced in \cite[Definition 2.1]{savinyu-altphillips-stable}, is used to characterize minimizers of the Alt-Phillips problem in terms of the existence of upper (resp. lower) foliations.
            \begin{definition}\label{def:minimo-alto-basso}
            Let $u\in H^1(B)$, we say that $u$ is a one-sided minimizer from below (resp. from above) if the comparison \eqref{eq:comparison} holds under the additional assumption $v\le u$ (resp. $v\ge u$). 
        \end{definition}
        \begin{definition}\label{def:weak-sol}
             Let $u\in H^1(B)$, we say that $u$ is a (weak) solution of the Alt-Phillips problem in $B$ if $$\Delta u=\frac\gamma2u^{\gamma-1}\ind_{\{u>0\}}\text{ in }B$$
             is satisfied in the sense of distributions. We say that $u\in H^1_{\text{loc}}(\R^d)$ is a global solution of the Alt-Phillips problem if $u$ is a weak solution in every ball $B$ in $\R^d$.
        \end{definition}
        We also have the following regularity result from \cite{Alt-Phillips}.
        \begin{proposition}\label{prop:regularity-of-solutions-alt-phillips}
            Let $u\in H^1(B_1)$ be a minimizer of the Alt-Phillips problem with $0\in\partial\Omega_u$. Then $u\in C^\beta_{\text{loc}}(B_1)$ and $$\|u\|_{C^\beta(B_{1/2})}\le C,$$ for some constant $C$ depending only on $d$ and $\gamma$. 
        \end{proposition}
        We also have the following characterization of blow-ups for $\gamma\in(0,1)$ in two dimensions \cite{Alt-Phillips}.
        \begin{proposition}\label{prop:no-singular-points-dim-2}
            Let $d=2$, $\gamma\in(0,1)$ and $b\in H^1_{\text{loc}}(\R^2)$ be a non-trivial $\beta$-homogeneous minimizer of the Alt-Phillips problem. Then $b=c_\beta (x\cdot\nu)_+^\beta$ for some $\nu\in \mathbb{S}^{d-1}$.
        \end{proposition}
        
        We notice that, for every $\gamma\in(0,2)$ minimizers are weak solutions \cite{Alt-Phillips}.
        On the other hand, in the case $\gamma\in[1,2)$, since the functional $\mathcal{J}_\gamma$ is convex, these two notions are equivalent.
        \begin{proposition}\label{prop:convex-functional}
            Let $\gamma\in[1,2)$ and $u\in H^1(B)$ be a weak solution of the Alt-Phillips problem. Then $u$ is also a minimizer of $\mathcal{J}_\gamma(\cdot,B)$.
        \end{proposition}
        Moreover, the global solutions are convex \cite{Alt-Phillips,bonorino}.
        \begin{proposition}\label{prop:convexity-of-blowups}
            Let $\gamma\in[1,2)$ and $u\in H^1_{\text{loc}}(\R^d)$ be a global solution of the Alt-Phillips problem, such that $0\in\partial\Omega_u$. Then $D^2 u\ge0$ in $\R^d$.
        \end{proposition}
        As a consequence, we get the following characterization of blow-ups \cite{bonorino}.
        \begin{proposition}\label{prop:classification-of-blowups-gamma-maggiore-1}
            Let $\gamma\in[1,2)$ and $b\in H^1_{\text{loc}}(\R^d)$ be a non-trivial $\beta$-homogeneous solution of the Alt-Phillips problem. Then, either $b=c_\gamma (x\cdot\nu)_+^\beta$ for some $\nu\in \mathbb{S}^{d-1}$, or $b\in\mathcal{B}_\ell$ for some $\ell=0,\ldots,d-1$.
        \end{proposition}
        \subsection{Weiss' energy and blow-ups}
        We define the Weiss' energy 
        $$
        W(u)=W_0(u) +\int_{B_1}u^\gamma\ind_{\{u>0\}}\,dx\quad\text{where}\quad W_0(u):=\int_{B_1}|\nabla u|^2 \,dx -\beta \int_{\partial B_1}u^2\,d\HH^{d-1}.
        $$
        Let $u\in H^1(B_1)$ be a minimizer of the Alt-Phillips problem, and suppose that $0\in\partial\Omega_u$ is a free boundary point. We consider the rescaled function $u_r(x):=r^{-\beta}u(rx)$, then we have the Weiss' monotonicity formula  \cite{weiss-alt-phillips-monotonicity-formula} \be\label{eq:weiss-formula}\frac{d}{dr}W(u_r)=\frac{d+2\beta-2}{r}(W(z_r)-W(u_r))+\frac{1}{r}\int_{\partial B_1}(\nabla u_r\cdot x-\beta u_r)^2\,d\HH^{d-1},\ee where $z_r$ is the $\beta$-homogeneous extension of the trace $u_r|_{\partial B_1}$.
        As a consequence, along a subsequence $r_k\to0^+$, the rescaling $u_r$ converges locally uniformly to a function $b\in H^1_{\text{loc}}(\R^d)$, which is $\beta$-homogeneous and a local minimizer of the Alt-Phillips problem in $\R^d$.
        
\subsection{Linearized operator}
        For a blow-up $b\in\mathcal{B}$ satisfying \cref{condition:assumption-uniqueness}, we define the spherical linearized operator \be\label{eq:linearized-operator}L_b:=-\Delta_\theta -\lambda(\beta)+\frac\gamma2(\gamma-1)b^{\gamma-2}\ee
        where $\lambda(\beta) :=\beta(\beta+d-2)$. We define the kernel $K_b:=\ker (L_b)$ as \be\label{eq:general-kernel}\ker (L_b):=\left\{\phi\in H^1(\partial B_1): \ L_b\phi=0\text{ on }\partial B_1\cap\{b>0\} \text{ and }\phi=0 \text{ on }\partial B_1\cap\{b=0\}\right\},\ee where $L_b\phi=0$ on $\partial B_1\cap\{b>0\}$ is understood in the $H^1$-weak sense, while $\phi=0$ on $\partial B_1\cap\{b=0\}$ is understood in a trace sense.
        
        We also observe that, if $b\in\mathcal{B}_\ell$ for some $\ell=0,\ldots,d-1$, then \be\label{eq:general-kernel-only-positive}K_b=\{\phi\in H^1(\partial B_1): \ L_b\phi=0\text{ on }\partial B_1\cap\{b>0\}\}.\ee
        Indeed, for $\ell=0,\ldots,d-2$, the set $\partial B_1\cap\{b=0\}$ has zero $H^1$-capacity on the sphere, while for $\ell=d-1$ it follows by \cref{lemma:ker-1d}.

\subsection{Some spherical energies}\label{subsection:spherical-energies} Let $\phi\in H^1(\partial B_1)$ be non-negative, we define the spherical energy \be\label{def:F}\mathcal{F}(\phi):=\int_{\partial B_1}
			\left(|\nabla_\theta\phi|^2-\lambda(\beta)\phi^{2}+\phi^\gamma\right)\,d\HH^{d-1}.\ee
This is precisely the spherical Weiss' energy coming from the following slicing lemma.
\begin{lemma}[Slicing lemma]\label{lemma:slicing}
			Let $\phi \in H^1( B_1)$ be non-negative, and set $\phi_r(\theta):=\phi(r,\theta)$. Then 
			\bea {W}(r^\beta\phi_{r})&=\int_{0}^1 r^{d+2\beta-3}\mathcal{F}(\phi_{r})\,dr+\int_0^1\int_{\partial B_1} r^{d+2\beta-1}(\partial_r\phi_r)^2 \,dr,\eea where $\mathcal{F}$ is defined in \eqref{def:F}.
		\end{lemma}
        \begin{proof}
            The proof is a straightforward modification of \cite[Lemma 12.10]{Velichkov:RegularityOnePhaseFreeBd}.
        \end{proof}
Let $b$ be a cone satisfying \cref{condition:assumption-uniqueness}, we consider the corresponding spherical energy $\mathcal{G}$ defined for every $\phi\in H^1(\partial B_1)$ such that $b+\phi\ge0$ \be\label{def:G}\mathcal{G}(\phi):=\mathcal{F}(b+\phi)-\mathcal{F}(b)=\int_{\partial B_1}\Big(|\nabla_\theta \phi|^2-\lambda(\beta)\phi^2+(b+\phi)^\gamma-b^\gamma-\gamma b^{\gamma-1}\phi\Big)\,d\HH^{d-1},\ee where the second identity follows by integrating by parts and using the equation of $b$. 
		Then, the first and the second variation of $\mathcal{G}$ are the following \bea  \frac12\delta\mathcal{G}(\phi)[\psi]&:=
		\int_{\partial B_1}\Big(\nabla_\theta\phi\cdot\nabla_\theta \psi-\lambda(\beta)\phi\psi+\frac\gamma2\big((b+\phi)^{\gamma-1}- b^{\gamma-1}\big)\psi\Big)\,d\HH^{d-1}\eea
		and 
		\bea \frac12\delta^2 \mathcal{G}(\phi)[\psi,\psi]&:=
		\int_{\partial B_1}\Big(|\nabla_\theta \psi|^2-\lambda(\beta)\psi^2+\frac\gamma2(\gamma-1)(b+\phi)^{\gamma-2}\psi^2\Big)\,d\HH^{d-1}.\eea
		In particular, the kernel of $L_b$ in \eqref{eq:general-kernel} is also given by $K_b=\ker(\frac12\delta^2\mathcal{G}(0))$.
        For simplicity, we also set \be\label{def:Q}\mathcal{Q}(\psi):=\frac12\delta^2\mathcal{G}(0)[\psi,\psi]=\int_{\partial B_1}\left(|\nabla_\theta \psi|^2-\lambda(\beta)\psi^2+\frac\gamma2(\gamma-1) b^{\gamma-2}\psi^2\right)\,d\HH^{d-1}.\ee

\subsection{Integrability and translationality}\label{subsection-integrability}
        In what follows, we first introduce the notions of integrability and sub-integrability for singular cones satisfying \cref{condition:assumption-uniqueness}.
        \begin{definition}[Integrability]\label{def:integrability}
		    Let $b\in\mathcal{B}$ satisfy \cref{condition:assumption-uniqueness}, we say that $b$ is integrable if for every $\phi\in K_b$, there exists a family $\{u_t\}_{t\in(-\eps,\eps)}$ of $\beta$-homogeneous weak solutions of the Alt-Phillips problem, such that $u_0=b$ and $\partial_t u_t\big|_{t=0}=\phi$. 
        \end{definition}   
        By \cref{prop:kernel-cylindric} (see also \cref{corollary:corollary3}), we have the following remark.
        \begin{remark}
            If $b\in\mathcal{B}_0$ is a positive cone, then $b$ is integrable if and only if for every $\phi\in K_b$ there exists a family of solutions $\{\psi_t\}_{t\in(-\eps,\eps)}$ of $\delta\mathcal{G}(\psi_t)=0$ such that $\psi_0=0$ and $\partial_t \psi_t\big|_{t=0}=\phi$. 
            
            Let $b\in\mathcal{B}_{\ell}$ for $\ell=1,\ldots d-2$, we write $b(y,z)=B(y)$ for $(y,z)\in\R^{d-\ell}\times \R^\ell$. Then $b$ is integrable if and only if $B$ is integrable.
        \end{remark}
       
       In the following, we introduce the notion of sub-integrability from our companion paper \cite{CT3}.
       \begin{definition}[Sub-integrability]\label{def:sub-integrability}
			Let $b\in\mathcal{B}_{0}$ and $G$ be the reduced functional associated to $\mathcal{G}$ at $b$ (see \cref{lemma:Lyapunov-Schmidt}). We say that $b$ is sub-integrable if $G\le0$ in a neighborhood of $0$.
			
			Let $b\in\mathcal{B}_{\ell}$ for $\ell=1,\ldots d-2$ satisfying \cref{condition:assumption-uniqueness}. We say that $b$ is sub-integrable if the base cone $B$ is sub-integrable.
		\end{definition}
We point out that if $b\in\mathcal{B}_0$ and $G$ is the reduced functional associated to $\mathcal{G}$ at $b$, by \cite[Lemma 1]{adamsimon}, the integrability condition in \cref{def:integrability} is equivalent to requiring that $G\equiv0$ in a neighborhood of $0$. This is the reason why we call the condition $G\le0$ sub-integrability. 
         
            We use the following definition of integrability through rotations.
			\begin{definition}[Integrability through rotations]\label{def:integrability-through-rotations}
				We say that $b\in\mathcal{B}$ satisfying \cref{condition:assumption-uniqueness} is integrable through rotations if 
				every function in $K_b$ can be written as a rotational derivative of $b$, namely \be\label{eq:integrability-rotations}K_b=\text{span}\{\nabla_{\theta} b\cdot (A\theta):\ A\in \text{so}(d)\}=\text{span}\{x_j\partial_{x_i}b-x_i\partial_{x_j}b: \ i,j=1,\ldots d\},\ee where $\text{so}(d)$ denotes the Lie algebra of skew-symmetric matrices generating infinitesimal rotations. Equivalently, every $\beta$-homogeneous solution of the linearized operator $-\Delta+\frac\gamma2(\gamma-1)b^{\gamma-2}$ is a rotational derivative of $b$. 
			\end{definition}
			We also define the eigenspace $E_b^{\mu}$ of the $\mu$-homogeneous solutions of the linearized operator as follows. 
            
            \begin{definition}[The eigenspace $E_b^{\mu}$]\label{def:space-Ebmu} 
                Let $b\in\mathcal{B}_\ell$ for $\ell=0,\ldots,d-1$, and $\mu\in\R$. We define $E_b^{\mu}\subset H^1(\partial B_1)$ as the eigenspace of the linearized operator $L_b$ corresponding to the eigenvalue $\lambda(\mu)-\lambda(\beta)$, namely
                $$E_b^\mu:=\ker(L_b-(\lambda(\mu)-\lambda(\beta))),$$ where the definition of the kernel is given as in \eqref{eq:general-kernel}. In particular $E_b^\beta=K_b$.
                
                Equivalently, $E_b^{\mu}$ consists of the traces whose $\mu$-homogeneous extensions are solutions of the linearized equation $-\Delta+\frac\gamma2(\gamma-1)b^{\gamma-2}$. 
            \end{definition}

            As we will see, the space $E_b^{\beta-1}$ will play a key role, which motivates the following definition. 
        
            \begin{definition}[Translationality]\label{def:beta-1translational}
				We say that $b\in\mathcal{B}_\ell$, for some $\ell=0,\ldots,d-1$, is translational if every function in $E^{\beta-1}_b$ can be written as a translation derivative of $b$, namely
				$$E^{\beta-1}_b=\text{span}\{\partial_{x_1}b,\ldots,\partial_{x_{d}}b\}.$$
				
                Equivalently, $b$ is translational if every $(\beta-1)$-homogeneous solution of the linearized equation $-\Delta+\frac\gamma2(\gamma-1)b^{\gamma-2}$ is a partial derivative of $b$.
			\end{definition}

            \begin{remark}
                If $b\in\mathcal{B}_\ell$, for $\ell=0,\ldots,d-2$, then the set $\{b=0\}$ has zero $H^1$-capacity in $\R^d$. Then, in \cref{def:integrability-through-rotations}, \cref{def:space-Ebmu} and \cref{def:beta-1translational} it is enough to require the corresponding condition only in $\{b>0\}$, as already observed for $K_b$ in \eqref{eq:general-kernel-only-positive}.
            \end{remark}

	\section{Epiperimetric inequality for positive cones}\label{section-positive}
		In this section we prove a logarithmic epiperimetric inequality around cones in $\mathcal{B}_0$, namely strictly positive cones in $\R^d\setminus \{0\}$, including the radial cone $b_{\text{rad}}:=c_{\text{rad}}|x|^\beta$. We also show the optimality of the logarithmic modulus of continuity, when the cone is not sub-integrable. Notice that these results apply for every $\gamma\in(0,2)$. 
		\begin{proposition}[Logarithmic epiperimetric inequality]\label{t:epi-rad}
			Let $\gamma\in(0,2)$, then
            the epiperimetric inequality in \cref{thm:epiperimetric-combined} holds for $b\in\mathcal{B}_0$, under the closeness assumptions
            \be\label{eq:closness-assumption-radial}\|c-b\|_{L^\infty(\partial B_1)}\le \delta\quad\text{and}\quad |W(z)-W(b)|\le \delta.\ee 
			Moreover, if $b$ is sub-integrable (see \cref{def:sub-integrability}), then we can take $\sigma=0$.
		\end{proposition}
\subsection{Linearized operator around positive cones} Let $b\in\mathcal{B}_0$ be a positive cone, we recall from \eqref{eq:linearized-operator} the linearized operator $L_b:=-\Delta_\theta -\lambda(\beta)+\frac\gamma2(\gamma-1)b^{\gamma-2}$ on the sphere $\partial B_1$. Since $b>0$ on $\partial B_1$, the kernel of $L_b$ defined in \eqref{eq:general-kernel} is finite-dimensional and given by
		$$K:=\ker(L_b)=\{\phi\in H^1(\partial B_1): L_b\phi=0\text{ on }\partial B_1\},\quad N:={\rm dim}\, K\in\N,$$ Thus there are $N$ orthonormal eigenfunctions $\Phi_1,\dots, \Phi_{N},$ such that $K=\text{span}\, \{\Phi_1,\dots, \Phi_{N}\}.$
		Moreover, there is a sequence of eigenvalues $\{\lambda_j\}\subset\R\setminus\{0\}$, counted with multiplicity, with $\lambda_{1}\le\ldots\le \lambda_j\le \ldots,$ and satisfying $\lambda_j\to+\infty$ as $j\to+\infty$, and there is a sequence of orthonormal eigenfunctions $\{\varphi_j\}\subset H^1(\partial B_1)$ such that $L_b \varphi_j=\lambda_j\varphi_j$ on $\partial B_1.$ 
		In particular, $\{\Phi_j\}_{j=1}^N \cup \{\varphi_j\}_{j\ge0}$
		is an orthonormal basis of $H^1(\partial B_1)$. 
		
		We also define the eigenspaces corresponding to negative and positive eigenvalues in the following way.
		If $M\in\N$ is such that $\varphi_1,\ldots,\varphi_M$ correspond to negative eigenvalues and $\varphi_{M+1},\ldots,\varphi_{M+j},\ldots$ correspond to positive eigenvalues, we define
		\be\label{eq:E<andE>}K_<:=\text{ span}\, \{\varphi_1,\ldots,\varphi_{M}\}\quad\text{and}\quad K_>:=\text{ span}\, \{\varphi_{M+1},\ldots,\varphi_{M+j},\ldots\}.\ee
        In particular $H^1(\partial B_1)= K \oplus K^\perp$, where $K^\perp = K_< \oplus K_>$.
		
        \subsection{Decomposition of the energy}
		Recalling the spherical energies in \cref{subsection:spherical-energies}, in the next lemma we decompose the spherical Weiss' energy $\mathcal{F}$ in terms of the spherical energies $\mathcal{G}$ and $\mathcal{Q}$. 
		\begin{lemma}\label{lemma:decomposition}
			For every $\ell>0$, there is $\delta=\delta(d,\gamma,b,\ell)>0$ such that the following holds. 
			Let $\psi_1,\psi_2\in H^1(\partial B_1)$ such that $\|\psi_1\|_{L^\infty(\partial B_1)}+\|\psi_2\|_{L^\infty(\partial B_1)}\le \delta.$ We suppose that $\delta \mathcal{G}(\psi_1)[\psi_2]=0$, 
			then the following decomposition holds 
			\bea\label{eq:decomp1} \mathcal{F}(b+\psi_1+\psi_2)-\mathcal{F}(b)=\mathcal{G}(\psi_1)+{\mathcal{Q}}(\psi_2)+R(\psi_2)
			\eea where \bea\label{eq:error} |R(\psi_2)|\le \ell \int_{\partial B_1}\psi_2^2\,d\HH^{d-1}.\eea
		\end{lemma}
		\begin{proof} 
			Fix $\ell>0$ and set $\psi:=\psi_1+\psi_2$,
			denoting by \bea \mathcal{G}_0(\psi):=\int_{\partial B_1}\Big(|\nabla_\theta \psi|^2-\lambda(\beta)\psi^2\Big)\,d\HH^{d-1},\eea we can find $\xi_1\in[0,1]$ depending also on $\psi_1$ and $\psi_2$, such that
            \begin{align*} \mathcal{G}(\psi_1+\psi_2)&=\mathcal{G}_0(\psi_1)+\mathcal{G}_0(\psi_2)+\int_{\partial B_1}\Big(2\nabla_\theta\psi_1\cdot\nabla_\theta\psi_2-2\lambda(\beta)\psi_1\psi_2\Big)\,d\HH^{d-1}\\&\qquad+\int_{\partial B_1}\Big(( b+\psi_1+\psi_2)^{\gamma}-b^\gamma-\gamma b^{\gamma-1} (\psi_1+\psi_2)\Big)\,d\HH^{d-1}\\&=\mathcal{G}_0(\psi_1)+\mathcal{G}_0(\psi_2)+\int_{\partial B_1}\Big(2\nabla_\theta\psi_1\cdot\nabla_\theta\psi_2-2\lambda(\beta)\psi_1\psi_2\Big)\,d\HH^{d-1}\\&\qquad+\int_{\partial B_1}\Big(( b+\psi_1)^{\gamma}+\gamma( b+\psi_1)^{\gamma-1}\psi_2+\frac\gamma2(\gamma-1)( b+\psi_1+\xi_1\psi_2)^{\gamma-2}\psi_2^2\\&\qquad\qquad-b^{\gamma}-\gamma b^{\gamma-1} (\psi_1+\psi_2)\Big)\,d\HH^{d-1}.
			\end{align*} Using the definition of $\mathcal{G}(\psi_1)$, we have
			\bea \mathcal{G}(\psi_1+\psi_2)&=\mathcal{G}(\psi_1)+\mathcal{G}_0(\psi_2)+\int_{\partial B_1}\Big(2\nabla_\theta\psi_1\cdot\nabla_\theta\psi_2-2\lambda(\beta)\psi_1\psi_2\Big)\,d\HH^{d-1}\\&\qquad+\int_{\partial B_1}\Big(\gamma( b+\psi_1)^{\gamma-1}\psi_2+\frac\gamma2(\gamma-1)( b+\psi_1+\xi_1\psi_2)^{\gamma-2}\psi_2^2-\gamma b^{\gamma-1} \psi_2\Big)\,d\HH^{d-1}\\&=\mathcal{G}(\psi_1)+\mathcal{G}_0(\psi_2)+\int_{\partial B_1}\frac\gamma2(\gamma-1)( b+\psi_1+\xi_1\psi_2)^{\gamma-2}\psi_2^2\,d\HH^{d-1},
			\eea 
            where in the second equality, we used that $\delta \mathcal{G}(\psi_1)[\psi_2]=0$.
			By using that $b>0$ on $\partial B_1$, we can choose $\delta$ small enough so that 
			\bea \Big|\int_{\partial B_1}\big( b+\psi_1+\xi_1\psi_2)^{\gamma-2}-b^{\gamma-2}\big)\psi_2^2\,d\HH^{d-1}\Big|\le \ell \int_{\partial B_1}\psi_2^2\,d\HH^{d-1},
			\eea which concludes the proof, by using the definition of $\mathcal{Q}$ in \eqref{def:Q}.
		\end{proof} 
		\subsection{Lyapunov-Schmidt reduction}		
		The following proposition is a standard application of the Lyapunov-Schmidt reduction (see, for instance, \cite[Section 2]{Simon} or \cite[Lemma B.1]{esv}), and it applies to the case in which the linearized operator $L_b$ has a nontrivial kernel $K$. 
       
        \begin{proposition}\label{lemma:Lyapunov-Schmidt}
			Let $b\in\mathcal{B}_0$, and assume that $N:={\rm dim}\, K>0$. Then, there exist a neighborhood $U\subset K$ of $0$ in $C^{1,\alpha}(\partial B_1)$ and an analytic map $$Y:K\cap U \to K^\perp\subset H^1(\partial B_1)$$ 
            such that the following holds.
            \begin{itemize}
            \item $Y(0)=0$, $\delta Y(0)=0$. Moreover,
            \be\label{eq:cond-on-perp}
			P_{K^\perp}\big(\delta\mathcal{G}(\phi+Y(\phi))\big)=0,\quad \mbox{for every }\phi \in K\cap U.
            \ee
    \item Let $\Phi_1,\dots,\Phi_N$ be an orthonormal basis of $K$. Then, there exists $\rho>0$ such that, for every $\mu\in B_\rho\subset\R^N$, the reduced functional $G:B_\rho\to\R$ defined by
\be\label{def:G-finito-dim}
G(\mu):=\mathcal{G}(\Phi_\mu +Y\left(\Phi_\mu\right)),
\quad \mbox{with}\quad
\Phi_\mu:=\sum_{j=1}^N\mu_j\Phi_{j},
\ee
satisfies $P_{K}\big(\delta\mathcal{G}(\Phi_\mu+Y(\Phi_\mu))\big)=\nabla_\mu G(\mu),$ for every $\mu\in B_{\rho}$.
			\end{itemize}
		\end{proposition}
        We point out that the analyticity of the map $Y$ in \cref{lemma:Lyapunov-Schmidt} will be crucial to apply the Łojasiewicz inequality in \cref{lemma:uguale} below, and it follows by the fact that $\mathcal{G}$ is analytic in a neighborhood of $0$, since $b\in\mathcal{B}_0$. 
        \subsection{Decomposition of the trace and definition of the competitor} Let $c\in H^1(\partial B_1)$ be a non-negative trace  satisfying the closeness assumptions \eqref{eq:closness-assumption-radial}. We decompose $\phi:=c-b$ as 
        $$\phi=\phi^K+\phi^\perp,\quad \text{where}\quad \phi^K\in K,\ \phi^\perp\in K^\perp.
        $$
        Recalling the map $Y$ from \cref{lemma:Lyapunov-Schmidt}, we observe that the elements in $K$ are regular enough to apply \cref{lemma:Lyapunov-Schmidt}. Hence, if we choose $\delta$ small enough, we have $\phi^K\in U$.
        Since $Y(\phi^K)\in K^\perp$ and $K^\perp=K_<\oplus K_>$, we can decompose 
        $$
        \phi^\perp-Y(\phi^K) = \phi_< +\phi_>,$$ where $\phi_<\in K_<$ and $\phi_>\in K_>$ are respectively the lower and higher modes corresponding to the linearized operator $L_b$. If we set $\phi_0:=\phi^K+Y(\phi^K)$, then we decompose the trace $c$ as $$c=b+\phi=b+\big(\phi^K+Y(\phi^K)\big)+\big(\phi^\perp-Y(\phi^K)\big)=b+\phi_<+\phi_0+\phi_>.$$ 
		Moreover, recalling the orthonormal basis $\Phi_1,\ldots,\Phi_{N}$ of $K$ and the definition of $\Phi_\mu$ in \eqref{def:G-finito-dim}, we can write $$\phi_0=\Phi_{\mu^0}+Y(\Phi_{\mu^0})=\sum_{j=1}^N \mu_j^0\Phi_{j}+Y\left(\sum_{j=1}^N \mu_j^0\Phi_{j}\right),$$
        where $\mu^0 :=(\mu_1^0,\ldots,\mu_N^0)\in\R^N$ corresponds to the coordinates of $\phi^K$ in the kernel $K$.

		The $\beta$-homogeneous extension of $c$ is $$z(r,\theta)=r^{\beta} b(\theta)+r^{\beta}\phi_<(\theta)+r^{\beta}\phi_0(\theta)+r^{\beta}\phi_>(\theta).$$ We define the competitor $$h(r,\theta):=r^{\beta} b(\theta)+r^{\beta}\psi_<(r,\theta)+r^\beta \psi_0(r,\theta)+r^{\beta}\psi_>(r,\theta),$$ where, for some $\rho=\rho(d,\gamma,b)\in(0,1)$ and $\tau=\tau(d,\gamma,b)>0$ to be chosen later, we set 
        \be\label{e:psi<>}\psi_<(r,\theta):=\frac{(r-\rho)_+^{\beta-\tau}}{r^\beta(1-\rho)^{\beta-\tau}}\phi_<(\theta),\qquad \psi_>(r,\theta):=\frac{(r-\rho)_+^{\beta+\tau}}{r^\beta(1-\rho)^{\beta+\tau}}\phi_>(\theta),\ee
		$$\psi_0(r,\theta):=\Phi_{\mu_j(\eta(r))}+Y(\Phi_{\mu_j(\eta(r))})=\sum_{j=1}^N \mu_j(\eta(r))\Phi_j(\theta)+Y\left(\sum_{j=1}^N \mu_j(\eta(r))\Phi_j(\theta)\right),$$ where $\mu_j(t)$, with $j=1,\ldots,N$, and $\eta(r)$  are defined as follows.
		
		Given $G$ as the reduced functional associated to $\mathcal{G}$ at $b$ (see \eqref{def:G-finito-dim}), 
		we define
		$\mu(t):=(\mu_1(t),\ldots,\mu_N(t))$ through the gradient flow \be\label{eq:flow}\begin{cases}
			\mu'(t)=-\frac{\nabla_\mu G(\mu(t))}{|\nabla_\mu G(\mu(t))|},\\
			\mu(0)=\mu^0, 
		\end{cases}\ee and we take $\mu'(t)=0$ if $|\nabla_\mu G(\mu(t))|=0$.
		On the other hand, $\eta(r)$ is chosen in such a way \be\label{def:etaeta1}\eta(r):=\begin{cases}
			0\quad&\text{for }r\in[\rho,1),\\
			\eta_1(r)\quad&\text{for }r\in(0,\rho)
		\end{cases}\ee for some $\eta_1(r)$ to be chosen, with $\eta_1(\rho)=0$.
		\begin{remark}
        We point out the effect of the correction term $\eta:[0,1] \to \R$. If we can choose $\eta_1\equiv0$ (for example, if $\phi_0\equiv0$), then no error term comes from the flow, and we can prove a classical epiperimetric inequality (i.e., $\sigma=0$). This happens also in the case when $b$ is sub-integrable (see \cref{def:sub-integrability}). But in general, we need to choose $\eta_1$ depending on a suitable power of the energy, and thus we get $\sigma>0$ (see \cref{lemma:uguale} and \cref{subsec:ultima} below).
        \end{remark}
		\begin{remark}\label{r:h-admissible}
		We observe that $h$ is an admissible competitor. Indeed, we have that $h=c$ on $\partial B_1$ and, given $\rho$, $\tau$ and $\eta_1$, we can choose $\delta$ small enough, so that the competitor $h$ is non-negative in $B_1$. Indeed, by the closeness assumptions \eqref{eq:closness-assumption-radial} and since $K_<$ and $K$ are finite-dimensional, $$\|\phi_<\|_{L^\infty(\partial B_1)}+\|\phi^K\|_{L^\infty(\partial B_1)}\le C\|\phi_<\|_{L^2(\partial B_1)}+C\|\phi^K\|_{L^2(\partial B_1)}\le C\|\phi\|_{L^2(\partial B_1)} \le C\delta.$$
		Moreover, since $\|Y(\phi^K)\|_{L^\infty(\partial B_1)}\le C\delta$, $$\|\phi_>\|_{L^\infty(\partial B_1)}\le \|\phi\|_{L^\infty(\partial B_1)}+\|\phi_<\|_{L^\infty(\partial B_1)}+\|\phi_0\|_{L^\infty(\partial B_1)}\le C\delta.$$ Therefore, the function $\psi:=\psi_<+\psi_0+\psi_>$ inherits the bound $
\|\psi\|_{L^\infty(B_1)}\le C\delta$. Finally, taking $\delta$ sufficiently small, we infer that
		$$h=r^\beta b+r^\beta\psi\ge r^\beta (\overline c- C\delta )\ge0,$$ where $\overline c:=\inf_{\partial B_1}b>0$.
		\end{remark}
		\subsection{Computations for the epiperimetric inequality}
        First, by setting $\overline W(v):=W(v)-W(b)$, we need to prove the epiperimetric inequality
        $$\overline{W}(h) - (1- \eps |\overline{W}(z)|^\sigma ) \overline W(z) \leq 0.$$
        
		We start by decomposing the quantity $\overline{W}(h) - (1- \eps) \overline W(z)$ into the sum of three terms, $E_<$, $E_0$ and $E_>$, corresponding respectively to the contributions of the lower, zero, and higher modes. In the next subsections, we will show that the sum of these three quantities is non-positive, and thus we prove the epiperimetric inequality.
		
		For the sake of readability, in the following we use the notation $$\psi_r:=\psi_<^r+\psi_0^r+\psi_>^r,\quad\text{where}\quad \psi_<^r=\psi_<(r,\cdot),\quad \psi_0^r=\psi_0(r,\cdot),\quad \psi_>^r=\psi_>(r,\cdot)$$
       Since $h=r^\beta (b+\psi_r)$, with $\psi_1=\phi$, the slicing lemma \cref{lemma:slicing} yields
		\bea \overline W(h)-(1- \eps)\overline W(z)&=\int_0^1r^{d+2\beta-3}\Big(\mathcal{F}( b+ \psi_r)-\mathcal{F}(b)-(1-\eps)(\mathcal{F}(b+\phi)-\mathcal{F}(b))\Big)\,d\HH^{d-1}\,dr\\&\qquad +\int_0^1\int_{\partial B_1}r^{d+2\beta-1}(\partial_r \psi_r)^2\,d\HH^{d-1}\,dr.\eea 
		By \eqref{eq:cond-on-perp}, and using that $\psi_<^r\equiv \psi_>^r\equiv0$ for $r\in(0,\rho)$ and $\psi_0^r=\phi_0$ for $r\in(\rho,1)$, we have
		$$\delta\mathcal{G}(\phi_0)[\phi_<+\phi_>]=0\quad\text{and}\quad \delta\mathcal{G}(\psi_0^r)[\psi_<^r+\psi_>^r]=0\quad\text{for every }r\in(0,1)$$
        Moreover, it is immediate to verify that if $\psi_1\in K_<$ and $\psi_2\in K_>$, then $$\mathcal{Q}(\psi_1+\psi_2)=\mathcal{Q}(\psi_1)+\mathcal{Q}(\psi_2),$$ 
		where $K_<$, $K_>$ are the eigenspaces corresponding to negative and positive eigenvalues respectively, defined in \eqref{eq:E<andE>}. Since $\phi_\gtrless,\psi_\gtrless^r\in K_\gtrless$ for every $r\in(0,1)$, the decomposition in \cref{lemma:decomposition} implies that
		\bea \overline W(h)-(1-\eps)\overline W(z)&=
		\int_0^1 r^{d+2\beta-3}\Big( \mathcal{Q}(\psi_<^r)-(1-\eps)\mathcal{Q}(\phi_<)+ \mathcal{G}(\psi_0^r)-(1-\eps)\mathcal{G}(\phi_0)\\&\qquad+\mathcal{Q}(\psi_>^r)-(1-\eps)\mathcal{Q}(\phi_>)
		\Big)\,dr+R(\phi_< +\phi_>)\\&\qquad \qquad+\int_0^1\int_{\partial B_1}r^{d+2\beta-1}(\partial_r \psi_r)^2\,d\HH^{d-1}\,dr,
		\eea 
        where, choosing $\delta$ (depending on $\tau$) small enough, we have \be\label{e:E1}|R(\phi_< + \phi_>)|\le \tau^2 \int_{\partial B_1}\phi_<^2\,d\HH^{d-1}+\tau^2 \int_{\partial B_1}\phi_>^2\,d\HH^{d-1}.\ee
		Therefore we have that
		\be\label{eq:epi-with-E} \overline W(h)-(1-\eps)\overline W(z)\le E_<+E_0+E_>,\ee where
		\bea E_\gtrless&:=\int_0^1 r^{d+2\beta-3}\Big( \mathcal{Q}(\psi_\gtrless^r)-(1-\eps)\mathcal{Q}(\phi_\gtrless)\Big)\,dr+\tau^2\int_{\partial B_1}\phi_\gtrless^2\,d\HH^{d-1}\\&\qquad+2\int_0^1\int_{\partial B_1} r^{d+2\beta-1}(\partial_r \psi_\gtrless^r)^2\,d\HH^{d-1}\,dr,
		\eea and
		\bea E_0&:=\int_0^1 r^{d+2\beta-3}\Big( \mathcal{G}(\psi_0^r)-(1-\eps)\mathcal{G}(\phi_0)\Big)\,dr+2\int_0^1\int_{\partial B_1} r^{d+2\beta-1}(\partial_r \psi_0^r)^2\,d\HH^{d-1}\,dr.
		\eea  
	
		\subsection{Estimate of $E_<$ and $E_>$}
		We first estimate $E_<$ and $E_>$ in the following lemma. 
		\begin{lemma}\label{lemma:maggiore}
			There are $\tau>0$, $\rho\in(0,1)$, $\eps_0>0$ and $c_0>0$ depending only on $d$, $\gamma$ and $b$ such that the following holds. For every $\eps\in(0,\eps_0]$, if $\psi_\gtrless$ is chosen as in \eqref{e:psi<>} using these values of $\tau$ and $\rho$, then we have \be\label{e:eqE>E<} E_<+E_>\le -c_0\big(|\mathcal{Q}(\phi_<)|+|\mathcal{Q}(\phi_>)|\big).\ee  
		\end{lemma}
		\begin{proof}
			Throughout the proof, the symbols $\pm$ and $\gtrless$ are paired so that $+$ corresponds to $>$, while $-$ corresponds to $<$. Moreover, the constants $\tau$ and $\rho$ are chosen below.
			
            First of all, we observe that, since $\phi_\gtrless\in K_\gtrless$, then $\mathcal{Q}(\phi_\gtrless)\gtrless 0$. 
            Moreover, by decomposing $\phi_\gtrless$ in eigenfunctions of $L_b$, it is immediate to verify that \be\label{eq:est>}|\lambda_{\gtrless}|\int_{\partial B_1}\phi_{\gtrless}^2\,d\HH^{d-1}=\pm\lambda_{\gtrless}\int_{\partial B_1}\phi_\gtrless^2\,d\HH^{d-1}\le \pm \mathcal{Q}(\phi_\gtrless)=|\mathcal{Q}(\phi_\gtrless)|,\ee where $\lambda_{>}$ is the smallest positive eigenvalue and $\lambda_<$ is the largest negative eigenvalue.
			
			Recalling the definition of $\psi_\gtrless^r$ in \eqref{e:psi<>}, and exploiting the sign of  $\mathcal{Q}(\phi_\gtrless)$ in \eqref{eq:est>},
			we have \bea \int_0^1r^{d+2\beta-3}\mathcal{Q}(\psi_\gtrless^r)\,dr&=\pm\frac{|\mathcal{Q}(\phi_\gtrless)|}{(1-\rho)^{2\beta\pm 2\tau}}\int_\rho^1\frac{(r-\rho)^{2\beta\pm 2\tau}}{r^{2\beta}}r^{d+2\beta-3}\,dr
			\\&\le 
			|\mathcal{Q}(\phi_\gtrless)|\left(\pm \frac{1}{d+2\beta-2\pm 2\tau}+ C\rho\right),
			\eea for $\rho$ small enough.
			Moreover, a direct computation of $\partial_r\psi_\gtrless^r$ shows that \bea \int_{\partial B_1} (\partial_r \psi_\gtrless^r)^2\,d\HH^{d-1}=\frac{(r-\rho)^{2\beta\pm 2\tau-2}}{r^{2\beta}(1-\rho)^{2\beta\pm2\tau}}\left(\pm \tau+\beta \frac \rho r\right)^2
			\int_{\partial B_1}\phi_\gtrless^2\,d\HH^{d-1}.\eea Then, by using \eqref{eq:est>}, we have \bea\int_0^1\int_{\partial B_1} r^{d+2\beta-1}(\partial_r \psi_\gtrless^r)^2\,d\HH^{d-1}\,dr&
			\le C\left(\tau^2+\rho^2\right)\int_{\partial B_1}\phi_\gtrless^2\,d\HH^{d-1}\le C\left(\tau^2+\rho^2\right)|\mathcal{Q}(\phi_\gtrless)|.
			\eea  
			We also have $$\int_0^1r^{d+2\beta-3}\mathcal{Q}(\phi_\gtrless)\,dr=\pm\frac{|\mathcal{Q}(\phi_\gtrless)|}{d+2\beta-2}.$$ 
			
			Choose $\tau$, and set $\eps_0:=\frac{\tau}{d+2\beta-2}$ and $\rho:=\tau^{2}$. By collecting the previous computations, for every $\eps\in(0,\eps_0]$, we have \bea E_\gtrless&\le\left(\pm\frac{1}{d+2\beta-2\pm2\tau}\mp(1-\eps)\frac{1}{d+2\beta-2}+C\rho+C\tau^2+C\rho^2\right)|\mathcal{Q}(\phi_\gtrless)|.\eea
            Then
            \bea E_>\le\left(\frac{1}{d+2\beta-2+2\tau}-\left(1-\frac{\tau}{d+2\beta-2}\right)\frac{1}{d+2\beta-2}+C\tau^{2}\right)|\mathcal{Q}(\phi_>)|,\eea
            while \bea E_<\le\left(-\frac{1}{d+2\beta-2-2\tau}+\frac{1}{d+2\beta-2}+C\tau^2\right)|\mathcal{Q}(\phi_<)|.\eea
            Therefore, there exists $c'_0>0$ such that \bea E_\gtrless&\le(-2c'_0\tau+C\tau^{2})|\mathcal{Q}(\phi_\gtrless)|\le -c'_0\tau|\mathcal{Q}(\phi_\gtrless)|,
			\eea where in the last inequality we choose $\tau$ small enough, concluding the proof.
		\end{proof}

		\subsection{Estimate of $E_0$}
        We proceed by estimating $E_0$, using as a key ingredient the Łojasiewicz inequality.
		\begin{lemma}\label{lemma:uguale}
			Let $\mathcal{G}(\phi_0)>0$ and let $\rho\in(0,1)$. Then, there is $\overline\eps=\overline\eps(d,\gamma,b,\rho)>0$, $C_0=C_0(d,\gamma,b,\rho)>0$ and $\sigma=\sigma(d,\gamma,b)\in(0,1/2]$, such that,
			setting $$\eta_1(r):=C_0\eps \mathcal{G}(\phi_0)^\sigma(\rho-r)_+\quad\text{and}\quad \eps:=\overline \eps\,\mathcal{G}(\phi_0)^{1-2\sigma},$$ and choosing $\delta=\delta(d,\gamma,b,\rho)$ small enough, we have $E_0\le 0.$
		\end{lemma}
		\begin{proof} 
			Let $G$ be the function defined in \eqref{def:G-finito-dim} and $\mu(t)$ be the gradient flow defined in \eqref{eq:flow}. By the Łojasiewicz inequality for the analytic function $G$ (see \cite{lojasiewicz}), there is $\sigma\in(0,1/2]$ and $U\subset \R^N$ such that \bea\label{eq:lojasievic} |G(\mu)|^{1-\sigma}\le C_L|\nabla_\mu G(\mu)|\quad\text{for every }\mu\in U
			\eea for some $C_L>0$.
			In particular, if we choose $\delta$ small enough we can apply the Łojasiewicz inequality, since $\mu(\eta(r))\in U$ for every $r\in(0,1)$.
    			Therefore $$\frac{d}{dt}G(\mu(t))=\nabla_\mu G(\mu(t))\cdot\mu'(t)=-|\nabla_\mu G(\mu(t))|\le -C_L^{-1}|G(\mu(t))|^{1-\sigma}.$$ 
            Then, setting $C_1:=C_L^{-1}\sigma$, the function $t\mapsto G(\mu(t))^{\sigma}+C_1t$ is non-increasing. Moreover, denoting by  $T_\ast:=C_1^{-1} G(\mu^0)^\sigma$, 
            we obtain \be\label{eq:previous-inequality} G(\mu(t))\le \left(G(\mu^0)^\sigma-C_1 t\right)^{1/\sigma}\le G(\mu^0)-C_1G(\mu^0)^{1-\sigma}t\quad\text{for every }t\in(0,T_\ast),\ee where in the last inequality we used that $(1-a)^{1/\sigma}\le 1-a$, where $a:=t/T_\ast\in(0,1)$.
			
			Now, set $$\eta_1(r):=\frac{d+2\beta}{C_1}\eps G(\mu^0)^\sigma\frac{(\rho-r)_+}{\rho^{d+2\beta-1}},$$ and consider $\eps$ sufficiently small so that $\eta_1(r)\le T_\ast$ for every $r\in(0,\rho)$. Then the previous inequality \eqref{eq:previous-inequality} with $t=\eta(r)$ gives
			\be\label{e:prima}G(\mu(\eta(r)))-(1-\eps)G(\mu^0)\le -(C_1 \eta_1(r)-\eps G(\mu^0)^\sigma)G(\mu^0)^{1-\sigma}.\ee
			Observing that $\mathcal{G}(\psi^r_0)=G(\mu(\eta(r)))$ and $\mathcal{G}(\phi_0)=G(\mu^0)$, we need to estimate \bea E_0&=\int_0^1 r^{d+2\beta-3}\Big(G(\mu(\eta(r)))-(1-\eps) G(\mu^0)\Big)\,dr+\int_0^1\int_{\partial B_1} r^{d+2\beta-1}(\partial_r \psi_0^r)^2\,d\HH^{d-1}\,dr \\&:=I_1+I_2,\eea 
            and we proceed as follows. For $I_1$, by applying \eqref{e:prima}, we compute 
            \bea I_1&\le -G(\mu^0)^{1-\sigma}\int_0^\rho r^{d+2\beta-3}\Big(C_1\eta_1(r)-\eps G(\mu^0)^\sigma\Big)\,dr+\int_\rho^1r^{d+2\beta-3}\eps G(\mu^0)\,dr\\&=-\eps G(\mu^0)\bigg(\int_0^\rho r^{d+2\beta-3}\Big((d+2\beta)\frac{\rho-r}{\rho^{d+2\beta-1}}-1 \Big)\,dr-\frac{1-\rho^{d+2\beta-2}}{d+2\beta-2}\bigg)
			\\&=-\frac{\eps G(\mu^0)}{(d+2\beta-2)(d+2\beta-1) },
			\eea where the last identity follows by an explicit integration.
			Regarding $I_2$, since $|\eta_1'(r)|^2\le C(\rho)\eps^2G(\mu^0)^{2\sigma},$ then
			we have $I_2\le C(\rho)\eps^2G(\mu^0)^{2\sigma}$.
			Therefore 
            $$E_0\le -\frac{\eps G(\mu^0)}{(d+2\beta-2)(d+2\beta-1) }+C(\rho)\eps^2G(\mu^0)^{2\sigma}.
            $$ Setting $\eps=\overline \eps G(\mu^0)^{1-2\sigma}$ and choosing $\overline \eps>0$ small enough depending on $\rho$, we get $E_0\le0$, concluding the proof.
		\end{proof}
		\subsection{Conclusion of the proof}\label{subsec:ultima} 
		We now conclude the proof of the epiperimetric inequality.
		\begin{proof}[Proof of \cref{t:epi-rad}]
			It is not restrictive to assume that $\overline W(z)>0$, otherwise the claim follows by choosing $h=z$. Moreover, we recall that, by slicing \cref{lemma:slicing} and the decomposition of the energy \cref{lemma:decomposition},
        \begin{align}\label{e:overWz}
        \begin{aligned}
        \overline{W}(z)
        &=\frac{1}{d+2\beta-2}\Big(\mathcal{G}(\phi_0)+\mathcal{Q}(\phi_<)+\mathcal{Q}(\phi_>)\Big) + R(\phi_<+\phi_>),
        \end{aligned}
        \end{align}
			with $R$ satisfying \eqref{e:E1}. 
            We also observe that, if we choose $\eta_1\equiv 0$, then 
        \be\label{e:estimate-simple-E0}
        E_0  = \eps\frac{\mathcal{G}(\phi_0)}{d+2\beta-2},
        \ee since, by definition of $\psi_0^r$ we have $\psi_0^r = \phi_0$.
            
            Now, let $\tau$, $\rho$ , $\eps_0$ and $c_0$ be the constants in \cref{lemma:maggiore}, and let $\psi_\gtrless$ be defined corresponding to the associated choices of $\tau$ and $\rho$. Thus, by combining \eqref{eq:epi-with-E} with \eqref{e:eqE>E<} we have
            \be\label{eq:fund-est-radial}
            \overline{W}(h)-(1-\eps)\overline{W}(z) \leq  -c_0\big(|\mathcal{Q}(\phi_<)|+|\mathcal{Q}(\phi_>)|\big)+E_0,
            \ee
            for every $\eps \in (0,\eps_0]$. Now, we split the proof into three cases.
			\begin{itemize}
				\item[(i)] If $\mathcal{G}(\phi_0)\le0$, then we can choose $\eta_1\equiv0$ so that $E_0$ satisfies \eqref{e:estimate-simple-E0}, and thus $E_0\le0$. Then the epiperimetric inequality holds with $\eps=\eps_0$ and $\sigma=0$, by \eqref{eq:fund-est-radial}
				
				\item[(ii)] Similarly, if  
                $$ \mathcal{G}(\phi_0)>0\quad \text{and}\quad |\mathcal{Q}(\phi_<)|+|\mathcal{Q}(\phi_>)|> \mathcal{G}(\phi_0),$$  we can still choose $\eta_1\equiv0$. Indeed, by combining \eqref{e:estimate-simple-E0} with \eqref{eq:fund-est-radial}, we have \bea \overline W(h)-(1-\eps)\overline W(z)&\le - c_0\left(|\mathcal{Q}(\phi_<)|+|\mathcal{Q}(\phi_>)|\right)+\frac{\eps}{d+2\beta-2}\mathcal{G}(\phi_0)\\&\le \left(-c_0+\frac{\eps}{d+2\beta-2}\right)\left(|\mathcal{Q}(\phi_<)|+|\mathcal{Q}(\phi_>)|\right),\eea
				and the epiperimetric inequality holds with $\eps=\min\{\eps_0,c_0(d+2\beta-2)\},$ and $\sigma=0$.
				\item[(iii)] Lastly, if
                $$\mathcal{G}(\phi_0)>0\quad\text{and}\quad
                |\mathcal{Q}(\phi_<)|+|\mathcal{Q}(\phi_>)|\le \mathcal{G}(\phi_0),$$ we choose $\eta_1$ as in \cref{lemma:uguale}. 
				By \cref{r:h-admissible}, since $\mathcal{G}(\phi_0)$ is small for $\delta$ small, it is not restrictive to assume that $\eps:=\overline \eps \mathcal{G}(\phi_0)^{1-2\sigma} \in(0,\eps_0)$, and so
                $$\overline W(h)-(1-\eps)\overline W(z)\le0,\quad \text{where } \eps:=\overline \eps \mathcal{G}(\phi_0)^{1-2\sigma} 
                $$
				On the other hand, by \eqref{eq:est>} and \eqref{e:overWz}, we have $\overline W(z)\le C_2\mathcal{G}(\phi_0)$,
                and thus, since $\overline W(z)>0$, it follows that $$\overline W(h)-(1-\overline\eps C_2^{2\sigma-1}\overline W(z)^{1-2\sigma})\overline W(z)\le\overline W(h)-(1-\overline\eps \mathcal{G}(\phi_0)^{1-2\sigma})\overline W(z)\le0. $$ Then the epiperimetric inequality holds with exponent $\sigma':=1-2\sigma\in[0,1)$ and constant $\eps':=\overline\eps C_2^{2\sigma-1}$. 
			\end{itemize}
			Finally, by the closeness assumption on the Weiss' energy in \eqref{eq:closness-assumption-radial}, we may assume that $$
            \eps'\,\overline W(z)^{\sigma'}\le \min\{\eps_0,c_0(d+2\beta-2)\},
            $$
            which concludes the proof of the logarithmic epiperimetric inequality.
		\end{proof}
		\begin{remark}\label{remark:sigma=0} We observe that if $b$ is sub-integrable, i.e., $G\le0$ in a neighborhood of $0$, then we are always in the case (i) above, i.e., $\mathcal{G}(\phi_0)\leq 0$. Then we can choose $\eta_1\equiv0$, so that $E_0\le0$ by \eqref{e:estimate-simple-E0}. 
        Then the epiperimetric inequality holds with $\sigma=0$.
			\end{remark}
	\subsection{Optimality of the logarithmic epiperimetric inequality}
		As already observed in \cref{remark:sigma=0}, when $b$ is sub-integrable one may choose $\sigma=0$ in the epiperimetric inequality, which in turn yields a $C^{1,\alpha}$ rate of convergence for blow-up sequences (see \cref{section:final}). We now show that this condition is in fact equivalent to sub-integrability: if $b$ is not sub-integrable, then one necessarily has $\sigma>0$, and the logarithmic rate of convergence in \eqref{eq:log-conv} is sharp.
		
		Let $b\in\mathcal{B}_0$ be a cone that is not sub-integrable.
        By following the notation of \cref{lemma:Lyapunov-Schmidt}, let $G:\R^N \to \R$ be the reduced (analytic) functional arising from the Lyapunov-Schmidt reduction (see \eqref{def:G-finito-dim}). We may suppose that $G\not\equiv0$, otherwise $b$ is integrable (see \cite[Lemma 1]{adamsimon}). Then, since $|\nabla_\mu G(0)|=|\nabla^2_\mu G(0)|=0$, there exists an integer $p\geq 3$ such that \be\label{eq:exp-G}G(\mu)=\sum_{j=p}^{+\infty} G_j(\mu),\ee where each $G_j(\mu)$ is a $j$-homogeneous polynomial and $G_p\not\equiv0$ is the first nontrivial term.

In \cite{adamsimon} Adams-Simon construct solutions with a slower logarithmic rate under a positivity assumption for the first non-trivial term $G_p$ (i.e., $G_p(\mu^0)>0$ for some $\mu^0\in \mathbb{S}^{N-1}$). Nevertheless, the property of not sub-integrability do not ensure the validity of the Adams-Simon positivity condition.

In the following lemma, we exploit the construction developed in our companion paper \cite{CT3} to exhibit a slow converging solution.

\begin{lemma}\label{lemma:adamsimon}
			Let $b \in \mathcal{B}_0$ be a not sub-integrable cone. 
			Then, one may construct a solution $u$ to the Alt-Phillips problem whose blow-up at the origin is $b$ and such that, for $r$ sufficiently small,
			\be\label{e:log-decay}\|u_r-b\|_{L^\infty(B_1)}\ge \frac c{|\log r|^{1/(\ell_\ast-2)}},\ee for some constant $c>0$ and $\ell_\ast\ge p$. In particular, if $G_p(\mu^0)>0$ for some $\mu^0\in\mathbb{S}^{N-1}$, then \eqref{e:log-decay} holds with $\ell_\ast=p$.
\end{lemma}
        
		\begin{proof}
            Following the notation in \cite{adamsimon, CT3}, we set $\Sigma:=\partial B_1$,
            $V:=\Sigma\times \R$ and $m:=d+2\beta-2>0$, then, for a function $\psi=\psi(t,\theta):(0,+\infty)\times \Sigma\to\R$, we define the functional $$\mathscr{F}(\psi):=\int_{(0,+\infty)\times \Sigma}F(\theta,\psi,\nabla_\theta \psi,\psi_t)e^{-mt}\,dt\,d\HH^{d-1},$$
            where $$F(\theta,z,p,q)=\frac12|p|^2+\frac12q^2-\frac{\lambda(\beta)}2z^2+\frac12\Big((b+z)^\gamma-b^\gamma-\gamma b^{\gamma-1}z\Big).$$ We observe that $\partial_{qq}F=1$ and $q \partial_{q}F=q^2$. Moreover, since $b>0$ on $\Sigma$, the map $z\mapsto (b+z)^\gamma$ is analytic for $|z|$ small and consequently, $F$ is analytic in a neighborhood of $(z,p,q)=(0,0,0)$.
            
            We define the Euler-Lagrange operator associated to $\mathscr{F}$ as 
            $$\mathcal{M}(\psi):=\psi_{tt}-m\psi_t+\mathcal{M}_\Sigma(\psi),\quad \mathcal{M}_\Sigma(\psi):=\Delta_\theta\psi+\lambda(\beta)\psi-\frac\gamma2((b+\psi)^{\gamma-1}-b^{\gamma-1}).$$ Since
            $\Delta_\theta b+\lambda(\beta)b=\frac\gamma2b^{\gamma-1}$ on $\partial B_1$,
            a direct computation shows that  
            $u(r,\theta):=r^\beta b(\theta)+r^\beta \psi(-\log r,\theta)$ is a solution to the Alt-Phillips problem 
            if and only if $\mathcal{M}(\psi)=0$. Notice that 
            $$
            \mathscr{F}_\Sigma(\phi)=\frac12\mathcal{G}(\phi)\quad \text{and}\quad \mathcal{M}_\Sigma(\phi)=-\frac12\delta\mathcal{G}(\phi),\quad \text{for every }\phi \in H^1(\partial B_1).
            $$
            Then the linearization $L_\Sigma$ of $\mathcal{M}_\Sigma$ at $0$, gives exactly the operator $L_b$, and thus the map $Y$ constructed in \cref{lemma:Lyapunov-Schmidt} coincides with the map $H\colon \ker(L_\Sigma)\to \ker(L_\Sigma)^\perp$ in \cite{CT3}. By following this analogy, the analytic finite-dimensional map arising from the Lyapunov-Schmidt reduction in \cite{CT3} is given by $f=\frac12G$, with $G$ as in \eqref{def:G-finito-dim}.

            Thus, the not sub-integrability of $b$ is equivalent to the not sub-integrability of the kernel $K_\Sigma$, in the sense of \cite{CT3}. 
            By applying \cite[Theorem 1.2]{CT3}, there exists a function $\psi=\psi(t,\theta)$ satisfying  $$\mathcal{M}(\psi)=0\text{ on }(0,+\infty)\times\Sigma\quad\text{and}\quad \psi(t,\theta)=\frac{\phi(\theta)}{(T+t)^{1/(\ell_\ast-2)}}+O((T+t)^{-1/(\ell_\ast-2)-\eps}),$$ for some $\phi\in K\setminus\{0\}$, $T>1$, $\eps>0$ and $\ell_\ast\geq p$. Therefore, the function $u$ defined as $u(r,\theta)=r^\beta b(\theta)+r^\beta \psi(-\log r,\theta)$ satisfies the logarithmic decay \eqref{e:log-decay}. 
            
            Finally, if there exists $\mu^0\in\mathbb{S}^{N-1}$ such that $G_p(\mu^0)>0$, by \cite[Theorem 1]{adamsimon} the estimate \eqref{e:log-decay} holds with $\ell_\ast=p$.
		\end{proof}
		 
        We point out that in \cref{lemma:adamsimon} we constructed only a solution of the Alt-Phillips problem, and not necessarily a minimizer. These two notions are equivalent only for $\gamma\in[1,2)$, see \cref{prop:convex-functional}. 
\section{Integrability and bifurcations of the radial cone}\label{section:integrable-bifurcation}
		In this section we study the integrability (and the sub-integrability) of the radial cone $b_{\text{rad}}:=c_{\text{rad}}|x|^\beta$, as well as the bifurcation results for $b_{\text{rad}}$ in \cref{thm:bifurcation-radial}.
        As a key ingredient, we will use the numerical computations in \cref{prop:numerical} below to write the first order terms in the expansion of the reduced functional $G$ at $b_{\text{rad}}$ in \eqref{eq:exp-G}.
		
 We recall the kernel $K:=\ker(L_{b_{\mathrm{rad}}})$ of the linearized operator at $b_{\text{rad}}$, defined in \eqref{eq:general-kernel}, the notions of integrability and sub-integrability from \cref{def:integrability} and \cref{def:sub-integrability}, as well as the values $\gamma_{k,d}$ introduced in \eqref{eq:def-gammakd} for $k\in\N_{\ge3}$.     
\subsection{Integrability of the radial cone}
        The main result about the integrability of the radial cone is the following proposition.  
		\begin{proposition}\label{prop:optimality-log}
			Let $d\ge 2$ and $\gamma\in(0,2)\setminus \{1\}$. Then,
			\begin{itemize}
				\item[(i)] if $\gamma\not=\gamma_{k,d}$ for every $k\in\N_{\ge3}$, then $K=\{0\}$ and thus $b_{\text{rad}}$ is integrable; 
				\item[(ii)] if $\gamma=\gamma_{k,d}$ for some $k\in\N_{\ge3}$, then $K$ coincides with the space of spherical harmonics $\mathcal{H}_k(\mathbb{S}^{d-1})$ and $b_{\text{rad}}$ is not sub-integrable (thus non-integrable).
			\end{itemize}
		\end{proposition}   
We point out that the case of the obstacle problem $\gamma=1$ corresponds to $k=2$, since $\gamma_{2,d}=1$ for every $d\ge2$. In this case, the kernel $K$ coincides with the space of spherical harmonics $\mathcal{H}_k(\mathbb{S}^{d-1})$, but, unlike \cref{prop:optimality-log}, $b_{\text{rad}}$ is integrable. 

\begin{remark}\label{remark:after}
Note that the characterization of the kernel in \cref{prop:optimality-log} is a straightforward consequence of the definition of $b_{\text{rad}}$. In fact, since $c_{\text{rad}}^{\gamma-2}=\frac2\gamma\lambda(\beta)$ we have
\be\label{eq:linearized-operator-radial}
L_{b_{\text{rad}}}=-\Delta_\theta - \lambda(\beta) + \frac{\gamma}{2}(\gamma-1)c_{\text{rad}}^{\gamma-2} = -\Delta_\theta -(2-\gamma)\lambda(\beta)
\ee
and the kernel $K$ is nontrivial if and only if 
$$
(2-\gamma)\lambda(\beta) = \lambda(k),\quad \text{for some $k\in \N_{\geq 3}$}.
$$
For $\gamma\in(0,2)\setminus \{1\}$, this condition is equivalent to requiring that  $\gamma=\gamma_{k,d}$, for some $k\in\N_{\ge3}$, and, in this case, $K$ coincides with the space of spherical harmonics $\mathcal{H}_k(\mathbb{S}^{d-1})$.
\end{remark}
		
		\subsection{Some numerical computations}
		In order to study the sub-integrability of the radial cone and the bifurcation results, we need to compute the first nontrivial term in the expansion \eqref{eq:exp-G} of the reduced functional $G$ associated to $\mathcal{G}$ at $b_{\text{rad}}$.
		In particular, we prove the following lemma. 
		\begin{lemma}\label{prop:numerical}
			Let $d\ge2$, $k\in\N_{\ge3}$ and $\gamma=\gamma_{k,d}$. Consider an orthonormal basis  $\Phi_1,\ldots,\Phi_N$ of $K$, and take $k_d$ as in \cref{thm:bifurcation-radial}. 
            Thus, let $G$ be the reduced functional defined in \eqref{def:G-finito-dim} associated to the radial cone $b_{\text{rad}}$, and $G_p$ be the first nontrivial term in the expansion \eqref{eq:exp-G}.
			\begin{itemize}
				\item[(i)] If $d=2$, then $p=4$ and $G_4(\mu)>0$ for every $\mu\in\mathbb{S}^{N-1}$.
				\item[(ii)] If $d\ge3$ and $k$ is odd, then $p=4$ and $G_4(\mu^0)>0$ for some $\mu^0\in\mathbb{S}^{N-1}$. More precisely 
				\begin{itemize}
					\item[1)] if $\Phi_1$ is the zonal spherical harmonic and $k<k_d$, then $G_4(e_1)>0$;
					\item[2)] if $\Phi_1$ is the sectorial spherical harmonic, then $G_4(e_1)>0$.
					\end{itemize}
				\item[(iii)] If $d\ge3$ and $k$ is even, then $p=3$ and $G_3(\mu^0)>0$ for some $\mu^0\in\mathbb{S}^{N-1}$. More precisely, if $\Phi_1$ is the zonal spherical harmonic, then $G_3(e_1)>0$.		
				\end{itemize}
		\end{lemma}
		\begin{proof}
            As already observed,
            if we set $\gamma^*:=\gamma_{k,d}$ and $b^\ast:=b_{\text{rad}}$, then 
            \be\label{e:already-obs}
K=\mathcal{H}_k(\mathbb{S}^{d-1}),\quad  
            (b^*)^{\gamma-2} = \frac{2}{\gamma(2-\gamma^*)}\lambda(k)\quad\text{and}\quad L_{b^\ast} = -\Delta_\theta - \lambda(k).
            \ee
            Now, given $\phi \in K$ to be chosen later, let $\Phi_1,\dots,\Phi_N$ be an orthonormal basis of $K$ with $\Phi_1=\phi$. Let $\mathcal{G}$ be as in \eqref{def:G} and $G$ be the associated reduced functional arising from the Lyapunov-Schmidt reduction. By following the notation of \cref{lemma:Lyapunov-Schmidt}, given $Y\colon K\cap U\to K^\perp$ we set
            $$
            \psi:=\delta^2 Y(0)[\phi,\phi] \in K^\perp.
            $$
            By differentiating two times the first equation in \eqref{eq:cond-on-perp} in $K^\perp$, and using that $Y(0)=0$, $\delta Y(0)=0$, we infer that 
            \be\label{eq:proj-psi}
            P_{K^\perp}\left(\delta^2 \mathcal{G}(0)[\psi]+\delta^3 \mathcal{G}(0)[\phi,\phi]  \right) =0,
            \ee
            namely
            \bea\label{eq:def-psi}\psi=-(P_{K^\perp}\delta^2\mathcal{G}(0))^{-1}P_{K^\perp}\delta^3\mathcal{G}(0)[\phi,\phi]=-\frac12(L_{b^\ast}|_{K^\perp})^{-1}P_{K^\perp}\delta^3\mathcal{G}(0)[\phi,\phi].
            \eea
            If we proceed by computing the terms $G_3(e_1)$ and $G_4(e_1)$, we get
			\be\label{eq:G3}\begin{aligned} G_3(e_1)&=\frac{1}{6}\partial^3_{\mu_1}G(0)=\frac{1}6\delta^3\mathcal{G}(0)[\phi,\phi,\phi]= \frac{\gamma^*(\gamma^*-1)(\gamma^*-2)(b^*)^{\gamma^*-3}}{6}\int_{\partial B_1}\phi^3\,d\HH^{d-1}\end{aligned}\ee and, exploiting the equation in \eqref{eq:proj-psi} by testing it with $\psi \in K^\perp$ \begin{equation}\label{eq:G4}\begin{aligned} G_4(e_1)&=\frac1{24}\partial^4_{\mu_1}G(0)=\frac{1}{24}\Big(\delta^4\mathcal{G}(0)[\phi,\phi,\phi,\phi]+6\delta^3\mathcal{G}(0)[\phi,\phi,\psi]+3\delta^2\mathcal{G}(0)[\psi,\psi]\Big),\\&=\frac{1}{24}\Big(\delta^4\mathcal{G}(0)[\phi,\phi,\phi,\phi]+3\delta^3\mathcal{G}(0)[\phi,\phi,\psi]\Big)\\&=\frac{1}{24}\left(c_4\int_{\partial B_1}\phi^4\,d\HH^{d-1}+3c_3\int_{\partial B_1} \phi^2\psi\,d\HH^{d-1}\right),\end{aligned}\end{equation}
			where $c_3:=\gamma^*(\gamma^*-1)(\gamma^*-2)(b^*)^{\gamma^*-3}$, $c_4:=\gamma^*(\gamma^*-1)(\gamma^*-2)(\gamma^*-3)(b^*)^{\gamma^*-4}$. Notice that since $\gamma^*>1$, we have $c_3<0$ and $c_4>0$. Now we divide the rest of the proof into several cases.
            \\ \vspace{-0.3cm}
			
			\noindent\textit{Case  1: $d=2$.} Then $N=2$ and $\Phi_1:=\frac{1}{\sqrt{\pi}}\cos(k\theta)$, $\Phi_2:=\frac{1}{\sqrt{\pi}}\sin(k\theta)$.
            Since $$\int_{0}^{2\pi}\cos(k\theta)^3\,d\theta=\int_{0}^{2\pi}\sin(k\theta)^3\,d\theta=0$$ for every $k\in \N_{\geq 3}$, we infer that $G_3\equiv0$. Therefore, up to a rotation, it is sufficient to show that $G_4(e_1)>0$. Since $\phi(\theta):=\Phi_1(\theta)$, we have 
			$$\psi(\theta):=-\frac12 L_b^{-1}\left(\frac{c_3}{2\pi}(1+\cos(2k\theta))\right)=\frac{c_3}{4\pi}\left(\frac{1}{k^2}-\frac{\cos(2k\theta)}{3k^2}\right).$$ Then an explicit computation shows that $$G_4(e_1) = \frac{1}{24}\left(c_4\int_{0}^{2\pi}\phi^4\,d\theta+3c_3\int_{0}^{2\pi} \phi^2\psi\,d\theta\right)=\frac{c_4}{32\pi}+\frac{5c_3^2}{192\pi k^2}>0.$$ 
			\\ \vspace{-0.3cm}
			
			\noindent\textit{Case 2: $d\ge3$ and $k$ even.} We proceed by showing that $p=3$ and $G_3(e_1)>0$. 
			
			Consider $\phi(\theta):=-h_{k,m}^{-1/2}C_k^m(\cos\theta_1)$, where $C_k^m$ is the Gegenbauer polynomial with $m:=\frac{d-2}{2}$ and $h_{k,m}$ is chosen in such a way that $\phi$ is normalized in $L^2(\partial B_1)$. 
			Using the linearization coefficients for Gegenbauer polynomials (see e.g.~\cite{gegenbauer}), we have that $C_k^m(t) ^2= \sum_{\ell=0,2,\ldots,2k} a_{k,\ell,m} C_\ell^m(t).$
            Then \be\label{eq:phi2}\phi(\theta)^2=h_{k,m}^{-1}\sum_{\ell=0,2,\ldots,2k} a_{k,\ell,m} C_\ell^m(\cos\theta_1).\ee
            More precisely, we have $$h_{k,m}=\frac{2^{1-2m}\pi\Gamma(k+2m)}{k!(k+m)\Gamma(m)^2}\alpha_{d-2},\quad\text{where }\alpha_{d-2}:=\HH^{d-2}(\mathbb{S}^{d-2}),$$ and 
			\bea\label{eq:def-a}a_{k,\ell,m}:=\frac{
				\Gamma(\ell+1) \Gamma(m+\ell/2)^2\Gamma(m+k-\ell/2)
				\Gamma(2m+k+\ell/2)\Gamma(\ell+m+1)}{
				\Gamma(\ell/2+1)^2\Gamma(k-\ell/2+1)\Gamma(m)^2\Gamma(\ell+2m)\Gamma(m+\ell)\Gamma(m+1+k+\ell/2)}.\eea
            Since $k$ is even and the Gegenbauer polynomials are orthogonal on $[-1,1]$ with respect to the weight $(1-t^2)^{(d-3)/2}$, we have
            $$\int_{\partial B_1}\phi^3\,d\HH^{d-1}=-\alpha_{d-2}h_{k,m}^{-3/2}a_{k,k,m}\int_{-1}^1 C_k^m(t)^2(1-t^2)^{(d-3)/2}\,dt=-h_{k,m}^{-1/2}a_{k,k,m}<0.$$ 
            Since $c_3<0$, by \eqref{eq:G3} we infer that $G_3(e_1)>0$.	
			\\ \vspace{-0.3cm}
			
			\noindent\textit{Case 3: $d=3,4,5,6,7$, $k$ odd and $\Phi_1$ zonal spherical harmonic.} Since $k$ is odd, the integral $\int_{\partial B_1}\varphi^3\,d\HH^{d-1}$ vanishes for every spherical harmonic $\varphi \in \mathcal{H}_k(\mathbb{S}^{d-1})$, and thus $G_3\equiv0$. Thus, we proceed by computing $G_4$ with a specific choice of $\phi$.
			
			Consider $\phi(\theta):=h_{k,m}^{-1/2}C_k^m(\cos\theta_1)$, with the sign opposite to that chosen in Case 2. Recalling the formula in \eqref{eq:phi2} and using that $k$ is odd, we can identify $\psi$ as the unique solution to $\Delta_\theta\psi+\lambda(k)\psi=\frac{c_3}{2}\phi^2$ on $\partial B_1$, namely $$\psi(\theta)=\sum_{\ell=0,2,\ldots,2k} \frac{c_3 }{2(\lambda(k)-\lambda(\ell))}h_{k,m}^{-1}a_{k,\ell,m} C_\ell^m(\cos\theta_1).$$ Therefore, by \eqref{eq:G4}, we have
			\bea G_4(e_1)&=\frac1{24}h_{k,m}^{-2}\sum_{\ell=0,2,\ldots,2k}|a_{k,\ell,m}|^2h_{\ell,m}\left(c_4+\frac{3c_3^2}{2(\lambda(k)-\lambda(\ell))}\right)\\&=\frac1{24}h_{k,m}^{-2}|c_3|(b^*)^{-1}\sum_{\ell=0,2,\ldots,2k}|a_{k,\ell,m}|^2h_{\ell,m}\left((3-\gamma^*)+\frac{3(\gamma^*-1)\lambda(k)}{\lambda(k)-\lambda(\ell)}\right),\eea
			where in the last equality we used that $c_4
            =|c_3|(b^*)^{-1}(3-\gamma^*)$ and \eqref{e:already-obs}. Precisely, 
            \be\label{eq:computationc3}c_3=\gamma^*(\gamma^*-1)(\gamma^*-2)(b^*)^{\gamma^*-3} 
            = 
            -2(\gamma^*-1)\lambda(k)(b^*)^{-1}.\ee
            Therefore, it is sufficient to show that \be\label{eq:def-I}I_{k,d}:=\sum_{\ell=0,2,\ldots,2k}|a_{k,\ell,m}|^2h_{\ell,m}\left((3-\gamma^*)+\frac{3(\gamma^*-1)\lambda(k)}{\lambda(k)-\lambda(\ell)}\right)>0.\ee
			
			First, consider the function  $R_d(k,\ell):=\frac{|a_{k,\ell,m}|^2h_{\ell,m}}{|a_{k,2k-\ell,m}|^2h_{2k-\ell,m}}$ with $0\le \ell\le k-1$, namely  \bea R_d(k,\ell)&=
			\frac{\Gamma(\ell/2+1/2)^2\Gamma(k+\ell/2+2m)^2\Gamma(k-\ell/2+m+1/2)^2\Gamma(2k-\ell/2+m+1)^2}{\Gamma(\ell/2+m+1/2)^2\Gamma(k+\ell/2+m+1)^2\Gamma(k-\ell/2+1/2)^2\Gamma(2k-\ell/2+2m)^2}\\&\qquad \cdot \frac{(\ell+m)\Gamma(\ell+2m)\Gamma(2k-\ell+1)}{(2k-\ell+m)\Gamma(\ell+1)\Gamma(2k-\ell+2m)}.\eea
			Then, using the following inequality for the Gamma function $$\frac{\Gamma(x+1/2)^2}{\Gamma(x)^2}\ge \left(\sqrt{x}\left(1-\frac1{8x}\right)\right)^2\ge x-\frac14,$$ 
            we get
			\begingroup\allowdisplaybreaks\begin{align*} R_3(k,\ell)&=
				\frac{
					\Gamma(\ell/2+1/2)^{2}\Gamma(k+\ell/2+1)^{2}\Gamma(k-\ell/2+1)^{2}\Gamma(2k-\ell/2+3/2)^{2}(2\ell+1)}
				{\Gamma(\ell/2+1)^{2}\Gamma(k+\ell/2+3/2)^{2}\Gamma(k-\ell/2+1/2)^{2}\Gamma(2k-\ell/2+1)^{2}(4k-2\ell+1)
				}
				\\&\ge\frac{(\ell/2+3/4)(k+\ell/2+5/4)(k-\ell/2+1/4)(2k-\ell/2+3/4)(2\ell+1)}{(\ell/2+1/2)^2(k+\ell/2+1)^2(4k-2\ell+1)}=:U_3(k,\ell)
				,\\
				R_4(k,\ell)&=1=:U_4(k,\ell),\\
				R_5(k,\ell)&=\frac{\Gamma(\ell/2+1/2)^2\Gamma(k+\ell/2+3)^2\Gamma(k-\ell/2+2)^2\Gamma(2k-\ell/2+5/2)^2}{\Gamma(\ell/2+2)^2\Gamma(k+\ell/2+5/2)^2\Gamma(k-\ell/2+1/2)^2\Gamma(2k-\ell/2+3)^2}\\&\qquad \cdot\frac{(2\ell+3)(\ell+1)(\ell+2)}{(4k-2\ell+3)(2k-\ell+1)(2k-\ell+2)}
				\\&\ge\frac{(\ell/2+7/4)(k+\ell/2+9/4)(k-\ell/2+1/2)^2(k-\ell/2+5/4)(2k-\ell/2+11/4)}{(\ell/2+1/2)^2(\ell/2+3/2)^2(2k-\ell/2+5/2)^2}\\&\qquad\cdot\frac{(2\ell+3)(\ell+1)(\ell+2)}{(4k-2\ell+3)(2k-\ell+1)(2k-\ell+2)}:=U_5(k,\ell),\\
				R_6(k,\ell)&=\frac{(2k+\ell+6)^2(2k-\ell+3)(2k-\ell+1)(\ell+2)^2}{(\ell+1)(\ell+3)(4k-\ell+6)^2(2k-\ell+2)^2}:=U_6(k,\ell),\\
				R_7(k,\ell)&=\frac{\Gamma(\ell/2+1/2)^2\Gamma(k+\ell/2+5)^2\Gamma(k-\ell/2+3)^2\Gamma(2 k-\ell/2+7/2)^2}{\Gamma(\ell/2+3)^2\Gamma(k+\ell/2+7/2)^2\Gamma(k-\ell/2+1/2)^2\Gamma(2 k-\ell/2+5)^2}\\&\qquad\cdot\frac{(2\ell+5)(\ell+4)(\ell+3)(\ell+2)(\ell+1)}{(4 k-2\ell+5)(2 k-\ell+4)(2 k-\ell+3)(2 k-\ell+2)(2 k-\ell+1)}
				\\&\ge \frac{(\ell/2+11/4)(k+\ell/2+7/2)^2(k+\ell/2+17/4)(k-\ell/2+1/2)^2(k-\ell/2+3/2)^2}{(\ell/2+1/2)^2(\ell/2+3/2)^2(\ell/2+5/2)^2(2k-\ell/2+7/2)^2(2k-\ell/2+9/2)^2}\\&\qquad\cdot\frac{(k-\ell/2+9/4)(2k-\ell/2+19/4)(2\ell+5)(\ell+4)(\ell+3)(\ell+2)(\ell+1)}{(4 k-2\ell+5)(2 k-\ell+4)(2 k-\ell+3)(2 k-\ell+2)(2 k-\ell+1)}\\&:=U_7(k,\ell).
			\end{align*}\endgroup
			Let us consider $W_d(k,\ell):=U_d(k,\ell) A_d(k,\ell)+A_d(k,2k-\ell)$, where $$A_d(k,\ell):=3-\gamma^*+\frac{3(\gamma^*-1)\lambda(k)}{\lambda(k)-\lambda(\ell)}.$$
			With computer assistance (we used Mathematica), we can verify that for every $d=3,4,5,6,7$ $$W_d(k,\ell)> 0\quad\text{for every }k\ge3,\ 0\le \ell\le k-1.$$
            Therefore $$I_{k,d}\ge \sum_{\ell=0,2,\ldots,k-1} |a_{k,2k-\ell}|^2h_{2k-\ell,m}W_d(k,\ell)>0,$$ concluding the proof for $d=3,4,5,6,7$ and $k$ odd.
			\\ \vspace{-0.3cm}
			
			\noindent\textit{Case 4: $d\ge8, k$ odd satisfying $k<k_d$ and $\Phi_1$ zonal spherical harmonic.}
			Using the formula in \eqref{eq:def-I}, by direct computation we get
			\bea & I_{3,8} \approx 2.7 \cdot 10^5, \ I_{5,8} \approx 4.4 \cdot 10^6, \ I_{7,8} \approx 3.2 \cdot 10^7, \ I_{9,8} \approx 7.4 \cdot 10^7, \ I_{3,9} \approx 5.6 \cdot 10^5, \\& I_{5,9} \approx 8.5 \cdot 10^5, \ I_{3,10} \approx 8.7 \cdot 10^5, \ I_{3,11} \approx 9.7 \cdot 10^5, \ I_{3,12} \approx 5.8 \cdot 10^5,
            \eea
			which immediately concludes the proof.
				\\ \vspace{-0.3cm}
				
				\noindent\textit{Case 5. The case $d\ge3$, $k$ odd and $\Phi_1$ is a sectorial spherical harmonic.}
				Since $k$ is odd, as already observed in Case 3, $G_3\equiv0$ and we need to consider the term $G_4$.
				
                We first identify the elements in $K$ with their $k$-homogeneous extensions, and we take $\phi(x):=c_0\text{Re}((x_1+ix_2)^k)$, where $c_0>0$ is chosen in such a way that $\phi$ is normalized in $L^2(\partial B_1)$.\\
				For every $x\in\R^d$, we set $r:=|x|$ and $\rho:=\sqrt{x_1^2+x_2^2}$.
				We observe that given $m\in \N$ and $Q_{2m}$ a $2m$-homogeneous polynomial, we can write $Q_{2m}=\pi_{2m}(Q_{2m})+r^2 Q_{2m-2}$, where $\pi_{2m}(Q_{2m})$ is a $2m$-homogeneous harmonic polynomial and $Q_{2m-2}$ is a $(2m-2)$-homogeneous polynomial. Precisely, $\pi_{2m}(Q_{2m})$ is given by 
				$$
				\pi_{2m}(Q_{2m})(x)=\sum_{j=0}^{m}
				\frac{(-1)^j}{4^j j!\left(2m+d/2-1-j\right)_j}
				r^{2j}\Delta^jQ_{2m}(x),$$ where $(a)_n:=a(a+1)\cdots(a+n-1)$ is the Pochhammer symbol. 
				
				If $P$ is a $2k$-homogeneous polynomial, then we can write $$P(x)=\sum_{\ell=0}^{k}r^{2(k-\ell)} H_{2\ell}(x),$$ where each $H_{2\ell}$ is a $2\ell$-homogeneous harmonic polynomial. Precisely, applying first $\Delta^{k-\ell}$ and then the harmonic projection $\pi_{2\ell}$ to both sides of the identity above, we obtain
				\bea 
				H_{2\ell}(x)&=\frac{\pi_{2\ell}(\Delta^{k-\ell} P)(x)}{4^{k-\ell}(k-\ell)!(2\ell+d/2)_{k-\ell}},\eea since $\Delta^q(r^{2q}H_{2\ell})=4^qq!({2\ell}+d/2)_qH_{2\ell}$.
                Taking $P(x):=\frac12\rho^{2k}$, we get
				\be\label{eq:Hell-harmonic}
				H_{2\ell}(x)=
				\frac{(k!)^2}{2(k-\ell)!\left(2\ell+d/2\right)_{k-\ell}}
				\sum_{j=0}^{\ell}
				\frac{(-1)^j}{j!(\ell-j)!^2\left(2\ell+d/2-1-j\right)_j}
				r^{2j}\rho^{2(\ell-j)},
				\ee where we used the identity $$\Delta^m(\rho^{2k})= 4^m\frac{k!^2}{(k-m)!^2}\rho^{2(k-m)}$$ to compute both $\Delta^{k-\ell}(\rho^{2k})$ and $\Delta^{j}(\rho^{2\ell})$. Now, since $\mathrm{Re}(z)^2 = \frac12 ( \mathrm{Re}(z^2)+|z|^2)$ in $\C$, we have
				$$\phi^2(x) =c_0^2\sum_{\ell=0}^{k}r^{2(k-\ell)} P_{2\ell}(x), \quad\text{where} \quad\begin{cases}P_{2k}:=\varphi_{2k}+H_{2k}&\\
                P_{2\ell}:=H_{2\ell},&\text{if $\ell=0,\ldots,k-1$,}
                \end{cases}$$ and $\varphi_{2k}(x):=\frac12\text{Re}((x_1+ix_2)^{2k})$ and $H_{2\ell}$ are defined in \eqref{eq:Hell-harmonic}.
				
				We now observe that, since $H_{2\ell}$ is the harmonic projector of $\frac12\rho^{2k}$, then $$\int_{\partial B_1} H_{2\ell}^2\,d\HH^{d-1}=\frac12\int_{\partial B_1} H_{2\ell}\rho^{2k}\,d\HH^{d-1}.$$ Therefore, denoting by $\alpha_{d-1}:=\HH^{d-1}(\partial B_1)$, we can compute
				\be\label{eq:intpellk} \int_{\partial B_1} P_{2k}^2\,d\HH^{d-1}= \frac{1}{8} \frac{(2k)!}{(d/2)_{2k}}\alpha_{d-1} + \frac{d-2}{8} \left(2k + \frac{d}{2} - 1\right) \frac{(k!)^2}{(k + d/2 - 1)_{k+1}^2}\alpha_{d-1}=:A_{2k}+B_{2k},\ee and, for every $\ell=0,\ldots,k-1$
				\be\label{eq:intpell} \int_{\partial B_1} P_{2\ell}^2\,d\HH^{d-1}
				=\frac{d-2}{8} \left(2\ell + \frac{d}{2} - 1\right) \binom{k}{\ell}^2 \frac{(k!)^2}{(\ell + d/2 - 1)_{k+1}^2} \alpha_{d-1}:=B_{2\ell},
				\ee 
				where we used that $$\int_{\partial B_1} \varphi_{2k}H_{2k}\,d\HH^{d-1} =0\quad\text{and}\quad \int_{\partial B_1} \rho^{2m}\,d\HH^{d-1}=\alpha_{d-1} \frac{m!}{(d/2)_m}.$$
				Now, since $\psi$ is the unique solution to
				$\Delta_\theta \psi + \lambda(k) \psi = \frac{c_3}2\phi^2$ on $\partial B_1$, we get
				$$
				\psi= c_0^2\sum_{\ell=0}^k \frac{c_3}{2(\lambda(k)-\lambda(2\ell))}P_{2\ell}.
				$$
				Using \eqref{eq:intpellk} and \eqref{eq:intpell}, we have 
				\be\label{eq:g-step5}\begin{aligned}
				&G_4(e_1)=\frac1{24}\int_{\mathbb{S}^{d-1}} \left(c_4 \phi^4 + 3 c_3 \phi^2 \psi\right) \,d\mathcal{H}^{d-1} = \frac{c_0^4}{24}\sum_{\ell=0}^{k} \left(c_4 + \frac{3c_3^2}{2(\lambda(k)-\lambda(2\ell))} \right)
				\int_{\mathbb{S}^{d-1}} P_{2\ell}^2\,d\mathcal{H}^{d-1}\\&=\frac{c_0^4}{24}|c_3|(b^*)^{-1}\left(\left((3-\gamma^*)+\frac{3(\gamma^*-1)\lambda(k)}{\lambda(k)-\lambda(2k)}\right)A_{2k}+\sum_{\ell=0}^{k}\left((3-\gamma^*)+\frac{3(\gamma^*-1)\lambda(k)}{\lambda(k)-\lambda(2\ell)}\right)B_{2\ell}\right),\end{aligned}\ee
				where in the last equality we used again $c_4=|c_3|(b^*)^{-1}(3-\gamma^*)$ and \eqref{eq:computationc3}.
				
				The proof is concluded once we show that  \be\label{eq:first-step}J_{\ell}:=\frac{1}{\lambda(k)-\lambda(2\ell)}B_{2\ell}+\frac{1}{\lambda(k)-\lambda(2k-2\ell)}B_{2k-2\ell}\ge0,\ee
                for every $\ell=0,\ldots,(k-1)/2$. If this is the case, by substituting \eqref{eq:first-step} in \eqref{eq:g-step5}, we get 
				\bea G_4(e_1)&\ge  \frac{c_0^4}{24}|c_3|(b^*)^{-1}\left(\left((3-\gamma^*)+\frac{3(\gamma^*-1)\lambda(k)}{\lambda(k)-\lambda(2k)}\right)A_{2k}+(3-\gamma^*)\sum_{\ell=0}^{k}B_{2\ell}\right)\\&=\frac{c_0^4}{8}|c_3|(b^*)^{-1}\left((3-\gamma^*)+\frac{(\gamma^*-1)\lambda(k)}{\lambda(k)-\lambda(2k)}\right)A_{2k},
				\eea
				where in the last equality, we used that $$\sum_{\ell=0}^kB_{2\ell}=\sum_{\ell=0}^k\int_{\mathbb{S}^{d-1}}H_{2\ell}^2\,d\HH^{d-1}=\int_{\mathbb{S}^{d-1}}\left(\frac12\rho^{2k}\right)^2\,d\HH^{d-1}=2A_{2k}.$$ Finally, since $\gamma^*=\gamma_{k,d}$, we have $$G_4(e_1)\geq \frac{c_0^4}8|c_3|(b^*)^{-1}\frac{2k(k^2-2k+12+d(k-2))+8(d-2)}
				{(3k+d-2)(k^2+kd-2k-2d+4)}A_{2k}>0,$$ concluding the proof.
                
                We proceed by showing the validity of \eqref{eq:first-step} for every $\ell=0,\ldots,(k-1)/2$. Combining the definition of $B_{2\ell}$ with the fact that  
				$
				\lambda(k)-\lambda(2\ell)= (k-2\ell)(d-2+k+2\ell),
				$
				an explicit computation shows that
				\bea J_\ell&=\left(\frac{(d-2+4\ell)}{(d-2+k+2\ell)\left(d/2-1+\ell\right)_{k+1}^2}-\frac{(d-2+4k-4\ell)}{(d-2+3k-2\ell)\left(d/2-1+k-\ell\right)_{k+1}^2}\right)\\&\qquad \cdot\frac{(k!)^2\binom{k}{\ell}^2(d-2)\alpha_{d-1}}{16(k-2\ell)}.\eea
				Thus we only need to prove that
				\be\label{eq:ineq-harmonic}
				\frac{\left(d/2 - 1 + k - \ell\right)_{k+1}^{2}}{\left(d/2 - 1 + \ell\right)_{k+1}^{2}}\ge \frac{\left(d/2-1+k-\ell\right)^{2}}{\left(d/2-1+\ell\right)^{2}}
				\ge\frac{(d-2+4k-4\ell)(d-2+k+2\ell)}{(d-2+3k-2\ell)(d-2+4\ell)}.
				\ee
				The first inequality in \eqref{eq:ineq-harmonic} follows by the fact that $$ \frac{\left(d/2 - 1 + k - \ell\right)_{k+1}}
				{\left(d/2 - 1 + \ell\right)_{k+1}}=\prod_{j=0}^k\left(\frac{d/2-1+k-\ell+j}{d/2-1+\ell+j}\right)=\prod_{j=0}^k\left(1+\frac{k-2\ell}{d/2-1+\ell+j}\right),$$ then, since $k-2\ell\ge1$, all the factors in the product above are larger than $1$ and the bound follows by considering the term associated to $j=0$. For the second inequality in \eqref{eq:ineq-harmonic}, we observe that $$\frac{\left(d/2-1+k-\ell\right)^{2}}{\left(d/2-1+\ell\right)^{2}}
				-\frac{(d-2+4k-4\ell)(d-2+k+2\ell)}{(d-2+3k-2\ell)(d-2+4\ell)}=T(d,k,\ell),$$ where $$T(d,k,\ell):=\frac{2(k-2\ell)T_0(d,k,\ell)}
				{(d-2+2\ell)^2(d-2+4\ell)(d-2+3k-2\ell)}$$
				and
				$$
				T_0(d,k,\ell) := (d-2)^3+ 6k(d-2)^2 + (6k^2 +20k\ell-20\ell^2)(d-2) + 24 k\ell(k-\ell)
				$$
				Therefore $T(d,k,\ell)\ge0$ for every $\ell\le (k-1)/2$, which concludes the proof of \eqref{eq:first-step}.
			\end{proof}
	Finally, we can prove \cref{prop:optimality-log}. 
        \begin{proof}[Proof of \cref{prop:optimality-log}]
				It follows immediately by the characterization of the kernel proved in \cref{remark:after}, combined with \cref{prop:numerical}.
			\end{proof}

			\subsection{Bifurcations close to the radial cone}\label{subsec:bif}
            In the following subsections, we prove the bifurcation results for the radial cone in \cref{thm:bifurcation-radial}.
            First, we recall the following parameter-dependent version of the Lyapunov-Schmidt reduction in \cref{lemma:Lyapunov-Schmidt} for the radial cone $b_{\text{rad},\gamma}:=c_{\text{rad},\gamma}|x|^\beta,$ with $c_{\text{rad},\gamma}^{\gamma-2}:=\frac2\gamma \lambda(\beta)$.

        \begin{proposition}\label{lemma:Lyapunov-Schmidt-bifurcation}
			Let $d\geq 2$, $k\in \N_{\geq 3}$, $\gamma ^*:=\gamma_{k,d}$, $b^*:=b_{\text{rad},\gamma^*}\in\mathcal{B}_0$ and $K:=\ker(L_{b^*})$. Then there exist $\eps>0$, a neighborhood $U\subset K$ of $0$ in $C^{1,\alpha}(\partial B_1)$ and an analytic map $$Y:(K\cap U)\times I_\eps\to K^\perp\subset H^1(\partial B_1),$$
            with $I_\eps:=(\gamma^*-\eps,\gamma^*+\eps)$,  such that the following holds.
            \begin{itemize}
            \item $Y(0,\gamma^*)=0$, $\delta Y(0,\gamma^*)=0$ and $Y(0,\gamma)=0$ for $\gamma \in I_\eps$. Moreover,
            $$
			P_{K^\perp}\big(\delta\mathcal{G}(\phi+Y(\phi,\gamma),\gamma)\big)=0,\quad \mbox{for every }(\phi,\gamma) \in (K\cap U)\times I_\eps.
            $$
            \item Let $\Phi_1,\dots,\Phi_N$ be an orthonormal basis of $K$. Then, there exists $\rho>0$ such that, for every $\mu\in B_\rho\subset\R^N$, the reduced functional $G: B_\rho \times I_\eps \to \R$ is defined as
\bea\label{def:G-finito-dim-bif}
G(\mu,\gamma):=\mathcal{G}(\Phi_\mu +Y\left(\Phi_\mu,\gamma\right),\gamma),\quad \mbox{with}\quad\Phi_\mu:=\sum_{i=1}^N\mu_i\Phi_{i},
\eea
            and we have that $\overline \varphi \in H^1(\partial B_1)$ is a critical point of $\mathcal{G}(\cdot,\overline \gamma)$ for some $\overline \gamma \in I_\eps$, if and only if there exists $\overline \mu\in\R^N$ such that 
            $$
            \overline \varphi = \Phi_{\overline \mu} + Y(\Phi_{\overline \mu}, \overline \gamma) \quad \mbox{and}\quad \nabla_\mu G(\overline \mu,\overline \gamma)=0.
            $$
			\end{itemize}
		\end{proposition}
Since, at $\gamma=\gamma_{k,d}$, the kernel $K$ coincides with the space of spherical harmonics $\mathcal{H}_k(\mathbb{S}^{d-1})$, see \cref{prop:optimality-log}, it is useful to formulate a symmetry-restricted version of the Lyapunov-Schmidt reduction, adapted to invariant subspaces under suitable group actions.
            \begin{corollary}\label{lemma:Lyapunov-Schmidt-bifurcation-sym}
Assume the hypotheses of \cref{lemma:Lyapunov-Schmidt-bifurcation}. Let $\Sigma$ be a symmetry subgroup such that
$$
\mathcal{G}(\varphi\circ \sigma,\gamma) = \mathcal{G}(\varphi,\gamma)\quad \mbox{for every }\varphi \in H^1(\partial B_1), \ \gamma \in (1,2),\  \sigma \in \Sigma. 
$$
Let $(H^1(\partial B_1))^\Sigma:=\{\varphi\in H^1(\partial B_1): \varphi\circ \sigma = \varphi\,\mbox{ for all }\sigma \in \Sigma\}$ and let $K^\Sigma$ be the restricted kernel
\[
K^\Sigma:=K\cap (H^1(\partial B_1))^\Sigma.
\]
Then, the conclusion of \cref{lemma:Lyapunov-Schmidt-bifurcation} remains valid with $K$ and $H^1(\partial B_1)$ replaced by $K^\Sigma$ and $(H^1(\partial B_1))^\Sigma$.

Moreover, if $K^\Sigma$ is one-dimensional, i.e., $K^\Sigma=\text{span}\{\Phi_1\}$,
then there exists a branch of nontrivial $\Sigma$-symmetric critical points $\varphi_\gamma\in H^1(\partial B_1)$ of $\mathcal{G}(\cdot,\gamma)$, defined for $\gamma$ in a one-sided neighborhood of $\gamma^*$, tangent to $\Phi_1$ at $\gamma=\gamma^*$. Precisely, for such values of $\gamma$, the critical point $\varphi_\gamma$ is of the form
$$
\varphi_\gamma = \mu(\gamma)\Phi_1 + Y(\mu(\gamma)\Phi_1,\gamma),\quad \text{with }\mu(\gamma)\to 0,\,\text{as }\gamma \to \gamma^*.
$$
\end{corollary}
\begin{proof}
The first part is a straightforward generalization of \cref{lemma:Lyapunov-Schmidt-bifurcation}. On the other hand, let $\Phi_1,\dots,\Phi_N$ be an orthonormal basis of $K$, since
$$\partial_\gamma \delta^2\mathcal{G}(0,\gamma^*)[\Phi_i]=\partial_\gamma c(\gamma^*)\Phi_i,\quad\text{with}\quad c(\gamma)=\gamma(\gamma-1)c_{\text{rad},\gamma}^{\gamma-2}-2\lambda(\beta)=2(\gamma-2)\lambda(\beta),$$ we deduce that 
\be\label{e:transversality-cra-rab}
\partial_\gamma \delta^2\mathcal{G}(0,\gamma^*)[\Phi_i]\in K, \qquad \text{for every $i=1,\dots,N$}.
\ee
In particular, if $K^\Sigma$ is one-dimensional, the condition \eqref{e:transversality-cra-rab} with $i=1$  coincides with the transversality condition of the linearized operator $L_{b_{\text{rad}}}$ at $\gamma=\gamma^*$. Thus, the second part of the result is a direct consequence of the Crandall-Rabinowitz bifurcation theorem (see for instance \cite[Section I.5]{Kielhofer2012}). 
\end{proof}

            \subsection{Expansion of the reduced bifurcation equation}\label{subsec:expansionbif}
            In this subsection, we derive the leading-order expansion of the reduced bifurcation equation at the resonant values
            $\gamma^*:=\gamma_{k,d}$ close to the radial cone $b^* := b_{\text{rad},{\gamma^*}}$ 
            This expansion will then serve as the starting point for the case-by-case analysis of the symmetry classes and dimensions considered in \cref{thm:bifurcation-radial}.

As already observed, we have
            $$
            K=\mathcal{H}_k(\mathbb{S}^{d-1}),\quad  
            (b^*)^{\gamma-2} = \frac{2}{\gamma(2-\gamma^*)}\lambda(k)\quad\text{and}\quad L_{b^\ast} = -\Delta_\theta - \lambda(k).
            $$
            We stress that, for $\eps$ small enough, by \cref{prop:optimality-log}, $\ker (L_{b_{\text{rad},\gamma}})=\{0\}$ for every $\gamma\in I_\eps\setminus \{\gamma^*\}$, where $$b_{\text{rad},\gamma}:=\left(\frac2\gamma\lambda(\beta)\right)^{\frac{1}{\gamma-2}}$$ on the sphere.
            
            By following the notation of \cref{lemma:Lyapunov-Schmidt-bifurcation} and \cref{lemma:Lyapunov-Schmidt-bifurcation-sym}, let $K^\Sigma$ be a reduced kernel (possibly coinciding with the kernel $K$), $\Phi_1,\ldots,\Phi_M$ be an orthonormal basis of $K^\Sigma$, with $M\leq N$. Let $Y$ be the Lyapunov-Schmidt map of \cref{lemma:Lyapunov-Schmidt-bifurcation-sym} and $G:B_\rho \times I_\eps \to \R$ be the corresponding reduced functional. Since $Y(0,\cdot)=0$ in $I_\eps$, we have 
            $$
            Y(0,\gamma^*)=\partial_\gamma Y(0,\gamma^*)=\partial_\gamma^2 Y(0,\gamma^*)=0,\quad \text{and}\quad \delta^2 Y(0,\gamma^*)[\Phi_i]=0,\quad \mbox{for every }i=1,\dots,M.
            $$
            On the other hand, by differentiating 
            \bea\label{e:proj}P_{(K^\Sigma)^\perp}\delta\mathcal{G}(\Phi_\mu+Y(\Phi_\mu,\gamma),\gamma)=0\eea
 with respect to $\mu \in\R^N $ and $ \gamma \in \R$, we infer that at $(0,\gamma^*)$
$$
P_{(K^\Sigma)^\perp}\left(\delta^2 \mathcal{G}(0,\gamma^*)[\partial_\gamma \delta Y(0,\gamma^*)[\Phi_i]]+\partial_\gamma \delta^2 \mathcal{G}(0,\gamma^*)[\Phi_i]\right)=0
$$
for every $i=1,\dots,M$. Thus
$$
\partial_\gamma \delta Y(0,\gamma^*)[\Phi_i] = -\frac12(L_{b}|_{(K^\Sigma)^\perp})^{-1}P_{(K^\Sigma)^\perp}\partial_\gamma\delta^2\mathcal{G}(0,\gamma^*)[\Phi_i]
$$
and by \eqref{e:transversality-cra-rab} we deduce that $\partial_\gamma \delta Y(0,\gamma^*)[\Phi_i]=0$, for every $i=1,\dots,M$. 
Then, we have
$$
Y(\Phi_\mu ,\gamma) = \frac12 \sum_{1\leq i,j\leq N}\delta^2 Y(0,\gamma^*)[\Phi_i,\Phi_j]\mu_i \mu_j + o(|(\mu,\gamma-\gamma^*)|^2).
$$
We notice that on one hand, the existence of a trivial branch for every $\gamma \in I_\eps$ is equivalent to  $\partial_{\mu_i}G(0,\gamma)=0$ for $i=1,\dots,M$, for every $\gamma \in I_\eps$. Therefore, $$\partial_\gamma\partial_{\mu_i}G(0,\gamma^\ast)=\partial_\gamma^2\partial_{\mu_i}G(0,\gamma^\ast)=0\quad\text{for every }i=1,\dots,M.$$ On the other hand, we infer that  
$$
\partial_{\mu_i} G(0,\gamma^*)=\partial^2_{\mu_i,\mu_j} G(0,\gamma^*) = 0 
\quad\text{for every }1\leq i,j\leq M.$$
By collecting the previous identities, we obtain that, for every $i=1,\dots,M$, the following expansion holds true 
\begin{align*}
\nabla_\mu G(\mu,\gamma)&= \frac12 \nabla^3_\mu G(0,\gamma^*)[\mu,\mu] + (\gamma-\gamma^*)  \nabla^2_\mu \partial_\gamma G(0,\gamma^*)[\mu]+   \frac16 \nabla^4_\mu G(0,\gamma^*)[\mu,\mu,\mu] \\
&\qquad + o(|\mu|^3 + |\mu||\gamma-\gamma^*|).
\end{align*}
Since $b$ is not integrable for $\gamma=\gamma^*$, then $G(\cdot,\gamma^*)\not\equiv 0$ (see \cite[Lemma 1]{adamsimon}). Then, since $|\nabla_\mu G(0,\gamma^*)|=|\nabla^2_\mu G(0,\gamma^*)|=0$, there exists an integer $p\geq 3$ such that the expansion of the function $G$ in \eqref{eq:exp-G} holds true for $\gamma=\gamma^*$, where each $\mu \mapsto G_j(\mu,\gamma^*)$ is a $j$-homogeneous polynomial and $G_p\not\equiv0$ is the first nontrivial term. Thus, we have
\begin{align*}
\nabla^2_\mu\partial_\gamma G(0, \gamma^*) &= \left(\partial_\gamma \delta^2\mathcal{G}(0,\gamma^*)[\Phi_i,\Phi_j]\right)_{ij} = \partial_\gamma c(\gamma^*) \mathrm{Id}_{\R^M},\\
\frac12\nabla^3_{\mu}G(0,\gamma^*)[\mu,\mu]&= \nabla_\mu G_3(\mu,\gamma^*), \qquad \frac{1}{6}\nabla^4_{\mu}G(0,\gamma^*)[\mu,\mu,\mu]=\nabla_\mu G_4(\mu,\gamma^*)
\end{align*}
and since $\partial_\gamma c(\gamma^*)=-2(\beta^*)^2$, for $(\mu,\gamma) \in B_\rho\times I_\eps$, the bifurcation equation takes the form
\be\label{e:expansion-bifurcation}
\nabla_\mu G_3(\mu,\gamma^*) -2(\gamma-\gamma^*)(\beta^*)^2  \mu +   \nabla_\mu G_4(\mu,\gamma^*) + o(|\mu|^3 + |\mu||\gamma-\gamma^*|)=0.
\ee
\subsection{Construction of the bifurcating branches}
We now apply the expansion of the reduced bifurcation equation in \eqref{e:expansion-bifurcation} to construct and classify the bifurcating branches arising in the symmetry classes and dimensions considered in \cref{thm:bifurcation-radial}.

\begin{proof}[Proof of \cref{thm:bifurcation-radial}]
We follow the notation introduced in \cref{subsec:bif} and \cref{subsec:expansionbif} and we strongly rely on the computations in the proof of  \cref{prop:numerical}. The proof is divided into several cases.
   \medskip
				
\noindent\textit{Case 1: $d=2$.} Then $N=2$ and $\Phi_1:=\frac{1}{\sqrt{\pi}}\cos(k\theta)$, $\Phi_2:=\frac{1}{\sqrt{\pi}}\sin(k\theta)$. 
By Case 1 of the proof of \cref{prop:numerical}, we already know that
$$
G_3(\cdot,\gamma_{k,2})\equiv 0\quad\text{for every $k\in \N_{\geq 3}$},\qquad 
G_4(e_1,\gamma_{k,2})>0.
$$
In this case, since the kernel $K$ is $O(2)$-equivariant, we can restrict to the symmetry-reduced subspace 
$$
K^\Sigma:=K\cap (H^1(\partial B_1))^\Sigma,\quad \text{with}\quad
(H^1(\partial B_1))^\Sigma:=\left\{\varphi \in H^1(\partial B_1)\colon \varphi(\theta) = \varphi(-\theta) \right\},
$$
and apply \cref{lemma:Lyapunov-Schmidt-bifurcation-sym}. Precisely, since $K^\Sigma = \text{span}\{\Phi_1\}$,
 the reduced bifurcation equation \eqref{e:expansion-bifurcation} takes the form 
$$
 -2(\gamma-\gamma^*)(\beta^*)^2  \mu +   4  G_4(e_1,\gamma^*)\mu^3 + o(|\mu|^3 + |\mu||\gamma-\gamma^*|)=0,
$$
for $(\mu,\gamma) \in (-\rho,\rho)\times I_\eps$. Thus
$$
 -2(\gamma-\gamma^*)(\beta^*)^2  +   4  G_4(e_1,\gamma^*)\mu^2 +o(|\mu|^2 + |\gamma-\gamma^*|)=0,
$$
and since $G_4(e_1,\gamma^*)>0$, there exists a branch of nontrivial critical points of the form $\varphi=\mu_\pm \Phi_1 + Y(\mu_\pm \Phi_1,\gamma)$, where $\mu_\pm: I_\eps \to (-\rho,\rho)$ is such that 
$$
\mu_\pm(\gamma) = \pm\frac{\beta^*}{\sqrt{2 G_4(e_1,\gamma^*)}}(\gamma-\gamma^*)_+^{1/2}
+o(|\gamma-\gamma^*|^{1/2})\quad \mbox{for }\gamma \in I_\eps.
$$
\smallskip
				
\noindent\textit{Case 2: $d\geq 3$, and $\Sigma=\Sigma_{\text{zon}}$ is the zonal symmetry subgroup.} Let $\Sigma_{\text{zon}}:= O(d-1)\subset O(d)$
be the subgroup that fixes the $x_d$-axis. Then 
$$
K^{\Sigma_{\text{zon}}} := K\cap (H^1(\partial B_1))^{\Sigma_{\text{zon}}} =  \text{span}\{\Phi_1\}\quad \text{and}\quad \Phi_1:=h_{k,m}^{-1/2}C_k^m(\cos\theta_d),
$$
where $C_k^m$ is the Gegenbauer polynomial with $m:=\frac{d-2}{2}$ and $h_{k,m}$ is chosen in such a way that $\Phi_1$ is normalized in $L^2(\partial B_1)$. Then, 
by combining Case 2-3-4 of the proof of \cref{prop:numerical} with
\cref{lemma:Lyapunov-Schmidt-bifurcation-sym}, we can study the bifurcation of nontrivial critical points as follows.
\begin{itemize}
    \item If $k$ is even, by Case 2 we have that $G_3(e_1,\gamma^*)>0$ and the reduced bifurcation equation \eqref{e:expansion-bifurcation} takes the form
    $$
    3 G_3(e_1,\gamma^*)\mu^2 -2(\gamma-\gamma^*)(\beta^*)^2  \mu + o(|\mu|^2+|\mu||\gamma-\gamma^\ast|)=0,
    $$
    with $(\mu,\gamma) \in (-\rho,\rho)\times I_\eps$. Since $G_3(e_1,\gamma^*)>0$, there exists a branch of nontrivial critical points of the form $\varphi=\mu \Phi_1 + Y(\mu \Phi_1,\gamma)$, with $\mu: I_\eps \to (-\rho,\rho)$ such that 
$$
\mu(\gamma) = \frac{2(\beta^*)^2}{3G_3(e_1,\gamma^*)}(\gamma-\gamma^*)
+o(|\gamma-\gamma^*|)\qquad \mbox{for }\gamma \in I_\eps.
$$
    \item If $k$ is odd, and either $d=3,4,5,6,7$ or $d\geq 8$ and $k<k_d$, then by Case 3 and Case 4
    $$
    G_3(\cdot ,\gamma^*)\equiv 0\quad\text{and}\quad G_4(e_1,\gamma^*)>0
    $$
    and the reduced bifurcation equation \eqref{e:expansion-bifurcation} takes the form
    $$
    -2(\gamma-\gamma^*)(\beta^*)^2  \mu +   4 G_4(e_1,\gamma^*)\mu^3 + o(|\mu|^3 + |\mu||\gamma-\gamma^*|)=0
    $$
    with $(\mu,\gamma) \in (-\rho,\rho)\times I_\eps$. Since $G_4(e_1,\gamma^*)>0$, there exists a branch of nontrivial critical points of the form $\varphi=\mu_\pm \Phi_1 + Y(\mu_\pm \Phi_1,\gamma)$, with $\mu_\pm: I_\eps \to (-\rho,\rho)$ such that 
$$
\mu_\pm(\gamma) = \pm \frac{\beta^*}{\sqrt{2 G_4(e_1,\gamma^*)}}(\gamma-\gamma^*)_+^{1/2}
+o(|\gamma-\gamma^*|^{1/2})\qquad \mbox{for }\gamma \in I_\eps.
$$
\end{itemize}

\smallskip
				
\noindent\textit{Case 3: $d\geq 3$, $k$ odd and $\Sigma=\Sigma_{\text{sec}}$ is the sectorial symmetry group.} Let $x=(x_1,x_2,x')\in\R\times\R\times \R^{d-2}$ and consider $\Sigma_{\text{sec}}:= D_k \times O(d-2)\subset O(d)$, where $O(d-2)$ acts on the $x'$-variables and $D_k\subset O(2)$ is generated by the rotations of angle $2\pi/k$ in the variables $(x_1,x_2)\in\R^2$, i.e., 
$$
R_{2\pi/k}(x_1,x_2) = \left(x_1\cos(2\pi/k) - x_2 \sin(2\pi/k), x_1\sin(2\pi/k) + x_2\cos(2\pi/k)\right),
$$
and by the reflection in the $x_2$-variable, i.e., $\sigma(x_1,x_2,x')=(x_1,-x_2,x')$. Then, for $k$ odd,
$$
K^{\Sigma_{\text{sec}}} := K\cap (H^1(\partial B_1))^{\Sigma_{\text{sec}}} =  \text{span}\{\Phi_1\}\quad \text{and}\quad\Phi_1:=c_0\text{Re}((x_1+ix_2)^k), 
$$
where $c_0>0$ is chosen in such a way that $\Phi_1$ is normalized in $L^2(\partial B_1)$.
By Case 5 of the proof of \cref{prop:numerical}, we have
$$
G_3(\cdot ,\gamma^*)\equiv 0\quad\text{and}\quad G_4(e_1,\gamma^*)>0,
$$
and by \cref{lemma:Lyapunov-Schmidt-bifurcation-sym}, the reduced bifurcation equation \eqref{e:expansion-bifurcation} takes the form
    $$
    -2(\gamma-\gamma^*)(\beta^*)^2  \mu +   4 G_4(e_1,\gamma^*)\mu^3 + o(|\mu|^3 + |\mu||\gamma-\gamma^*|)=0
    $$
    with $(\mu,\gamma) \in (-\rho,\rho)\times I_\eps$. Since $G_4(e_1,\gamma^*)>0$, there exists a branch of nontrivial critical points of the form $\varphi=\mu_\pm \Phi_1 + Y(\mu_\pm \Phi_1,\gamma)$, with $\mu_\pm: I_\eps \to (-\rho,\rho)$ such that 
$$
\mu_\pm(\gamma) = \pm \frac{\beta^*}{\sqrt{2 G_4(e_1,\gamma^*)}}(\gamma-\gamma^*)_+^{1/2}
+o(|\gamma-\gamma^*|^{1/2})\qquad \mbox{for }\gamma \in I_\eps.
$$

\smallskip
				
\noindent\textit{Case 4: $d\geq 3$, $k$ even and $\Sigma=\Sigma_{\text{sec}}$ is the sectorial symmetry group.} As before, let $x=(x_1,x_2,x')\in\R\times\R\times \R^{d-2}$ and consider $\Sigma_{\text{sec}}:= D_k \times O(d-2)\subset O(d)$, with $D_k:=\langle R_{2\pi/k},\sigma \rangle\subset O(2)$. 

Then, for $k$ even, $K^{\Sigma_{\text{sec}}} = \text{span}\{\Phi_1,\Phi_2\},$ with
$$
\Phi_1:=c_0\text{Re}((x_1+ix_2)^k),\qquad \Phi_2:=c_1|x|^{k} P_{k/2}^{(m-1,0)}\left(\frac{2(x_1^2+x_2^2)}{|x|^2}-1\right),
$$
where $m:=\frac{d-2}{2}$, $P_{k/2}^{m-1,0}$ is a Jacobi polynomial of degree $k/2$, and $c_0$, $c_1>0$ are chosen in such a way that $\Phi_1$, $\Phi_2$ are normalized in $L^2(\partial B_1)$. Notice that, if $d=3$, the function $\Phi_2$ coincides with the generator of $K^{\Sigma_{\text{zon}}}$ in Case 2. Let us assume there exists a branch of nontrivial critical points tangent to the sectorial symmetric solution $\Phi_1$. Then, for $\gamma \to \gamma^*$ we would have $\mu_2=o(\mu_1)$, and so, for $i=1,2$
$$
\partial_{\mu_i} G_j(\mu,\gamma^*) = \partial_{\mu_i} G_j \left(\left(1,\frac{\mu_2}{\mu_1}\right),\gamma^\ast\right)\mu_1^{j-1} = \partial_{\mu_i} G_j (e_1,\gamma^*) \mu_1^{j-1}+o(\mu_1^{j-1})
$$
where we used that $O(\mu_1^{j-2}|\mu_2|)=o(\mu_1^{j-1})$. 
Therefore, the first equation in 
the reduced bifurcation system \eqref{e:expansion-bifurcation} takes the form
$$
\partial_{\mu_1} G_3 (e_1,\gamma^*)  \mu_1^{2} -2(\gamma-\gamma^*)(\beta^*)^2  \mu_1  + o(|\mu_1|^2+|\mu_1||\gamma-\gamma^\ast|)=0, 
$$
and since by Case 5 of the proof of \cref{prop:numerical}, we know that $\partial_{\mu_1}G_3(e_1,\gamma^*)=0$, we infer
\be\label{e:olitte}
\gamma-\gamma^* =  o(|\mu_1|).
\ee
Similarly, if we consider the second equation in  \eqref{e:expansion-bifurcation}, by \eqref{e:olitte} we have
$$
\partial_{\mu_2} G_3(e_1,\gamma^*)\mu_1^2  =  2(\gamma-\gamma^*)(\beta^*)^2  \mu_2+ o(|\mu_1|^2 + |\mu_1||\gamma-\gamma^*|) = 
o(|\mu_1|^2),
$$ since $\mu_2=o(\mu_1)$.
This implies that $\partial_{\mu_2}G_3(e_1,\gamma^*) =0$. However, by direct computation, we infer that 
$$
\partial_{\mu_2}G_3(e_1,\gamma^*) = \frac{c_3}{2}\int_{\partial B_1}\Phi_1^2\Phi_2\,d\HH^{d-1}=
\frac{c_3}{2}c_0^2c_1\pi^{d/2}\binom{k}{k/2}\frac{\Gamma(k+1)\Gamma(k/2+m)}{\Gamma(m)\Gamma(3k/2+m+1)},
$$
where $c_3:=\gamma^*(\gamma^*-1)(\gamma^*-2)(b^*)^{\gamma^*-3}$. The left-hand side above is different from zero and leads to a contradiction.
\end{proof}
\section{Minimality of the radial cone}\label{section:minimality-radial-cone}
            In this section we prove the minimality result for the radial cone $b_{\text{rad}}:=c_{\text{rad}}|x|^\beta$, stated in \cref{thm:minimality-of-radial-cone}. 
            Throughout the section, we assume that $d\ge3$ and $\gamma\in(0,1)$. Indeed, 
            the case $\gamma \in [1,2)$ is already covered by \cref{prop:convex-functional}, where $b_{\mathrm{rad}}$ is shown to be a minimizer, while in dimension $d=2$ the radial cone is not minimizing for $\gamma\in(0,1)$, as shown in \cref{prop:no-singular-points-dim-2}. 

We recall the definition of $\gamma_\Delta(d)$ in \eqref{eq:def-gamma-delta}. By the analysis of Savin-Yu \cite{savinyu-altphillips-stable}, for $\gamma\ge \gamma_\Delta(d)$, the minimality of $b_{\mathrm{rad}}$ is equivalent to its one-sided minimality from below (see  \cref{def:minimo-alto-basso}). Therefore, in order to prove the minimality of the radial cone, it is enough to establish this weaker minimality property. To this aim, by exploiting a spherical rearrangement, we reduce the analysis to a one-dimensional minimization problem.

In the following, competitors with infinite one-dimensional energy are included in the admissible classes. The only possible loss of integrability occurs in the derivative terms near $0$. However, this does not affect the minimality conditions, since in that case the desired inequalities are trivially satisfied. Therefore, we consider competitors in $AC_{\mathrm{loc}}(0,1]$, i.e., the class of one-dimensional functions that are absolutely continuous on $[\eps,1]$, for every $\eps \in (0,1)$.

			\begin{lemma}\label{lemma:radial-one-dimensional}
			    Let $d\geq 3$ and $\gamma \in (0,1)$. Then, the radial cone $b_{\text{rad}}$ is a one-sided minimizer from below, if and only if the one-dimensional profile $b(r):=c_{\text{rad}}\,r^\beta$ is a one-sided minimizer from below of  $$J(v):=\int_0^1 \left(|v'|^2+v^\gamma\right)r^{d-1}\,dr.$$ More precisely, this means that $J(b)\le J(v)$ for every $v\in AC_{\text{loc}}(0,1]$ such that $v(1)=c_{\text{rad}}$ and $0\le v(r)\le b(r)$, for every $r\in(0,1)$.
			\end{lemma}
            \begin{proof}
                For the sake of readability, we set $c:=c_{\text{rad}}$.
				For every function $v\in H^1(B_1)$ such that $v=c$ on $\partial B_1$ and $0\le v\le b_{\text{rad}}\le c$ in $B_1$, we define $w:=c-v\ge0$. Now, let $w^\ast$ be the decreasing spherical rearrangement of $w$, and consider  $v^\ast:=c-w^\ast$. On one hand, we have $v^\ast=c$ on $\partial B_1$.
                On the other hand, since $w\ge c-b_{\text{rad}}=c(1-|x|^\beta)$, we infer that $w^\ast\ge c-b_{\text{rad}}$, namely $v^\ast\le b_{\text{rad}}$. Similarly, we deduce that $v^\ast\ge0$.
                By a standard spherical rearrangement argument, we also have 
                $$\int_{B_1}|\nabla v^\ast|^2\,dx\le \int_{B_1}|\nabla v|^2\,dx\quad\text{and}\quad \int_{B_1}(v^\ast)^\gamma\,dx= \int_{B_1}v^\gamma\,dx,$$
                which implies that  $\mathcal{J}_\gamma(v^\ast)\le \mathcal{J}_\gamma(v)$. Therefore, it is not restrictive to consider radial competitors in the minimization problem for $\mathcal{J}_\gamma$. Finally, the conclusion follows by computing $\mathcal{J}_\gamma$ for radial functions.
            \end{proof}
   
            \begin{proposition}\label{prop:minimality-radial}
                Let $d \geq 3$ and $\gamma \in (0,1)$. For every $h\in AC_{\text{loc}}(0,1]$, we set \be\label{eq:def-mathcal-L}
                \mathcal{L}(h):=\int_0^1\Big(A(d,\gamma)t^2(\partial_th)^2 +\Psi(h)\Big)\,dt,
                \ee
                where \bea\Psi(h):=(1-h)^\gamma-1+\gamma h-\frac\gamma2h^2\quad\text{and} \quad A(d,\gamma):=
				\frac{\gamma(d+2\beta-2)^2}{2\lambda(\beta)}.\eea
                Then $b_{\text{rad}}$ is a one-sided minimizer from below if and only if $\mathcal{L}(h)\ge0$ for every $h\in AC_{\text{loc}}(0,1]$ satisfying $h(1)=0$ and $0\le h\le 1$. 
            \end{proposition}
            \begin{proof}
			  In order to study the minimality from below of $b_{\text{rad}}$, we use \cref{lemma:radial-one-dimensional}. Hence, let $c:=c_{\text{rad}}$ and $v\in AC_{\text{loc}}(0,1]$ such that $0\le v(r)\le b(r):=cr^\beta$ and $v(1)=c$. Then, we write $v=(1-h)b$, where $h\in AC_{\text{loc}}(0,1]$ satisfies $h(1)=0$ and $0\le h\le 1$. By direct computation $$|v'|^2=(-h'b+(1-h)b')^2=(h')^2b^2+(1-h)^2(b')^2+((1-h)^2)'bb'$$
                Integrating by parts the last term, and using that $(r^{d-1}b')'=\frac\gamma2 b^{\gamma-1}r^{d-1}$ we get 
                $$\int_0^1((1-h)^2)'bb'r^{d-1}\,dr=\beta c^2-\int_0^1(1-h)^2\Big((b')^2+\frac\gamma2 b^{\gamma}\Big)r^{d-1}\,dr.$$
                Since $v^\gamma-b^\gamma=((1-h)^\gamma-1) b^\gamma$, we get $$J(v)-J(b)=\beta c^2+\int_0^1\left((h')^2b^2-(b')^2+\left((1-h)^\gamma-1-\frac\gamma2(1-h)^2\right) b^\gamma\right)r^{d-1}\,dr.$$  
                Integrating again by parts and using the equation of $b$ as above, we obtain $$\int_0^1 (b')^2r^{d-1}\,dr=\beta c^2-\int_0^1\frac\gamma2b^\gamma r^{d-1}\,dr.$$
                Therefore, we have $$J(v)-J(b)=\int_0^1\Big((h')^2b^2+\Psi(h)b^\gamma\Big)r^{d-1}\,dr=c^\gamma\int_0^1 \left(\frac{\gamma}{2\lambda(\beta)}r^2(h')^2+\Psi(h)\right)r^{d+2\beta-3}\,dr,$$ where we used that $b=cr^\beta$ and $c^{2-\gamma}=\frac{\gamma}{2\lambda(\beta)}$. Now, set $t:=r^{d+2\beta-2}$ and $\widetilde h(t):=h(r)$, then $$dt=(d+2\beta-2)r^{d+2\beta-3}\,dr\quad\text{and}\quad\partial_t\widetilde h=h'(r) \frac{r}{(d+2\beta-2)t}.$$
                Then, we obtain
                \bea J(v)-J(b)=\frac{c^\gamma}{d+2\beta-2}\int_0^1\Big(A(d,\gamma)t^2(\partial_t\widetilde h)^2 +\Psi(\widetilde h)\Big)\,dt=\frac{c^\gamma}{d+2\beta-2}\mathcal{L}(\widetilde h),\eea concluding the proof since $\widetilde h(1)=0$ and $0\le \widetilde h\le 1$.     
			\end{proof}
            
            \begin{corollary}\label{lemma:minimality1}
				Let $d\geq 3$ and $\gamma \in (0,1)$, we consider $\Psi(h)$ and $A(d,\gamma)$ as in \cref{prop:minimality-radial}. Suppose there exists a calibration $C\in C^1([0,1])$, such that \be\label{eq:calibration}C(0)=0\quad\text{and}\quad C(h)+\Psi(h)\ge \frac{(C'(h))^2}{4A(d,\gamma)}\quad\text{for every }h\in[0,1].\ee Then $b_{\text{rad}}$ is a one-sided minimizer from below.
			\end{corollary}
    \begin{proof}
        The result is an application of \cref{prop:minimality-radial}. First, given $h\in AC_{\text{loc}}(0,1]$ satisfying $h(1)=0$ and $0\le h\le 1$, we take $H(\tau):=h(e^{-\tau})$. Then, by applying the change of coordinates $t=e^{-\tau}$, we get that \eqref{eq:def-mathcal-L} can be rewritten as
        \be\label{e:inequ-H} 
        \mathcal{L}(h)=\int_0^\infty e^{-\tau}\Big(A(d,\gamma)(\partial_\tau H)^2 +\Psi(H)\Big)\,d\tau.
        \ee
        Now, let $C(h)$ be a calibration satisfying \eqref{eq:calibration}, then Young's inequality implies that $$e^{-\tau}\Big(A(d,\gamma) (\partial_\tau H)^2+\Psi(H)\Big)\ge e^{-\tau}\Big(C'(H)\partial_\tau H-C(H)\Big)=\frac{d}{d\tau} \Big(e^{-\tau}C(H)\Big).$$ 
        On one hand, since $H(0)=h(1)=0$ and $C(0)=0$, we get $e^{-\tau}C(H(\tau))\to0$ as $\tau\to0^+$. On the other hand, since $H(t)\in[0,1]$ and $C$ is bounded, we have $e^{-\tau}C(H(\tau))\to0$ as $\tau\to+\infty$. Thus, by integrating \eqref{e:inequ-H} on $(0,\infty)$ and applying the last inequality, we get
        $$\mathcal{L}(h) \geq \lim_{\tau \to +\infty}e^{-\tau} C(H(\tau))-\lim_{\tau \to 0^+} e^{-\tau} C(H(\tau)) = 0,$$
        concluding the proof.
        \end{proof}
    In the next lemma we construct two explicit calibrations satisfying \eqref{eq:calibration}. These calibrations will be crucial in the study of the minimality of the radial cone in \cref{thm:minimality-of-radial-cone}.         
\begin{lemma}\label{lemma:minimality2}
Let $d\geq 3$, $\gamma \in (0,1)$ and $\Psi(h)$ be as in \cref{prop:minimality-radial}. Consider            
$$
C_\alpha(h):=U_\alpha(h)-\Psi(h),\quad\text{where }\,U_\alpha(h):=(1-h)^\gamma\Big(1-(1-h)^{\alpha}\Big)^2.
$$
Then $C_\alpha$ satisfies \eqref{eq:calibration} for the following choices of $\alpha$:
             \begin{itemize}
                \item[(i)] if $\gamma\ge \frac{2}{d-2}$, we may take $\alpha:=\frac{2-\gamma}2$;
                \item[(ii)] if $\gamma\geq \gamma_\Delta(d)$ and $\gamma \geq \frac23$, we may take $\alpha:=\sqrt{\frac{\gamma(2-\gamma)}2}$.
            \end{itemize}
			\end{lemma}
            \begin{proof}
            For the sake of readability, we omit the dependence of $C$ and $U$ on $\alpha$. 
                First, we observe that $C(0)=0$.
                For $t:=1-h$, we compute \bea U'(h)&=-\gamma t^{\gamma-1}(1-t^\alpha)^2+t^{\gamma}2(1-t^\alpha)\alpha t^{\alpha-1}
                \\&=-\gamma t^{\gamma-1}(1-2t^{\alpha}+t^{2\alpha})+2\alpha t^{\alpha+\gamma-1}(1-t^{\alpha})\\&=-\gamma t^{\gamma-1}+2\gamma t^{\alpha+\gamma-1}-\gamma t^{2\alpha+\gamma-1}+2\alpha t^{\alpha+\gamma-1}-2\alpha  t^{2\alpha+\gamma-1}\\&=-\gamma t^{\gamma-1}+2(\gamma+\alpha) t^{\alpha+\gamma-1}-(\gamma+2\alpha) t^{2\alpha+\gamma-1}.
                \eea
                We also have
                $\Psi'(h)=-\gamma t^{\gamma-1}+\gamma t$, then \be\label{eq:C'}C'(h)=2(\gamma+\alpha) t^{\alpha+\gamma-1}-(\gamma+2\alpha) t^{2\alpha+\gamma-1}-\gamma t.\ee 
                We divide the rest of the proof into two cases.
                \\ \vspace{-0.3cm}

    \noindent\textit{Case 1.} If $\gamma\ge \frac{2}{d-2}$, then we choose $\alpha:=\frac{2-\gamma}2$. By \eqref{eq:C'}, we have
                $$C'(h)=(\gamma+2)t^\frac\gamma2
                -(\gamma+2)t=(\gamma+2)t^\frac\gamma2(1-t^{\frac{2-\gamma}2})=(\gamma+2)\sqrt{U(h)}.$$ 
                We now observe that the condition $\gamma\ge \frac{2}{d-2}$ is equivalent to requiring that $A(d,\gamma)\ge \frac{(\gamma+2)^2}{4}$. Indeed, by direct computation
                \be\label{eq:boundA}A(d,\gamma)- \frac{(\gamma+2)^2}{4}=(d-1)(2-\gamma)^2\frac{(d-2)\gamma-2}{4(d(2-\gamma)-2(1-\gamma))},\ee where $d(2-\gamma)-2(1-\gamma)=(d-1)(2-\gamma)+\gamma>0$. Therefore, by definition of $C$, we have $$C(h)+\Psi(h)=U(h)=\frac{(C'(h))^2}{(\gamma+2)^2}\ge\frac{(C'(h))^2}{4A(d,\gamma)},$$ namely $C(h)$ satisfies \eqref{eq:calibration}.
                \\ \vspace{-0.3cm}

    \noindent\textit{Case 2.} If $\gamma\ge\gamma_\Delta(d)$ and $\gamma\ge\frac23$, we choose $\alpha:=\sqrt{\frac{\gamma(2-\gamma)}2}$, and we want to prove that 
    \be\label{eq:calibration2}C(h)+\Psi(h)\ge \frac{(C'(h))^2}{4\cdot2\gamma(2-\gamma)}.\ee
    Indeed, if \eqref{eq:calibration2} holds, then $C(h)$ is a calibration satisfying \eqref{eq:calibration} once we notice that $A(d,\gamma)\ge 2\gamma(2-\gamma)$. This last condition 
    is equivalent to requiring that $\gamma\ge\gamma_\Delta(d)$, in fact $$A(d,\gamma)-2\gamma(2-\gamma)=\frac{\gamma\big((d-2)^2-4(1-\gamma)\lambda(\beta)\big)}{2\lambda(\beta)}=\frac{\gamma}{2\lambda(\beta)}\Delta(d,\gamma),$$ where $\Delta(d,\gamma)$ is defined in \eqref{e:Delta-d-gamma}. Thus we only need to prove \eqref{eq:calibration2}.

    By \eqref{eq:C'}, we have that $C(h)$ satisfies the inequality in \eqref{eq:calibration2} if and only if
                $$|2(\gamma+\alpha) t^{\alpha+\gamma-1}-(\gamma+2\alpha) t^{2\alpha+\gamma-1}-\gamma t|\le 2\sqrt{2\gamma(2-\gamma)}t^\frac\gamma2(1-t^\alpha)=4\alpha t^\frac\gamma2(1-t^\alpha).$$
                Dividing by $\gamma t^{\gamma-1}$, this is equivalent to requiring that, for $s:=t^\alpha$ and $\theta:=\frac{\alpha}{\gamma}$ $$|2(1+\theta)s-(1+2\theta)s^2-s^{2\theta}|\le 4\theta s^\theta(1-s)\qquad\text{for }s \in (0,1),$$
               where we used that $s^{\theta}=t^{\alpha^2/\gamma}=t^{\frac{2-\gamma}2}$. 
               
               First we notice that $2(1+\theta)s-(1+2\theta)s^2-s^{2\theta}\ge0$. Indeed, since $\gamma<1$, we have $\theta>\frac12$, and thus $s^{2\theta}\le s$, for $s\in(0,1)$. On the other hand, since $\gamma\ge\frac23$, then $\theta\le 1$ and thus $$2(1+\theta)s-(1+2\theta)s^2-s^{2\theta}=s-s^{2\theta}+(1+2\theta)s(1-s)\le s-s^{2\theta}+(1+2\theta)s^\theta(1-s).$$ Then, we only need to prove that $s-s^{2\theta}\le (2\theta-1)s^\theta(1-s)$ for every $s \in (0,1)$. Dividing by $s^{\theta}$, it is sufficient to prove that $$g(s):=(2\theta-1)(1-s)-s^{1-\theta}+s^\theta\ge0\quad\text{for every }s\in(0,1).$$ Since $\theta\in(\frac12,1]$, then $$g'(s)=-(2\theta-1)-(1-\theta)s^{-\theta}+\theta s^{\theta-1}\quad\text{and}\quad g''(s)=\theta(1-\theta)s^{-\theta-1}(1-s^{2\theta-1})\ge0.$$ Then $g'$ is non-decreasing, and since $g'(1)=0$, then $g'\le0$ in $(0,1)$. Then $g$ must be non-increasing, and since $g(1)=0$, we obtain $g(s)\ge0$, concluding the proof.
            \end{proof}
         
In contrast, with the aim of identifying the values of $\gamma$ for which the radial cone is not minimizing, in the following lemma we construct an explicit competitor showing that $b_{\text{rad}}$ fails to be minimizing in a suitable range of $\gamma$. We point out that this is relevant only in dimensions $d\geq 6$, since for $d=3,4,5$ the range provided by the lemma lies below $\gamma_\Delta(d)$.

	\begin{lemma}\label{lemma:minimality3}
				Let $d\ge3$, $\gamma\in(0,1)$ and $\eps\in(0,1)$. We define
				$$h_{\eps}(t):=\begin{cases}
				    1\quad&\text{ for } t\in (0,\eps],\\
                    1-(1-z_{\eps}(t))^\beta\quad&\text{ for } t\in (\eps,1],
				\end{cases}$$
                where \be\label{eq:sceltap}z_{\eps}(t):=\frac{\eps^p(t^{-p}-1)}{1-\eps^p}\quad\text{and}\quad p:=\frac{(2-\gamma)(2+\gamma)}{4A(d,\gamma)}\ee
                If $\gamma<\frac{2}{d-2}$ and $\gamma<\frac23$, then there exists $\overline{\eps}>0$ sufficiently small, such that $\mathcal{L}(h_{\overline{\eps}})<0$. In particular, $b_{\text{rad}}$ is not minimizing. 
			\end{lemma}
			
			\begin{proof}
				We first compute
				$$\int_0^1t^2 (\partial_t h_{\eps})^2\,dt=\int_{\eps}^1\beta^2(1-z_\eps(t))^{2\beta-2}\frac{p^2\eps^{2p}}{(1-\eps^p)^2}t^{-2p}\,dt.$$
                By substituting $t=\eps s$ and using that $z_\eps(\eps s)=\frac{s^{-p}-\eps^p}{1-\eps^p}$, we have
                \be\label{eq:contribution1}
                \int_0^1t^2 (\partial_t h_{\eps})^2\,dt =\frac{\eps\beta^2p^2}{(1-\eps^p)^{2\beta}}\int_{1}^{\frac1\eps}(1-s^{-p})^{2\beta-2}s^{-2p}\,ds =\eps\beta^2p^2\int_{1}^{\infty}(1-s^{-p})^{2\beta-2}s^{-2p}\,ds+o(\eps).
                \ee
				Notice that the integral in the right-hand side is finite if and only if $p>1/2$. In order to verify such condition, we notice that, by our choice of $p$, it is equivalent to requiring that $$A(d,\gamma)<\frac{(2-\gamma)(2+\gamma)}2.$$
                Thus, since $\gamma<\frac{2}{d-2}$, by \eqref{eq:boundA} we have \be\label{eq:boundA1}A(d,\gamma)<\frac{(\gamma+2)^2}4< \frac{(2-\gamma)(2+\gamma)}2,\ee where the second inequality follows by the fact that $\gamma<\frac{2}{3}$. Therefore $p>1/2$ is verified. 
                
                Since $h_{\eps}\equiv1$ on $(0,\eps)$, then we have				\be\label{eq:contribution2}\int_0^\eps \Psi(h_{\eps})\,dt=\eps\,\Psi(1)=\eps\left(
                -1+\frac{\gamma}{2}\right).\ee
				On the other hand, on $(\eps,1)$, by the change of variables $t=\eps s$ and using again that $z_\eps(\eps s)=\frac{s^{-p}-\eps^p}{1-\eps^p}$, we have
				\be\label{eq:contribution3}\begin{aligned} \int_\eps^1 \Psi(h_{\eps})\,dt&=\eps\int_1^{\frac1\eps}\Psi\left(1-\left(1-\frac{s^{-p}-\eps^p}{1-\eps^p}\right)^\beta\,\right)\,ds\\&=
				\eps\int_1^{\frac1\eps}\left(\left(1-\frac{s^{-p}-\eps^p}{1-\eps^p}\right)^{\gamma\beta}-1+\frac\gamma2-\frac\gamma2\left(1-\frac{s^{-p}-\eps^p}{1-\eps^p}\right)^{2\beta}\,\right)\,ds\\&=\eps\int_1^{\infty}\left((1-s^{-p})^{\gamma\beta}-1+\frac\gamma2-\frac\gamma2(1-s^{-p})^{2\beta}\right)\,ds+o(\eps).\end{aligned}\ee 
				If $\sigma:=\gamma\beta=2\beta-2$, then we claim that \be\label{eq:contrubution4}\Psi(1)+\int_1^{\infty}\left((1-s^{-p})^{\sigma}-1+\frac\gamma2-\frac\gamma2(1-s^{-p})^{\sigma+2}\right)\,ds=-(2p(\sigma+1)-1)\int_{1}^{\infty}(1-s^{-p})^{\sigma}s^{-2p}\,ds.\ee 
                               If the claim \eqref{eq:contrubution4} is true, we can conclude the proof by showing that $\mathcal{L}(h_\eps)<0$. Indeed, by combining \eqref{eq:contribution1}, \eqref{eq:contribution2}, \eqref{eq:contribution3} and \eqref{eq:contrubution4}, we get
                $$\mathcal{L}(h_\eps)=\eps\left(A(d,\gamma)\frac{\sigma^2}{\gamma^2}p^2-2p(\sigma+1)+1\right)\int_1^\infty (1-s^{-p})^\sigma s^{-2p}\,ds+o(\eps).$$
        
        We point out that the exponent $p$ in \eqref{eq:sceltap} is chosen as the minimizer of the quadratic coefficient. Indeed, differentiating this term with respect to $p$ and imposing that the derivative vanishes yields the choice in 
        \eqref{eq:sceltap}, i.e.,
                $$p=\frac{\gamma^2(\sigma+1)}{A(d,\gamma)\sigma^2}=\frac{(2-\gamma)(2+\gamma)}{4A(d,\gamma)}.$$ 
                Using the second expression for $p$ above, together with the identities $\sigma^2=\frac{4\gamma^2}{(2-\gamma)^2}, \sigma+1=\frac{2+\gamma}{2-\gamma}$, we obtain
                \bea \mathcal{L}(h_\eps)&=\eps\left(\frac{\sigma^2}{\gamma^2}\frac{(2-\gamma)^2(2+\gamma)^2}{16 A(d,\gamma)}-\frac{(2-\gamma)(2+\gamma)}{2A(d,\gamma)}(\sigma+1)+1
                \right)\int_1^\infty (1-s^{-p})^\sigma s^{-2p}\,ds+o(\eps)\\&=\eps\left(-\frac{(2+\gamma)^2}{4 A(d,\gamma)}+1
                \right)\int_1^\infty (1-s^{-p})^\sigma s^{-2p}\,ds+o(\eps).\eea Since $A(d,\gamma)<\frac{(\gamma+2)^2}4$ by \eqref{eq:boundA1}, we finally have that $\mathcal{L}(h_\eps)<0 $, by choosing $\eps$ small enough.

                It remains to prove the claim \eqref{eq:contrubution4}. First, if we define $$y(s):=1-s^{-p}\quad\text{and}\quad F(x):=x^{\sigma}-1+\frac\gamma2-\frac\gamma2x^{\sigma+2},$$
                we can rewrite \eqref{eq:contrubution4} as
$$
\Psi(1) + \int_1^\infty F(y(s))\,ds = -(2p(\sigma+1)-1)\int_1^\infty (1-y(s))^2 (y(s))^\sigma\,ds.
$$
Since $\sigma=\frac{2\gamma}{2-\gamma}$, we have $\frac{\gamma}{2}(\sigma+2)=\sigma$, and so \be\label{eq:derivativeF}F'(x)=\sigma x^{\sigma-1}-\frac\gamma2(\sigma+2)x^{\sigma+1}=\sigma x^{\sigma-1}(1-x^2).\ee
Then, we have that $F(0)=\Psi(1)$, $F(1)=F'(1)=0$ and, since $F$ is regular near $x=1$, a Taylor expansion gives
$$sF(y(s))=sO\big((1-y(s))^2\big)=O(s^{1-2p})\to0\quad\text{as }s\to+\infty,$$ 
where the asymptotic relies on the lower bound $p>1/2$.
Thus, integrating by parts we get 
\begin{align*}\int_1^\infty F(y(s))\,ds &=\lim_{s \to +\infty}sF(y(s)) - F(0)-
\int_1^\infty s F'(y(s))y'(s)\,ds\\ &= -\Psi(1)-
\int_1^\infty s F'(y(s))y'(s)\,ds.
\end{align*}
Using that $s y'(s)=p(1-y)$ and \eqref{eq:derivativeF}, we have
$$\Psi(1)+\int_1^\infty F(y(s))\,ds=-p\sigma
\int_1^\infty y(s)^{\sigma-1}(1-y(s))^2(1+y(s))\,ds.$$
Finally, the conclusion follows by integrating on $(1,\infty)$ the following identity
$$\frac{d}{ds}\left(sy^\sigma(1-y)^2\right)=p\sigma y^{\sigma-1}(1-y)^2(1+y)-\big(2p(\sigma+1)-1\big)y^\sigma(1-y)^2.
$$
We simply point out that there are no boundary contributions at $s=1$ or at $s=+\infty$. Indeed, $y(1)=0$, while $s y^\sigma(1-y)^2=O(s^{1-2p})\to 0$ as $s\to+\infty$.
			\end{proof}
			Now we can prove the main result for the minimality of the radial cone in \cref{thm:minimality-of-radial-cone}.
			\begin{proof}[Proof of \cref{thm:minimality-of-radial-cone}] 
                Since the Alt-Phillips functional is convex for $\gamma \in [1,2)$, it is not restrictive to consider the case $\gamma \in (0,1)$, see \cref{prop:convex-functional}. As proved in \cite{savinyu-altphillips-stable}, if $\gamma\ge \gamma_{\Delta}(d)$, then $b_{\text{rad}}$ is minimizing if and only if $b_{\text{rad}}$ is one-sided minimizing from below. We split the proof into three cases.

                \begin{itemize}
                    \item[(i)] Suppose that $d=3,4,5$. 
                In this case, since $\gamma_{\Delta}(d)\ge\frac{2}3$, the assumptions of case (ii) of  
\cref{lemma:minimality2} are satisfied and so $b_{\text{rad}}$ is minimizing. 
\item[(ii)] Now consider the case $d\geq 6$ and $\gamma\ge\frac{2}{3}$. Since $\gamma_{\Delta}(d)<\frac{2}{3}$, we proceed again by applying case (ii) of \cref{lemma:minimality2} and we conclude that $b_{\text{rad}}$ is minimizing.

\item[(iii)] Finally, we suppose that $d\ge6$ and $\gamma\in(0,\frac23)$. In this case, we proceed by applying case (i) of \cref{lemma:minimality2} and \cref{lemma:minimality3}. Thus, we obtain that $b_{\text{rad}}$ is minimizing if and only if $\gamma\ge \frac{2}{d-2}$, concluding the proof.
                \end{itemize}
			\end{proof}

			\section{Epiperimetric inequality for the one-dimensional cone}\label{section-1d}
			In this section we prove an epiperimetric inequality near the one-dimensional cone $b_{\text{one}}:=c_{\text{one}}|x_d|^\beta$, which is the only singular cone in $\mathcal{B}_{d-1}$, up to a rotation. We also show that $b_{\text{one}}$ is not integrable in the sense of \cref{def:integrability}, see \cref{remark:non-int-one-dim-cone}.
            
            We notice that $b_{\text{one}}$ is minimizing if and only if $\gamma\in[1,2)$, since there are no minimizing cones in dimension $d=2$ for $\gamma\in(0,1)$. The case $\gamma=1$ corresponding to the obstacle problem has already been extensively studied \cite{Weiss-improvement,csv18,fs19}, so we restrict our analysis to $\gamma>1$.
			\begin{proposition}[Epiperimetric inequality]\label{prop:epi-one-dim-cone} Let $\gamma\in (1,2)$, then 
            the epiperimetric inequality in \cref{thm:epiperimetric-combined} holds with $b=b_{\text{one}}$ and $\sigma=0$, under the closeness assumption 
            \bea \|z-b_{\text{one}}\|_{H^1(B_1)}\le \delta.\eea 
			\end{proposition}
			We point out that the epiperimetric inequality in \cref{prop:epi-one-dim-cone} is classical, namely $\sigma=0$, although the one-dimensional cone is not integrable, see \cref{remark:non-int-one-dim-cone}.
			\subsection{Linearized operator around the one-dimensional cone}
            In the following lemma, we characterize the kernel of the linearized operator near the one-dimensional cone for $\gamma>1$.

            We point out that, for $\gamma>1$, no additional boundary condition across the interface $\{x_d=0\}$ is needed in order to classify the $\beta$-homogeneous solutions of the associated linearized equation. Indeed, it suffices to impose the weak equation in $\{|x_d|>0\}$.
            
			\begin{lemma}\label{lemma:ker-1d} Let $\gamma\in(1,2)$ and $w\in H^1(B_1)$ be a $\beta$-homogeneous weak solution of
				\bea
				\Delta w=\frac\gamma2(\gamma-1)b_{\text{one}}^{\gamma-2}w\quad\text{in }B_1\cap\{|x_d|>0\},
				\eea then, there exists $ c_i^{(1)},c_i^{(2)}\in\R$ such that $$w(x)=\sum_{i=1}^{d-1} c_i^{(1)}|x_d|^{\beta-1}x_i+\sum_{i=1}^{d-1}c_i^{(2)}\,\text{sgn}(x_d)|x_d|^{\beta-1}x_i.$$ 
                In particular, the kernel of the spherical linearized operator in \eqref{eq:general-kernel} at $b=b_{\text{one}}$ is given by \bea\ker( L_{b_{\text{one}}})=\text{span}\left\{|x_d|^{\beta-1}x_i,\text{sgn}(x_d)|x_d|^{\beta-1}x_i\right\}_{i=1}^{d-1}.\eea
			\end{lemma}
			\begin{proof}
				Let us set $x=(x',x_d)\in \R^{d-1}\times \R$. By the definition of $b_{\text{one}}$, we get that $w$ solves $$\Delta w=\alpha(\alpha-1)\frac{w}{|x_d|^2}\quad\text{in }B_1\cap\{|x_d|>0\},\quad\text{where }\alpha:=\beta-1.$$ If $v:B_1\cap\{x_d>0\}\to\R$ is defined by $v(x):=x_d^{-\alpha}w(x)$, then $v$ is $1$-homogeneous and solves \be\label{eq:equation-v}\text{div}(x_d^{2\alpha}\nabla v)=0\quad\text{in }B_1^+:=B_1\cap\{x_d>0\}.\ee 
                Let $i=1,\dots,d-1$, then $v_i:=\partial_i v$ is a $0$-homogeneous solution to \eqref{eq:equation-v}, namely  \be\label{eq:test}\text{div}_{\theta}(x_d^{2\alpha}\nabla_\theta v_i)=0\quad\text{in }(\partial B_1)^+:=(\partial B_1)\cap\{ x_d>0\},\ee where $\text{div}_{\theta}$ is the spherical divergence operator.
				
                We now observe that, since $\alpha>1/2$, the boundary $\{x_d=0\}\cap \partial B_1$ has zero weighted capacity with respect to the weight $x_d^{2\alpha}$, and thus the space $C^\infty_c((\partial B_1)^+)$ is dense in the weighted Sobolev space $H^1((\partial B_1)^+,x_d^{2\alpha})$ (see for instance \cite[Proposition 2.2]{STV}).
				Then, we may use $v_i$ itself as a test function in the weak formulation of \eqref{eq:test}, obtaining that 
                $$\int_{(\partial B_1)^+}x_d^{2\alpha}|\nabla_\theta v_i|^2\,d\HH^{d-1}=0.$$
                Therefore, for every $i=1,\dots,d-1$, $v_i$ is constant on ${(\partial B_1)^+}$ and so  $v(x)= a\cdot x'+ cx_d$, for some $a\in\R^{d-1}$, $c \in \R$. 
                Using \eqref{eq:equation-v}, we get $$0=\text{div}(x_d^{2\alpha}\nabla v)=2\alpha cx_d^{2\alpha-1}\quad\text{in }B_1^+,$$ namely $c=0$. Then, by the definition of $v$, we obtain $w(x)=x_d^{\beta-1} (a\cdot x')$ in $B_1^+$.  
                
                Similarly, we obtain $w(x)=(-x_d)^{\beta-1}\widetilde a\cdot x'$ in $B_1\cap\{x_d<0\}$, for some $\widetilde a\in\R^{d-1}$. Finally, the conclusion follows by choosing $c_i^{(1)}:=(a_i+\widetilde a_i)/2$ and $c_i^{(2)}:=(a_i-\widetilde a_i)/2$.
			\end{proof}
             \begin{remark}
                 A similar Liouville-type theorem was proved in \cite[Lemma 3.5]{ros-restrepo} for the linearized operator around the flat solution. 
            In the present setting, one has to couple the corresponding Liouville classifications on the two sides of the interface. On each side, the possible normal behaviors include the two modes $(x_d)_+^{\beta-1}x_i$ and $(x_d)_+^{2-\beta}x_i$, but the $H^1$-admissibility rules out the singular branch $(x_d)_+^{2-\beta}x_i$.
             \end{remark}
             
             \begin{remark}
                 We point out that, unlike the obstacle problem $\gamma=1$, in the case $\gamma\in(1,2)$ the kernel of the linearized operator contains the functions $|x_d|^{\beta-1}x_i$, which give the non-integrability of the one-dimensional cone, see \cref{remark:non-int-one-dim-cone}.
			
             \end{remark}
			\subsection{The family $b_{\xi,\omega}$}
            We point out a fundamental difference between the two families of functions $\text{sgn}(x_d)|x_d|^{\beta-1}x_i$ and $|x_d|^{\beta-1}x_i$ appearing in the kernel of the linearized operator in \cref{lemma:ker-1d}.
            The functions $\text{sgn}(x_d)|x_d|^{\beta-1}x_i$ are generated by rotations of $b_{\text{one}}$, whereas the modes $|x_d|^{\beta-1}x_i$ do not arise from rotations. One way to construct the latter family is to rotate the two half-plane components of $b_{\text{one}}$ in antipodal directions.

Precisely, let $\xi,\omega \in \R^{d-1}$. We denote by $R_\omega\in SO(d)$ the rotation of angle $\arctan|\omega|$ which fixes the subspace of horizontal vectors orthogonal to $\omega$ and rotates the two-dimensional plane spanned by $e_d$ and  $(\omega,0)$. The orientation is chosen so that 
$$
R_\omega e_d=\frac{(\omega,1)}{\sqrt{1+|\omega|^2}},$$ and we set $R_0:=\mathrm{Id}$.
For $i=1,2$ we define 
$$
b_{\xi,\omega,i}:= c_{\text{one}} (x\cdot \nu_i(\xi,\omega))_+^\beta\quad\text{where}\quad\nu_1(\xi,\omega):=R_\omega\frac{(\xi,1)}{\sqrt{1+|\xi|^2}},\quad \nu_2(\xi,\omega):=R_\omega\frac{(\xi,-1)}{\sqrt{1+|\xi|^2}},
$$
and finally, we set $b_{\xi,\omega}:=b_{\xi,\omega,1}+b_{\xi,\omega,2}$, namely $$b_{\xi,\omega}:=c_{\text{one}}\left(x\cdot R_\omega\frac{(\xi,1)}{\sqrt{1+|\xi|^2}}\right)_+^\beta+c_{\text{one}}\left(x\cdot R_\omega\frac{(\xi,-1)}{\sqrt{1+|\xi|^2}}\right)_+^\beta.$$
	
             In this construction, the parameter $\omega$ describes a standard rotation of $b_{\mathrm{one}}$
    whose action generates the modes $\text{sgn}(x_d)|x_d|^{\beta-1}x_i$. By contrast, $\xi$ describes an antipodal rotation of the two half-plane components, producing the modes $|x_d|^{\beta-1}x_i$.
    
    Indeed, we first note that $\partial_{\xi_i}\nu_1(0,0)=\partial_{\xi_i}\nu_2(0,0)=e_i$. Moreover, if $A_i:=\partial_{\omega_i}R_\omega|_{\omega=0}$, then 
    $$
    A_ie_d=e_i,\qquad A_ie_j=-\delta_{ij}e_d,\quad\text{for }j=1,\ldots,d-1.
    $$
    Hence $\partial_{\omega_i} \nu_1(0,0)=e_i$ and $\partial_{\omega_i} \nu_2(0,0)=-e_i$. Therefore
    \be\label{e:useful-for-notintegrable}\partial_{\xi_i} b_{\xi,\omega}\Big|_{(\xi,\omega)=(0,0)}=c_{\text{one}} \beta |x_d|^{\beta-1}x_i\quad\text{and}\quad \partial_{\omega_i} b_{\xi,\omega}\Big|_{(\xi,\omega)=(0,0)}=c_{\text{one}} \beta\, \text{sgn}(x_d)|x_d|^{\beta-1}x_i.\ee

Last, we observe that $b_{\xi,\omega}$ is obtained from $b_{\xi,0}$ by the rotation $R_\omega$, namely $b_{\xi,\omega}=R_\omega b_{\xi,0}$. Accordingly, when $\omega=0$ we omit the dependence on the parameter $\omega$ and simply write $b_\xi:=b_{\xi,0}$. On the other hand, the family $b_{\xi,\omega}$ consists of solutions only for $\xi=0$. In particular, the $\xi$-directions produce a non-trivial first variation.

            In the next lemma we show that the Weiss' energy of $b_{\xi}$ is strictly smaller than the one of $b_{\text{one}}$, and their difference is comparable to $|\xi|^{2\beta-1}$. This asymptotic will be crucial in the proof of \cref{prop:epi-one-dim-cone} to estimate the first variation produced by $b_{\xi}$.
            \begin{lemma}\label{lemma:negative-b}
				Let $b_{\xi}:=b_{\xi,0}$, then for every $\xi\in\R^{d-1}$ small enough, we have that $$ W(b_{\xi})-W(b_{\text{one}})=-\kappa_\beta|\xi|^{2\beta-1}+o(|\xi|^{2\beta-1})\quad\text{as }|\xi|\to0^+,$$ for some $\kappa_\beta>0$ depending only on $d$ and $\gamma$.             
            \end{lemma}
			\begin{proof}
Without loss of generality, we suppose that $\xi=te_1$, for some $t>0$ small enough. We also call $f_t=b_{te_1,0,1}$ and $g_t=b_{te_1,0,2}$. Denoting by $D_t:=\{f_t>0\}\cap\{g_t>0\}\cap B_1=\{|x_d|< tx_1\}\cap B_1$ and using that $W(b_{\text{one}})=W(f_0)+W(g_0)$ and that $W(f_0)=W(f_t)$ and $W(g_0)=W(g_t)$, we have, by integration by parts, that \be\label{eq:eq12} W(b_{te_1})-W(b_{\text{one}})=W(f_t+g_t)-W(f_t)-W(g_t)=\int_{D_t}\Big((f_t+g_t)^\gamma- f_t^\gamma-g_t^\gamma-\gamma f_t^{\gamma-1}g_t\Big)\,dx,\ee
				where we used that the latter integral is zero when computed on $B_1\setminus D_t$. Denoting by $x'=(x_1,\ldots,x_{d-1})$, we change coordinates by defining $s:=x_d/t$ and $E_t$ to be the rescaling of $D_t$ along the $e_d$-direction, i.e., $$ E_t:=\{(x',s):\ x_1>0, \ |s|< x_1,\ |x'|^2+t^2s^2<1\}.$$ Notice that $$f_t(x_1,ts)=c_\beta\left(\frac{1}{\sqrt{1+t^2}}\right)^\beta t^\beta(x_1+s)^\beta\quad\text{and}\quad g_t(x_1,ts)=c_\beta\left(\frac{1}{\sqrt{1+t^2}}\right)^\beta t^\beta(x_1-s)^\beta.$$ We consider the function $$\Theta(x_1,s):=\left((x_1+s)^\beta+(x_1-s)^\beta\right)^\gamma-(x_1+s)^{\beta\gamma}-(x_1-s)^{\beta\gamma}-\gamma (x_1+s)^{\beta(\gamma-1)}(x_1-s)^\beta.$$ Since $dx=t\,dx'\,ds$, by \eqref{eq:eq12}, we have $$W(b_{\xi})-W(b_{\text{one}})=c_\beta^\gamma(1+t^2)^{-\beta\gamma/2} t^{\beta\gamma+1}\int_{E_t}\Theta(x_1,s)\,dx'\,ds=-\kappa_\beta t^{2\beta-1}+t^{2\beta-1}R(t),$$ where, for $I_t:=c_\beta^\gamma\int_{E_t}\Theta(x_1,s)\,dx'\,ds$, we define $$\kappa_\beta:=- I_0\quad\text{and}\quad R(t):=\big((1+t^2)^{-\beta\gamma/2}-1\big)I_0+(1+t^2)^{-\beta\gamma/2}(I_t-I_0).$$ 
				Then the claim follows by the fact that $R(t)\to 0$ as $t\to0^+$ and that $I_0<0$, since in $E_0$ we have $$\Theta(x_1,s)=x_1^{\beta\gamma}\Theta(1,s/x_1)=x_1^{\beta\gamma}(1+s/x_1)^{\beta\gamma}\Psi\left(\frac{(1-s/x_1)^\beta}{(1+s/x_1)^\beta}\right),$$ where $\Psi(r):=(1+r)^\gamma-1-r^\gamma-\gamma r<0$ for every $r>0$.
			\end{proof}

			\subsection{Proof of the epiperimetric inequality}
			Now we can proceed with the proof of the epiperimetric inequality near the one-dimensional cone $b_{\text{one}}$ in \cref{prop:epi-one-dim-cone}.
			\begin{proof}[Proof of \cref{prop:epi-one-dim-cone}]
				By contradiction, let us suppose that there is a sequence of non-negative traces $c_j\in H^1(\partial B_1)$ and constants $\eps_j\to0^+$, $\delta_j\to0^+$ such that if $z_j$ is the $\beta$-homogeneous extension of $c_j$, then we have for $b:=b_{\text{one}}$
				$$\|z_j-b\|_{H^1(B_1)}\le \delta_j$$
				and \be\label{eq:contr-hyp1} (1-\eps_j)(W(z_j)-W(b))<W(h)-W(b),\ee for every non-negative function $h\in H^1(B_1)$ such that $h=c_j$ on $\partial B_1$.
				
				For every $j$, we can find $\xi_j,\omega_j \in \R^{d-1}$ such that $b_j:=b_{\xi_j,\omega_j}$ is the projection of $z_j$ into the space $\{b_{\xi,\omega}\}_{\xi,\omega\in\R^{d-1}}$, i.e., \be\label{eq:best-approx}\eta_j:=\|z_j-b_{j}\|_{H^1(B_1)}=\min_{\xi,\omega\in\R^{d-1}}\|z_j-b_{\xi,\omega}\|_{H^1(B_1)}\le \|z_j-b\|_{H^1(B_1)}\le \delta_j.\ee
                Up to a rotation, it is not restrictive to assume that $\omega_j=0$ and so that $\eta_j,\xi_j\to0$, as $j\to+\infty$. Lastly, for $i=1,2$, we set $b_{j,i}:=b_{\xi_j,0,i}$.
				
				By \cref{lemma:negative-b}, we have that $W(b_{j})\le W(b)$, then the contradiction hypothesis \eqref{eq:contr-hyp1} becomes
				\be\label{eq:contrad-hyp2}(1-\eps_j)(W(z_j)-W(b_j))<W(h)-W(b_j).\ee
                Moreover, we have that $\eta_j>0$, otherwise we can take $h=b_j=z_j$ in \eqref{eq:contrad-hyp2}, obtaining a contradiction.
                
				By integrating by parts, we first observe that for every non-negative $h\in H^1(B_1)$\bea W_0(h)-W_0(b_j)&=\int_{B_1}|\nabla (h-b_j)|^2\,dx-\beta\int_{\partial B_1}(h-b_j)^2\,d\HH^{d-1}+2\int_{B_1}\nabla b_j\cdot \nabla (h-b_j)\,dx\\&\qquad-2\beta \int_{\partial B_1}b_j (h-b_j)\,d\HH^{d-1}\\&=\int_{B_1}|\nabla (h-b_j)|^2\,dx-\beta\int_{\partial B_1}(h-b_j)^2\,d\HH^{d-1}-2\int_{B_1}\Delta b_j (h-b_j)\,dx.\eea
				Using that $ 2\Delta b_j=\gamma b_{j,1}^{\gamma-1}+\gamma b_{j,2}^{\gamma-1}$, we have 
				\be\label{eq:formula-in-epi} W(h)-W(b_j)=W_0(h-b_j)+
				\int_{B_1}\Big(f_j(h-b_j)+Q_j(h-b_j)\Big)\,dx,\ee where \be\label{def:fj}f_j:=\gamma  \big(b_{j}^{\gamma-1}-b_{j,1}^{\gamma-1}-b_{j,2}^{\gamma-1}\big)\le0,\ee 
                and \bea Q_j(w)&:=(b_j+w)^\gamma-b_j^\gamma-\gamma b_j^{\gamma-1}w=\int_0^1\int_0^t \gamma(\gamma-1)(b_j+sw)^{\gamma-2}
				w^2\,ds\,dt\ge0\eea is defined for every $w$ such that $b_j+w\ge0$. We stress that the term $f_j$ corresponds to the first variation, which is non-zero, since $b_{\xi,\omega}$ is not a solution for $\xi\not=0$.
				
                Therefore, the inequality \eqref{eq:contrad-hyp2} becomes
				\be\label{eq:contrad-hyp3}
				\begin{aligned}
					(1-\eps_j)&\bigg( W_0(z_j-b_j)+\int_{B_1}\Big(f_j(z_j-b_j)+Q_j(z_j-b_j)\Big)\,dx
					\bigg)\\&\qquad\qquad\qquad\le W_0(h-b_j)+\int_{B_1}\Big(f_j(h-b_j)+Q_j(h-b_j)\Big)\,dx.
				\end{aligned}
				\ee
				Now, we introduce $w_j:=(z_j-b_j)/\eta_j$, and we choose a subsequence $j\to+\infty$ such that $ w_j\to w$ weakly in $H^1(B_1)$, strongly in $L^2(B_1)$, to a $\beta$-homogeneous $w \in H^1(B_1).$
                
                We divide the rest of the proof in several steps.
				\\ \vspace{-0.3cm}
				
				\noindent\textit{Step 1.} Let us show that there exists $C>0$ such that  \be\label{eq:proof-step-1-cylindrical}\int_{B_1}\bigg(\frac{f_j(z_j-b_j)}{\eta_j^2}+\frac{Q_j(z_j-b_j)}{\eta_j^2}\bigg)\,dx\le C,\ee for every $j>0$ large enough. 
				To prove such uniform estimate, consider     $h:=(1-\zeta)z_j+\zeta b_j$ where $\zeta\in C^\infty_c(B_1)$ is a radial function satisfying $0\le\zeta\le1$. Since $h\ge 0$ in $B_1$ and $h=c_j$ on $\partial B_1$, it is an admissible competitor in \eqref{eq:contrad-hyp3}.
                
                We observe that since $\gamma\in (1,2)$, the function $\varphi\mapsto Q_j(\varphi)$ is convex and, being $Q_j(0)=0$, it implies that
				\bea Q_j((1-\zeta)(z_j-b_j))\le
				(1-\zeta)Q_j(z_j-b_j).\eea 
				Using that $w_j$ is bounded in $H^1(B_1)$, then \eqref{eq:contrad-hyp3} reads as
				\bea(1-\eps_j) \int_{B_1}\bigg(\frac{f_j(z_j-b_j)}{\eta_j^2}+\frac{Q_j(z_j-b_j)}{\eta_j^2}\bigg)\,dx\le C+\int_{B_1} (1-\zeta)\bigg(\frac{f_j(z_j-b_j)}{\eta_j^2}+\frac{Q_j(z_j-b_j)}{\eta_j^2}\bigg)\,dx.
				\eea
				Then $$\int_{B_1}(\zeta-\eps_j)\bigg(\frac{f_j(z_j-b_j)}{\eta_j^2}+\frac{Q_j(z_j-b_j)}{\eta_j^2}\bigg)\,dx\le C.$$
				Using that $\zeta$ is radial and the homogeneity of the terms involved, we get
				$$\int_0^1 ( \zeta(r)-\eps_j)r^{d+2\beta-3}\,dr \int_{\partial B_1} \bigg(\frac{f_j(z_j-b_j)}{\eta_j^2}+\frac{Q_j(z_j-b_j)}{\eta_j^2}\bigg)\,d\HH^{d-1}\le C.$$ 
                Finally, \eqref{eq:proof-step-1-cylindrical}, follows by choosing a suitable $\zeta$ and by exploiting the homogeneity of terms in \eqref{eq:proof-step-1-cylindrical}.\\ \vspace{-0.3cm}
				
				\noindent\textit{Step 2.}
				We prove that $w\in H^1(B_1)$ is a weak solution to the linearized equation \be\label{eq:proof-step-2-cylindrical}
				\Delta w=\frac\gamma2(\gamma-1)b^{\gamma-2}w\quad\text{in }B_1\cap\{|x_d|>0\}.
				\ee
                We will use as a key ingredient that the support of the first variation $f_j$ concentrates near $\{x_d=0\}$ as $j\to+\infty$.
                
				Given $\zeta\in C^\infty_c(B_1\cap\{|x_d|>0\})$, $0\le\zeta\le1$ and $\varphi\in H^1(B_1)\cap L^\infty(B_1)$, we choose $h:=\zeta (b_j+\eta_j\varphi)+(1-\zeta)z_j$. Since the support of $\zeta$ is away from $\{x_d=0\}$ and $\varphi$ is bounded, then $h\ge0$ for $j$ large enough, thus $h$ is an admissible competitor in \eqref{eq:contrad-hyp3}.
				
				Since $w_j$ is bounded in $H^1(B_1)$ and $h-b_j=(1-\zeta)(z_j-b_j)+\zeta\eta_j\varphi$, then \eqref{eq:contrad-hyp3} reads as
				\bea 
				\int_{B_1}&|\nabla w_j|^2\,dx+\int_{B_1}\bigg(\frac{f_j(z_j-b_j)}{\eta_j^2}+\frac{Q_j(z_j-b_j)}{\eta_j^2}\bigg)\,dx\\&\qquad\le C\eps_j+\int_{B_1}\Big(|\nabla(\zeta\varphi)|^2+|\nabla ((1-\zeta)w_j)|^2+2\nabla(\zeta\varphi)\cdot \nabla((1-\zeta)w_j)\Big)\,dx\\&\qquad\qquad+\int_{B_1}\bigg(\frac{f_j(h-b_j)}{\eta_j^2}+\frac{Q_j(h-b_j)}{\eta_j^2}\bigg)\,dx,
				\eea where we used \eqref{eq:proof-step-1-cylindrical}.
				Therefore
				\bea 
				\int_{B_1}&(1-(1-\zeta)^2)|\nabla w_j|^2\,dx+\int_{B_1}\bigg(\frac{f_j(z_j-b_j)}{\eta_j^2}+\frac{Q_j(z_j-b_j)}{\eta_j^2}\bigg)\,dx\\&\qquad\le C\eps_j+\int_{B_1}\Big(|\nabla(\zeta\varphi)|^2+|\nabla\zeta|^2 w_j^2-2w_j(1-\zeta)\nabla \zeta\cdot\nabla w_j
				\\&\qquad\qquad+2\nabla(\zeta\varphi)\cdot \nabla((1-\zeta)w_j)\Big)\,dx+\int_{B_1}\bigg(\frac{f_j(h-b_j)}{\eta_j^2}+\frac{Q_j(h-b_j)}{\eta_j^2}\bigg)\,dx. 
				\eea
				Now we notice that $Q_j(z_j-b_j)=Q_j(h-b_j)$ where $\zeta=0$. Moreover, since $\xi_j\to 0$, the support of $f_j$ in $B_1$ concentrates near $\{x_d=0\}$. Indeed $f_j$ is nontrivial whenever both $b_{j,1}$ and $b_{j,2}$ are strictly positive. On the other hand, since $\zeta$ vanishes in a neighborhood of $\{x_d=0\}$, then, for $j$ large enough $$f_j(z_j-b_j)-f_j(h-b_j)=\zeta f_j(z_j-b_j)-\zeta\eta_jf_j\varphi\equiv0.$$ Therefore
				\bea 
				\int_{B_1}&(1-(1-\zeta)^2)|\nabla w_j|^2\,dx+\int_{B_1\cap\,\text{supp}(\zeta)}\frac{Q_j(z_j-b_j)}{\eta_j^2}\,dx\\&\qquad\le C\eps_j+\int_{B_1}\Big(|\nabla(\zeta\varphi)|^2+|\nabla\zeta|^2 w_j^2-2w_j(1-\zeta)\nabla \zeta\cdot\nabla w_j
				\\&\qquad\qquad+2\nabla(\zeta\varphi)\cdot \nabla((1-\zeta)w_j)\Big)\,dx+\int_{B_1\cap\,\text{supp}(\zeta)}\frac{Q_j(h-b_j)}{\eta_j^2}\,dx. 
				\eea
				Since $\text{supp}(\zeta)\cap \{x_d=0\}=\emptyset$, by passing to the limit as $j\to+\infty$ we get
				\bea 
				\int_{B_1}&|\nabla w|^2\,dx+\int_{B_1\cap\,\text{supp}(\zeta)}\frac\gamma2(\gamma-1)b^{\gamma-2}w^2\,dx\\&\qquad\le \int_{B_1}|\nabla(\zeta \varphi+(1-\zeta)w)|^2\,dx+\int_{B_1\cap\,\text{supp}(\zeta)}\frac\gamma2(\gamma-1)b^{\gamma-2}(\zeta\varphi+(1-\zeta)w)^2\,dx. 
				\eea
                Notice that by approximation we can drop the assumption $\varphi\in L^\infty(B_1)$. Finally, let $B\Subset B_1\cap\{|x_d|>0\}$, by choosing $\zeta\equiv 1$ in $B$ and $\varphi=w$ outside $B$, we obtain the desired minimality condition
				\bea 
				\int_{B}&|\nabla w|^2\,dx+\frac\gamma2(\gamma-1)\int_{B}b^{\gamma-2}w^2\,dx\le \int_{B}|\nabla\varphi|^2\,dx+\frac\gamma2(\gamma-1)\int_{B}b^{\gamma-2}\varphi^2\,dx, 
				\eea which concludes the proof of \eqref{eq:proof-step-2-cylindrical}.
				\\ \vspace{-0.3cm}
				
				\noindent\textit{Step 3.} We proceed by showing that $w\equiv0$. This result is a direct consequence of the fact that $b_j$ is the best approximation of $z_j$ in the class $\{b_{\xi,\omega}\}_{\xi,\omega\in\R^{d-1}}$. Indeed, since $w\in H^1(B_1)$ is a $\beta$-homogeneous solution of \eqref{eq:proof-step-2-cylindrical}, by \cref{lemma:ker-1d} we have $$w(x)=\sum_{i=1}^{d-1} c_i^{(1)}|x_d|^{\beta-1}x_i+\sum_{i=1}^{d-1}c_i^{(2)}\,\text{sgn}(x_d)|x_d|^{\beta-1}x_i,\quad \text{for some } c_i^{(1)},c_i^{(2)}\in\R.
                $$
				By \eqref{eq:best-approx}, the map $(\xi,\omega)\mapsto \|z_j-b_{\xi,\omega}\|^2_{H^1(B_1)}$ has a minimum at $(\xi,\omega)=(\xi_j,0)$, then by differentiating it we get \be\label{eq:passing-to-limit}\left\langle w_j, \partial_{\xi_i} b_{\xi,\omega}\Big|_{(\xi,\omega)=(\xi_j,0)}\right\rangle_{H^1}=\left\langle w_j, \partial_{\omega_i} b_{\xi,\omega}\Big|_{(\xi,\omega)=(\xi_j,0)}\right\rangle_{H^1}=0,\ee 
                for every $i=1,\ldots,d-1$. In the previous equation we denote by $\langle\cdot,\cdot\rangle_{H^1}$ the scalar product in $H^1(B_1)$.
				By combining the weak convergence $w_j\to w$ in $H^1(B_1)$, with the strong convergence $$\partial_{\xi_i} b_{\xi,\omega}\Big|_{(\xi,\omega)=(\xi_j,0)} \to \partial_{\xi_i} b_{\xi,\omega}\Big|_{(\xi,\omega)=(0,0)}\quad\text{and}\quad \partial_{\omega_i} b_{\xi,\omega}\Big|_{(\xi,\omega)=(\xi_j,0)} \to \partial_{\omega_i} b_{\xi,\omega}\Big|_{(\xi,\omega)=(0,0)}$$ in $H^1(B_1)$, we can pass to the limit as $j\to+\infty$ in \eqref{eq:passing-to-limit}. Hence, by \eqref{e:useful-for-notintegrable}, we get $c_i^{(1)}=c_i^{(2)}=0$ for every $i=1,\ldots,d-1$, thus $w\equiv0$.
				\\ \vspace{-0.3cm}
				
				\noindent\textit{Step 4.} We now improve the estimate \eqref{eq:proof-step-1-cylindrical}, by showing that for every $\tau>0$, we have \be\label{eq:step4-cylindrical}
				\int_{B_1}\bigg(\frac{f_j(z_j-b_j)}{\eta_j^2}+\frac{Q_j(z_j-b_j)}{\eta_j^2}\bigg)\,dx\le \tau
				\ee
				for every $j>j_0(\tau)$ large enough. This improvement follows from the fact that we now know that $w_j\to 0$ weakly in $H^1(B_1)$, strongly in $L^2(B_1)$.
				Indeed, by the weak convergence in $H^1(B_1)$, we observe that as $j\to+\infty$
				$$
				\int_{B_1}|\nabla w_j|^2\,dx-\int_{B_1}|\nabla ((1-\zeta)w_j)|^2\,dx
				=\int_{B_1}(2\zeta-\zeta^2)|\nabla w_j|^2\,dx+o(1).
				$$ 
                Then, since $w_j$ is bounded in $H^1(B_1)$, we get $$\int_{B_1}|\nabla ((1-\zeta)w_j)|^2\,dx-(1-\eps_j)\int_{B_1}|\nabla w_j|^2\,dx\le C\eps_j+\int_{B_1}\zeta^2|\nabla w_j|^2\,dx+o(1).$$
                Then, by using the competitor $h:=(1-\zeta)z_j+\zeta b_j$ as in Step 1, the same computation shows that
				\bea \int_{B_1}&(\zeta-\eps_j)\bigg(\frac{f_j(z_j-b_j)}{\eta_j^2}+\frac{Q_j(z_j-b_j)}{\eta_j^2}\bigg)\,dx\le C\eps_j+\int_{B_1}\zeta^2|\nabla w_j|^2\,dx+o(1).
				\eea
				First, by exploiting \eqref{eq:proof-step-1-cylindrical} and the homogeneity of the terms involved, we get \bea \int_0^1& \zeta(r) r^{d+2\beta-3}\,dr\int_{\partial B_1}\bigg(\frac{f_j(z_j-b_j)}{\eta_j^2}+\frac{Q_j(z_j-b_j)}{\eta_j^2}\bigg)\,d\HH^{d-1}\\&\qquad\le C\eps_j+\int_0^1 \zeta(r)^2 r^{d+2\beta-3}\,dr\int_{\partial B_1}|\nabla w_j|^2\,d\HH^{d-1}+o(1),
				\eea as $j\to+\infty$. Finally, since  $w_j$ is bounded in $H^1(B_1)$, by choosing $\zeta(r):=\tau\zeta_0(r)$, where $\zeta_0\in C^\infty_c(B_1)$ is a radial cut-off function satisfying $0\le\zeta_0\le 1$, we obtain the claim, up to replacing $C\tau$ with $\tau$.
				\\ \vspace{-0.3cm}
				
				\noindent\textit{Step 5.} In this step, we prove that for every $\tau>0$
				\be\label{eq:step5-cylindrical}\frac{|\xi_j|^{2\beta-1}}{\eta_j^2}\le \tau,\ee for every $j\geq j_0(\tau)$ large enough. This essentially follows by \cref{lemma:negative-b}.
				Indeed, choosing $h=z_j$ in \eqref{eq:contr-hyp1}, we get $W(z_j)-W(b)\ge0$, which implies that $W(h)-W(b)\ge (1-\eps_j)(W(z_j)-W(b))\ge0$ for every admissible competitor $h$. In particular, since $\xi_j\to0$, by \cref{lemma:negative-b}  we have \be\label{eq:proof-one-eq1}\frac{\kappa_{\beta}}2|\xi_j|^{2\beta-1}\le W(b)-W(b_j)\le W(h)-W(b_j).
                \ee 
                On the other hand, since $w_j\to 0$ weakly in $H^1(B_1)$, strongly in $L^2(B_1)$, we get 
				$$\int_{B_1}|\nabla ((1-\zeta)w_j)|^2\,dx
				=\int_{B_1}(1-\zeta)^2|\nabla w_j|^2\,dx+o(1),\qquad \text{as $j\to +\infty$}.
                $$
                Now, we consider the same competitor of Step 1, i.e., $h=(1-\zeta)z_j+\zeta b_j,$ with $\zeta\in C^\infty_c(B_1)$ radial such that $0\le\zeta\le1$. Arguing as in Step 1, for $j$ large enough we obtain \be\label{eq:proof-one-eq2} \frac{W(h)-W(b_j)}{\eta_j^2}\le\int_{B_1} (1-\zeta)^2|\nabla w_j|^2\,dx+\int_{B_1}(1-\zeta)\bigg(\frac{f_j(z_j-b_j)}{\eta_j^2}+\frac{Q_j(z_j-b_j)}{\eta_j^2}\bigg)\,dx+o(1).
				\ee
                Combining \eqref{eq:proof-one-eq1} and \eqref{eq:proof-one-eq2}, and choosing $ \zeta(r)\equiv 1$ in $(0,1-\rho)$, we have
				\bea \frac{\kappa_{\beta}}2\frac{|\xi_j|^{2\beta-1}}{\eta_j^2}&\le \int_{1-\rho}^1(1-\zeta(r))^2r^{d+2\beta-3}\,dr\int_{\partial B_1}|\nabla w_j|^2\,d\HH^{d-1}\\&\qquad+\int_{1-\rho}^1(1-\zeta(r))r^{d+2\beta-3}\,dr\int_{\partial B_1}\bigg(\frac{f_j(z_j-b_j)}{\eta_j^2}+\frac{Q_j(z_j-b_j)}{\eta_j^2}\bigg)\,d\HH^{d-1}+o(1)
				\eea
				where we used the homogeneity of the terms involved. 
                By using \eqref{eq:proof-step-1-cylindrical} with the boundedness of $w_j$ in $H^1(B_1)$, we get $$\frac{\kappa_{\beta}}2\frac{|\xi_j|^{2\beta-1}}{\eta_j^2}\le C\int_{1-\rho}^1r^{d+2\beta-3}\,dr+o(1),$$ 
                which concludes the proof of \eqref{eq:step5-cylindrical} once we choose  $\rho$ small and then $j$ large enough, possibly depending on $\tau$.
				\\ \vspace{-0.3cm}
				
				\noindent\textit{Step 6.} Next we prove that for every $\tau>0$
				\be\label{eq:step6-cylindrical}\left|\int_{B_1} \frac{f_j(z_j-b_j)}{\eta_j^2}\,dx\right|\le \tau,\ee for every $j\geq j_0(\tau)$ large enough. 
                This is the fundamental step in order to prove that the first variation around $b_j$ vanishes in the limit $j\to \infty$. 
                
                Since $Q_j(z_j-b_j)\ge0$, from \eqref{eq:step4-cylindrical} we infer that  $\int_{B_1} f_j(z_j-b_j)\,dx \le \tau\eta_j^2$. Thus, we only need to show that $ -\int_{B_1} f_j(z_j-b_j)\,dx \le \tau\eta_j^2$.
				
				Set $D_j:=\{b_{j,1}>0\}\cap\{b_{j,2}>0\}\cap B_1$. Since outside $D_j$ we have $f_j\equiv0$, then by \eqref{def:fj} \be\label{eq:fjbound}0\le -f_j\le C(b_{j,1}^{\gamma-1}+b_{j,2}^{\gamma-1})\chi_{D_j}\le Cb_j^{\gamma-1}\chi_{D_j}.\ee  
				On the other hand, since $D_j \subset \{|x_d|\leq x\cdot \xi_j\}\cap B_1$, 
                we have $|D_j|\le C|\xi_j|$. Now, using that $b_j\le C|\xi_j|^{\beta}$ in $D_j$, we have \be\label{eq:eq1-cylindrical} \int_{D_j}b_j^\gamma\,dx\le C|\xi_j|^{2\beta-1}.
				\ee
				
				Now we proceed by showing that for every $\sigma>0$ there exists $C_\sigma>0$ such that for every $b\ge0, t\ge-b$, we have \be\label{eq:claim-cylindrical}b^{\gamma-1}t^+\le \sigma Q_b(t)+C_\sigma b^{\gamma},\quad\text{where}\quad Q_b(t)=(b+t)^\gamma-b^\gamma-\gamma b^{\gamma-1}t.\ee
				This claim is trivial if either $t\le0$ or $t\ge0$ and $b=0$. Thus, suppose that $t\ge0, b>0$ and define $s:=t/b$, then the claim \eqref{eq:claim-cylindrical} is equivalent to showing that $$s\le \sigma((1+s)^\gamma-1-\gamma s)+C_\sigma.$$ This statement is proved once we notice that the function $s\mapsto g(s):=s-\sigma((1+s)^\gamma-1-\gamma s)$ is continuous, $g(0)=0$ and satisfies $g(s)\to-\infty$ as $s\to+\infty$. This observation concludes the proof of the claim \eqref{eq:claim-cylindrical}.
				
				By combining \eqref{eq:claim-cylindrical} together with \eqref{eq:fjbound}, and by using that $z_j-b_j\ge -b_j$, we get \bea -f_j(z_j-b_j)\le-f_j(z_j-b_j)^+\le Cb_j^{\gamma-1}(z_j-b_j)^+\ind_{D_j}\le C\sigma Q_j(z_j-b_j)+C_\sigma b_j^{\gamma}\ind_{D_j},\eea where in the last inequality we also used that $Q_j(z_j-b_j)\ge0$.
				
				Integrating in $B_1$ and using \eqref{eq:eq1-cylindrical}, we have \bea -\int_{B_1}f_j(z_j-b_j)\,dx&\le C\sigma \int_{B_1} Q_j(z_j-b_j)\,dx+C_\sigma|\xi_j|^{2\beta-1}\\&\le -C\sigma \int_{B_1}f_j(z_j-b_j)\,dx +C\tau \eta_j^2+C_\sigma \tau \eta_j^2,\eea where the second inequality follows by using \eqref{eq:step4-cylindrical} and \eqref{eq:step5-cylindrical}. Choosing first $\sigma$ small enough and then replacing $C\tau+C_\sigma\tau$ with $\tau$, we obtain \eqref{eq:step6-cylindrical}.
				\\ \vspace{-0.3cm}
				
				\noindent\textit{Step 7.} Finally, we conclude the proof by showing that $w_j \to 0$ strongly in $H^1(B_1)$, which is a contradiction with the normalization $\|w_j\|_{H^1(B_1)}=1$. Let $\zeta\in C^\infty_c(B_1)$ be a radial function with $0\le\zeta\le1$, and consider $h:=(1-\zeta)z_j+\zeta b_j$ as an admissible competitor in \eqref{eq:contrad-hyp3}. Thus, as in Step 1, we get 
				\bea \int_{B_1}&|\nabla  w_j|^2\,dx+\int_{B_1}\bigg(\frac{f_j(z_j-b_j)}{\eta_j^2}+\frac{Q_j(z_j-b_j)}{\eta_j^2}\bigg)\,dx\\&\qquad\le C\eps_j+ \int_{B_1}\Big(|\nabla \zeta|^2w_j^2+(1-\zeta)^2|\nabla w_j|^2-2w_j(1-\zeta)\nabla \zeta\cdot\nabla w_j\Big)\,dx\\&\qquad\qquad+\int_{B_1}(1-\zeta)\bigg(\frac{f_j(z_j-b_j)}{\eta_j^2}+\frac{Q_j(z_j-b_j)}{\eta_j^2}\bigg)\,dx,
				\eea for some $C>0$ independent of $j$.
				Therefore $$\int_{B_1}(1-(1-\zeta)^2)|\nabla w_j|^2\,dx+\int_{B_1}\zeta\bigg(\frac{f_j(z_j-b_j)}{\eta_j^2}+\frac{Q_j(z_j-b_j)}{\eta_j^2}\bigg)\,dx\le C\eps_j+o(1),$$ as $j\to+\infty$.
				Since $Q_j(z_j-b_j)\ge0$, by choosing a radial function $\zeta$ such that $\zeta\equiv1$ on $B_{1/2}$ and using \eqref{eq:step6-cylindrical}, we get 
				$$\int_{B_{1/2}}|\nabla w_j|^2\,dx\le C\left|\int_{B_1}\frac{f_j(z_j-b_j)}{\eta_j^2}\,dx\right|+o(1)\le C\tau+o(1),\qquad \text{as $j\to+\infty$}.$$ In the previous estimate, we also exploit the homogeneity of the terms in order to remove $\zeta$ from the integral. Moreover, by using the homogeneity of $w_j$, we can replace the integral on $B_{1/2}$ in the left hand side with the same integral on $B_1$, up to a multiplicative constant. Since $\tau>0$ was arbitrary, the conclusion follows.
			\end{proof}
			\begin{remark}[Non-translationality of $b_{\text{one}}$]\label{remark:non-translation-one-dim-cone} By a slight modification of \cref{lemma:ker-1d}, we can prove that the only $(\beta-1)$-homogeneous solutions of the linearized operator $-\Delta+\frac\gamma2(\gamma-1)b^{\gamma-2}$ are $|x_d|^{\beta-1}$ and $\text{sgn}(x_d)|x_d|^{\beta-1}$. While $\text{sgn}(x_d)|x_d|^{\beta-1}$ is a derivative of $b_{\text{one}}$, the mode $|x_d|^{\beta-1}$ does not come from a partial derivative, and gives the non-translationality of $b_{\text{one}}$.
            \end{remark}
            \begin{remark}[Non-integrability of $b_{\text{one}}$]\label{remark:non-int-one-dim-cone}
				By a slight refinement of the epiperimetric inequality in \cref{prop:epi-one-dim-cone}, we can show that $b_{\text{one}}$ is not integrable, in the sense of \cref{def:integrability}.
                
                We already know that $b_{\text{one}}$ is not integrable through rotations. Indeed, by \cref{lemma:ker-1d}, the kernel of $L_{b_{\text{one}}}$ contains the functions $|x_d|^{\beta-1} x_i$, which are not generated by rotations. 
                
                Let us now prove the non-integrability of $b_{\text{one}}$ with respect to a generic family of solutions. We first observe that, in the proof of \cref{prop:epi-one-dim-cone}, the strict inequality in the contradiction hypothesis \eqref{eq:contr-hyp1} was used only to ensure that $\eta_j>0$. Hence, if in \cref{prop:epi-one-dim-cone} we add the assumption $z\not \in \{b_{\xi,\omega}\}_{\xi,\omega\in\R^{d-1}}$,
                then the same contradiction argument yields a strict epiperimetric inequality: there exists a competitor $h$ such that 
                $$W(h)-W(b_{\text{one}})< (1-\eps)(W(z)-W(b_{\text{one}})).
                $$
                Assume now that there exists a family of solutions $\psi_t$ such that $\delta\mathcal{G}(\psi_t)=0$ and $\psi_0=0$. Set $v_t:=r^\beta b_{\text{one}}+r^\beta \psi_t$. Then $v_t$ is a solution of the Alt-Phillips problem and, since $\gamma>1$, it is also a minimizer, see \cref{prop:convex-functional}. Suppose by contradiction that, for some sufficiently small $t\neq 0$, $v_t\not \in \{b_{\xi,\omega}\}_{\xi,\omega\in\R^{d-1}}.$
                Since $\delta\mathcal{G}(\psi_t)=0$, the map $t\mapsto \mathcal{G}(\psi_t)$ is constant, and so $W(v_t)=W(v_0)=W(b_{\text{one}})$. The strict epiperimetric inequality then gives a competitor $h_t\in H^1(B_1)$, with $h_t=v_t$ on $\partial B_1$, such that 
                $$W(h_t)-W(b_{\text{one}})< (1-\eps)(W(v_t)-W(b_{\text{one}}))=0.$$ 
                Thus
                $W(h_t)<W(b_{\text{one}})=W(v_t)$, contradicting the minimality of $v_t$. 
                
                Therefore $v_t \in \{b_{\xi,\omega}\}_{\xi,\omega\in\R^{d-1}}$, for every $t$ small enough. On the other hand, since $b_{\xi,\omega}$ is a solution if and only if $\xi=0$, we must have $v_t=b_{0,\omega(t)}$ for some $\omega(t)$ with $\omega(0)=0$. Consequently, by \eqref{e:useful-for-notintegrable}, we have  $\partial_t v_t\big|_{t=0}\in \text{span}\{\text{sgn}(x_d)|x_d|^{\beta-1}x_i\}_{i=1}^{d-1}$. Thus, it follows that the directions $\text{span}\{|x_d|^{\beta-1}x_i\}_{i=1}^{d-1}$ cannot be generated by a family of solutions of this type, and so $b_{\text{one}}$ is not integrable.                \end{remark}
			
			\section{The translational cones}\label{section-translational}
			In this section, we begin the study of cylindrical extensions of cones in the Alt-Phillips problem, for $\gamma>1$. More precisely, given $b\in\mathcal B_\ell$, with $\ell=1,\ldots,d-2$, we write
$$
b(y,z)=B(y),\quad (y,z)\in\R^{d-\ell}\times \R^\ell,
$$
where $B$ is a cone in $\R^{d-\ell}$ satisfying $B>0$ in $\R^{d-\ell}\setminus\{0\}$, that is, $B\in\mathcal B_0$ in dimension $d-\ell$. We point out that the case $\ell=0$ was already treated in \cref{section-positive}, while the case $\ell=d-1$ corresponds to the one-dimensional cone $b_{\text{one}}$, studied in \cref{section-1d}.

            We recall the spherical linearized operator $L_{b}$ and its kernel $K_b:=\ker(L_{b})$, defined in \eqref{eq:linearized-operator} and \eqref{eq:general-kernel} respectively. Throughout the section, we use the analogous notation $L_B$ and $K_B$ for the cone $B$. Notice that $K_B\subset K_b$, under the natural cylindrical identification, and thus the following decomposition holds 
            \be\label{eq:Kbz-decomposition}K_b=K_B\oplus K_b^z.\ee
            Here $K_b^z$ denotes the complement of the  kernel $K_B$ in $K_b$. The main results of this section can be summarized as follows. 
            \begin{itemize}
                \item[(i)] We characterize $K_b^z$ in terms of the $(\beta-1)$-homogeneous solutions of the linearized operator $L_B$ (see \cref{prop:kernel-cylindric}). This explains the role of the translational cones, defined in \cref{def:beta-1translational}.
                \item[(ii)] If a cone $B$ is translational, then $K_b^z$ is generated by rotations (see \cref{corollary:corollary1}). This is a crucial point in the proof of the epiperimetric inequality in \cref{section:intermediate-epiperimetric}.
                \item[(iii)] If $b$ is integrable through rotations (see \cref{def:integrability-through-rotations}), then $b$ is translational, see \cref{corollary:corollary1}.
                \item[(iv)] The cylindrical cones $b_\ell$ defined in \eqref{eq:parabolas} are translational, for $\ell=0,\ldots,d-2$, and we have a complete characterization of the kernel $K_{b_\ell}$, see \cref{prop:radial-cone-is-int-traslations} and \cref{corollary:kernel-cylindrical}.
                \item[(v)] In some regimes of $d$ and $\gamma$, we can verify the translational hypothesis, even if the cones are not explicitly known, see \cref{prop:trans0} and \cref{rem:translational-bif}.
                \item[(vi)] Finally, we show some integrability through rotations results in dimension $d=2,3$, see \cref{corollary:integrability2}.
            \end{itemize}
            \subsection{Kernel of cylindrical extensions}
			In the following proposition we characterize the $z$-component $K_b^z$ of the kernel in \eqref{eq:Kbz-decomposition} in terms of the eigenspace $E_B^{\beta-1}$ defined in \cref{def:space-Ebmu}. We also show that the translational condition is stable under cylindrical extensions, namely $E_b^{\beta-1}=E_B^{\beta-1}$, under the natural cylindrical identification. 
			\begin{proposition}\label{prop:kernel-cylindric}
			    Let $\gamma\in(1,2)$ and $b\in\mathcal B_\ell$, for some $\ell=1,\dots,d-2$. Then $$K_b=K_B\oplus K_b^z,\quad\text{where}\quad K_b^z=\textrm{span}\left\{z_j|y|^{\beta-1}\psi\left(\frac{y}{|y|}\right):\ \psi\in E^{\beta-1}_B,\ j=1,\ldots,\ell\right\}.$$ Moreover $E_b^{\beta-1}=E_B^{\beta-1}$.
			\end{proposition}
            The proof of \cref{prop:kernel-cylindric} follows from the general decomposition of the eigenspaces $E_b^{\sigma}$ proved in the next lemma.
			\begin{lemma}\label{lemma:kernel-cylindric}
				Let $\gamma \in (0,2)$ and $b\in\mathcal B_\ell$, for some $\ell=1,\dots,d-2$. Consider the following cylindrical coordinates: given $\theta \in \mathbb{S}^{d-1}$, set
                $$
                \theta=\left(\omega\sin\alpha,\eta \cos\alpha\right),
                \quad\text{where}\quad \omega\in \mathbb{S}^{d-\ell-1},\ \eta\in \mathbb{S}^{\ell-1},\ \alpha\in(0,\pi/2). 
                $$
                Then, for every $\sigma \in\R$ the following decomposition holds true
                \be\label{eq:block-decomposition}E_b^\sigma=\bigoplus_{m\ge0}\ \bigoplus_{p\ge0}\text{span}\{q^\sigma_{m,p}\}\otimes\mathcal H_m(\mathbb{S}^{\ell-1})\otimes E_B^{\sigma-(m+2p)},\ee
where, if we denote by $P_p^{(k_1,k_2)}$ a Jacobi polynomial of degree $p\in \N_{\ge0}$, we set \be\label{eq:defqmpsigma}q_{m,p}^\sigma(\alpha):=(\sin\alpha)^{\sigma-(m+2p)}(\cos\alpha)^m P_p^{({\sigma-(m+2p)}+\frac{d-\ell}2-1,m+\frac\ell2-1)}(\cos2\alpha).\ee
Moreover, if $\gamma\in (1,2)$, then $E^{\mu}_B=\{0\}$ for every $\mu\le\beta-2$.
			\end{lemma}
			\begin{proof}
				In the new coordinates $\theta=\left(\omega\sin\alpha,\eta \cos\alpha\right)$ the spherical Laplacian $\Delta_{\theta}$ is given by
				$$\Delta_\theta=\partial_{\alpha\alpha}+\big((d-\ell-1)\cot\alpha-(\ell-1)\tan\alpha\big)\partial_\alpha+\frac1{\sin^2\alpha}\Delta_\omega+\frac{1}{\cos^2\alpha}\Delta_\eta.$$
				Then, the linearized operator $L_b$ in $\mathbb{S}^{d-1}$ can be written as
				$$L_b=-\partial_{\alpha\alpha}-\Big((d-\ell-1)\cot\alpha-(\ell-1)\tan\alpha\Big)\partial_\alpha+\frac{1}{\sin^2\alpha}\big(L_B+\lambda_{d-\ell}(\beta)\big)-\frac{1}{\cos^2\alpha}\Delta_\eta-\lambda_d(\beta),$$ where we used that, since $b(\alpha,\omega,\eta)=(\sin\alpha)^\beta B(\omega)$, then $b^{\gamma-2}=(\sin\alpha)^{-2}B^{\gamma-2}$.
				
				Thus, the operator $L_b$ is diagonal with respect to the product decomposition of the variables: it involves the $\eta$-variable only through $-\Delta_\eta$, the
$\omega$-variable only through $L_B$, and the remaining variable only through
derivatives in $\alpha$. Hence, the variables can be separated and every $\phi\in E_b^\sigma$ can be written as a finite sum of functions of the form
$\phi(\alpha,\omega,\eta)=q(\alpha)\,\psi(\omega)\,Y(\eta),$
where
$$
        \psi\in E_B^\mu,
        \quad\text{i.e., } 
        L_B\psi=
        \bigl(\lambda_{d-\ell}(\mu)-\lambda_{d-\ell}(\beta)\bigr)\psi\quad\text{on }\mathbb{S}^{d-\ell-1},
$$
and $Y\in \mathcal{H}_{m}(\mathbb{S}^{\ell-1})$, i.e., $
        -\Delta_\eta Y=\lambda_\ell(m)Y$ 
        on $\mathbb S^{\ell-1}$. Therefore, a direct substitution of $\psi$ and $Y$ gives
				$$-q''-\Big((d-\ell-1)\cot\alpha-(\ell-1)\tan\alpha\Big)q'+\frac{\lambda_{d-\ell}(\mu)}{\sin^2\alpha}q+\frac{\lambda_\ell(m)}{\cos^2\alpha}q=\lambda_d(\sigma)\,q.$$
				Using the ansatz $q(\alpha)= (\sin\alpha)^\mu (\cos\alpha)^m h(\cos 2\alpha)$, we obtain a Jacobi-type equation for $h$
                $$(1-t^2)h''+\left(b_m-a_\mu-(a_\mu+b_m+2)t\right)h'+\Lambda h=0,$$ where
                $$
                a_\mu:=\mu+\frac{d-\ell}2-1,\qquad b_m:=m+\frac\ell2-1\quad\text{and}\quad \Lambda:=\frac{\lambda_d(\sigma)-\lambda_d(\mu+m)}4.
                $$
                The $H^1$-admissible solutions occur precisely when $h$ is a Jacobi polynomial $h(t)=P_p^{(a_\mu,b_m)}(t)$ and $\Lambda=p(p+a_\mu+b_m+1)$, for some $p\in\N_{\ge0}$. This last condition is equivalent to requiring that $\mu=\sigma-(m+2p)$. Then, $q(\alpha)$ must be of the form $q_{m,p}^\sigma(\alpha)$, as in \eqref{eq:defqmpsigma}. This concludes the proof of the decomposition \eqref{eq:block-decomposition}.
				
				In order to conclude the proof, we need to check that for $\gamma\in(1,2)$, we have $E^{\mu}_B=\{0\}$ for every $\mu\le\beta-2$.
                If $\widetilde L_B=L_B+\lambda_{d-\ell}(\beta)$ and $\psi=B^{\gamma-1}>0$, then on the sphere $\mathbb{S}^{d-\ell-1}$, we have $$\widetilde L_B\psi=(\gamma-1)\lambda_{d-\ell}(\beta)B^{\gamma-1}+(\gamma-1)(2-\gamma)B^{\gamma-3}|\nabla_\omega B|^2\ge(\gamma-1)\lambda_{d-\ell}(\beta)\psi, $$
                where we used that $\gamma>1$.
                Then, if we denote by $\mu_1(\widetilde L_B)$ the first eigenvalue of $\widetilde L_B$, we infer that $\mu_1(\widetilde L_B)\ge (\gamma-1)\lambda_{d-\ell}(\beta)$. Therefore, if $E^{\mu}_B\not=\{0\}$ for some $\mu\le\beta-2$, then $\lambda_{d-\ell}(\mu)$ would be an eigenvalue of $\widetilde L_B$, and thus $\lambda_{d-\ell}(\mu)\ge \mu_1(\widetilde L_B)$. Combining the above two inequalities, we get $$\lambda_{d-\ell}(\mu)\ge (\gamma-1)\lambda_{d-\ell}(\beta).$$ 
				On the other hand, since $\mu\le \beta-2$, then $$\lambda_{d-\ell}(\mu)\le \lambda_{d-\ell}(\beta-2)=(\beta-2)(\beta+d-\ell-4)<(\beta-2)(\beta+d-\ell-2)=(\gamma-1)\lambda_{d-\ell}(\beta),$$ where we used that $\mu>1-\frac{d-\ell}2$, by $H^1$-admissibility.
				This is the desired contradiction. 
				\end{proof}
            \begin{proof}[Proof of \cref{prop:kernel-cylindric}] The result is a direct consequence of \cref{lemma:kernel-cylindric}. Indeed, since $q^\beta_{0,0}(\alpha)=(\sin\alpha)^\beta$ and $q_{1,0}^\beta(\alpha)=(\sin\alpha)^{\beta-1}\cos\alpha$, then, by taking $\sigma=\beta$ in the decomposition \eqref{eq:block-decomposition}, we see that only the cases $m+2p=0$ and $m+2p=1$ survive, i.e., 
            \begin{align*}
            E^\beta_b &= q^\beta_{0,0}\otimes\mathcal{H}_0(\mathbb{S}^{\ell-1})\otimes E^\beta_B
            \,\oplus\, q^\beta_{1,0}\otimes\mathcal{H}_1(\mathbb{S}^{\ell-1})\otimes E^{\beta-1}_B\\
            &= (\sin \alpha)^\beta  E^\beta_B
            \,\oplus\, (\sin\alpha)^{\beta-1}\cos\alpha\otimes\mathcal{H}_1(\mathbb{S}^{\ell-1})\otimes E^{\beta-1}_B
            \end{align*}
            Since $K_b=r^\beta E_b^\beta$ and $K_B= r^\beta (\sin \alpha)^\beta  E^\beta_B$, we infer that 
            $$
            K_b = K_B \,\oplus \,  r^\beta(\sin\alpha)^{\beta-1}\cos\alpha\otimes\mathcal{H}_1(\mathbb{S}^{\ell-1})\otimes E^{\beta-1}_B
            $$
            and the last expression gives $K_b^z$. On the other hand, let $\psi\in E_B^{\beta-1}$ and choose one of the coordinate spherical harmonics $Y_1(\eta)=\eta_j$, for some $j=1,\ldots,\ell$, in $\mathcal{H}_1(\mathbb{S}^{\ell-1})$. Then, we can write
$
r^\beta q_{1,0}^\beta(\alpha)\psi(\omega)Y_1(\eta)
=
z_j |y|^{\beta-1}\psi(\omega),
$ as we claimed.

                Finally, if we choose $\sigma=\beta-1$ in \eqref{eq:block-decomposition}, then only the cases $m+2p=0$ survive, and thus $E_b^{\beta-1}=E_B^{\beta-1}$, concluding the proof.
            \end{proof}
           As an immediate corollary of \cref{prop:kernel-cylindric}, we get the following result.
            \begin{corollary}\label{corollary:corollary1}
				Let $\gamma\in (1,2)$ and $b\in\mathcal{B}_\ell$ for some $\ell=1,\dots,d-2$. Then, the following conditions are equivalent:
                \begin{itemize}
                    \item $b$ is translational;
                    \item $B$ is translational;
                    \item $K^z_b$ is generated by mixed rotations, namely
				\be\label{eq:Krot}K^z_b=\text{span}\{z_j\partial_{y_i}B:\ i=1,\ldots, d-\ell, \ j=1,\ldots,\ell\}.\ee
                \end{itemize}  
                In particular, if $b$ is integrable through rotations, then $b$ is translational. 
			\end{corollary}
			\begin{proof}
				The first equivalence follows from the identity $E_b^{\beta-1}=E_B^{\beta-1}$, proved in \cref{prop:kernel-cylindric}, together with the relations $\partial_{z_j}b=0$ and $\partial_{y_i}b=\partial_{y_i}B$. The second equivalence follows from the observation that $b$ is translational if and only if
$
E_B^{\beta-1}=\text{span}\{\partial_{y_i}B:\ i=1,\ldots,d-\ell\},
$
which is in turn equivalent to
$
K_b^z=\text{span}\{z_j\partial_{y_i}B:\ i=1,\ldots,d-\ell,\ j=1,\ldots,\ell\}.
$

                The last statement follows from the fact that if $b$ is integrable through rotations, then every element in $K_b^z$ is generated by rotations. Since rotations in the $y$-variables belong to $K_B$ and rotations in the $z$-variables are trivial, the elements of $K_b^z$ must be generated by mixed rotations.
			\end{proof}

            We point out that, in the previous corollary, the assumption $\ell\not=0$ is used to recover information on $E_B^{\beta-1}$ from $K_b^z$, by using \cref{lemma:kernel-cylindric}, which is absent when $\ell=0$.
        We also have the following characterization. 
			\begin{corollary}\label{corollary:corollary3}
				Let $\gamma\in (1,2)$ and $b\in\mathcal{B}_\ell$ with $\ell=1,\dots,d-2$. Then $b$ is integrable through rotations if and only if $B$ is translational and integrable through rotations.
			\end{corollary}
			\begin{proof}
				By \cref{prop:kernel-cylindric}, $b$ is integrable through rotations if and only if both $K_B$ and $K_b^z$ are generated by rotations. The first condition is equivalent to saying that $B$ is integrable through rotations, while the second condition is equivalent, by \cref{corollary:corollary1}, to saying that $b$, and thus $B$, is translational.
			\end{proof}
			\subsection{Kernel of the parabola cones}
			We can use \cref{prop:kernel-cylindric} to describe the kernel of the linearized operator $L_b$ for the family of parabola cones $b_\ell$ defined in \eqref{eq:parabolas}. In particular, we obtain that $b_\ell$ is translational for $\ell=0,\ldots,d-2$ and $\gamma>1$.
			
			First we prove that the radial cone is translational.
			\begin{lemma}\label{prop:radial-cone-is-int-traslations}
				Let $\gamma\in(0,2)$, then the radial cone $b_{\text{rad}}:=c_{\text{rad}}|x|^\beta$ is translational.
			\end{lemma}
			\begin{proof}
				As already observed in \eqref{eq:linearized-operator-radial}, we have $L_{b_{\text{rad}}}=-\Delta_\theta-(2-\gamma)\lambda(\beta)$. Moreover, since $\lambda(\beta-1)-(\gamma-1)\lambda(\beta)=d-1$, then $E^{\beta-1}_{b_{\text{rad}}}=\mathcal{H}_1(\mathbb{S}^{d-1})$ and it is generated by  coordinate spherical harmonics of the form $\theta_j$, with $j=1,\ldots,d$. On the other hand, the $(\beta-1)$-homogeneous extension of $\theta_j$ is given by $|x|^{\beta-2}x_j$, 
which coincides, up to a multiplicative constant, with $\partial_{x_j}b_{\mathrm{rad}}$, as desired.
			\end{proof}
			
			The following corollary follows by combining the characterization of the kernel of $L_{b_{\text{rad}}}$ in $\mathbb{R}^{d-\ell}$, established in \cref{prop:optimality-log}, with the cylindrical extension argument of \cref{prop:kernel-cylindric}. We recall the values $\gamma_{k,d}$ defined in \eqref{eq:def-gammakd}.
			\begin{corollary}\label{corollary:kernel-cylindrical}
				Let $\gamma\in (1,2)$ and $b_\ell\in\mathcal{B}_\ell$ be a parabola cone as in \eqref{eq:parabolas}, for some $\ell=1,\ldots,d-2$. 
				\begin{itemize}
					\item[(i)] If $\gamma\not=\gamma_{k,d-\ell}$ for every $k\in\N_{\ge3}$, then $$K_{b_\ell}=\text{span}\{|y|^{\beta-2}y_iz_j,\  i=1,\ldots,d-\ell,\ j=1,\ldots,\ell\}.$$
					\item[(ii)] If $\gamma=\gamma_{k,d-\ell}$ for some $k\in\N_{\ge3}$, then $$K_{b_\ell}=|y|^{\beta}\mathcal{H}_k(\mathbb{S}^{d-\ell-1})\oplus \text{span}\{|y|^{\beta-2}y_iz_j,\  i=1,\ldots,d-\ell,\ j=1,\ldots,\ell\},$$ where $\mathcal{H}_k(\mathbb{S}^{d-\ell-1})$ is the space of spherical harmonics of degree $k$ on $\mathbb{S}^{d-\ell-1}$.
				\end{itemize}
				In particular $b_\ell$ is integrable through rotations if and only if $\gamma\not=\gamma_{k,d-\ell}$ for every $k\in\N_{\ge3}$. On the other hand, $b_\ell$ is translational for every $\gamma\in (1,2)$.
			\end{corollary}
            \begin{proof}
                The result follows by combining \cref{prop:optimality-log}, \cref{prop:kernel-cylindric} and \cref{prop:radial-cone-is-int-traslations}.
            \end{proof}
            We stress that \cref{corollary:kernel-cylindrical} is also valid for $\ell=0$ (see \cref{prop:optimality-log}), whereas the case $\ell=d-1$ is different for at least two reasons. First, by \cref{lemma:ker-1d}, the kernel of $L_{b_{\text{one}}}$ is not sensitive to the value $\gamma_{k,d}$. Secondly, by \cref{remark:non-translation-one-dim-cone}, the one-dimensional cone is not translational.
\subsection{Translationality in low dimensions.}
In this subsection, we prove that in some regimes the translationality property can be verified directly, even if we do not have a characterization of the cones.
\begin{proposition}\label{prop:trans0}
	Let $b\in\mathcal{B}_\ell$, for $\ell=0,\ldots,d-2$, and suppose that \bea\text{either}\qquad d=2,3\,\text{ and }\, \gamma \in (1,2)\qquad\text{or}\qquad d\ge 4\,\text{ and }\,\gamma\in\left(1,\frac32+\frac{1}{2d}\right).\eea
	Then $b$ is translational.
\end{proposition}

\begin{lemma}\label{lemma:trans1}
	Let $d=3$, $\gamma \in (1,2)$ and $b\in\mathcal{B}_0$. Then $b$ is translational.
\end{lemma}
\begin{proof}
	Let $v(r,\theta)=r^{\beta-1}\psi(\theta)$ be the $(\beta-1)$-homogeneous extension of a trace $\psi\in E_b^{\beta-1}$. We need to show that $v$ is a partial derivative of $b$.
	
	By the convexity of $b$ in \cref{prop:convexity-of-blowups}, we have that $D^2 b>0$ in $\R^3\setminus\{0\}$. Then $y=\nabla b(x)$ is a diffeomorphism in $\R^3\setminus\{0\}$ and we set $w(y):=v(x)$, where $y:=\nabla b(x)$. Notice that, since $v$ is $(\beta-1)$-homogeneous and $b$ is $\beta$-homogeneous, the function $w$ is necessarily $1$-homogeneous. Since $d=3$, we also observe that $w\in W^{2,2}_{\text{\loc}}(\R^3)$, indeed $$\int_{B_\eps}|D^2 w|^2\,dy\le C\int_{B_\eps} \frac1{|y|^2}\,dy<+\infty.$$
    The proof will be complete once we show that $w$ satisfies an equation of the form
\be\label{eq-for-w-a}
    \text{tr}(A D^2 w)=0
    \qquad\text{in }\mathbb R^3\setminus\{0\},
\ee
for some uniformly elliptic matrix $A=(a_{ij})\in \R^{3\times 3}$. Indeed, since $w$ is
$1$-homogeneous and $w\in W^{2,2}_{\text{loc}}(\mathbb R^3)$, we can then infer from
\cite{nadirashvili} that $w$ is linear. Hence there exists $a\in\mathbb R^3$ such that
$w(y)=y\cdot a$. In terms of $v$, this gives $v=\nabla b\cdot a$, and the proof follows.

	In order to prove \eqref{eq-for-w-a}, we write $v(x)=w(\nabla b(x))$. For the sake of readability, in the following computation we emphasize the dependence on $x$. First, we have $$\Delta_x v=\Delta_x \big(w(\nabla_x b(x))\big)=\mathrm{tr}\left((D^2_x b)^2 D^2 w\right) +\nabla w\cdot \nabla_x (\Delta_x b(x)).$$ Since $\Delta_x b=\frac\gamma2 b^{\gamma-1}$, the second term in the right-hand side becomes $$\nabla w\cdot \nabla_x(\Delta_x b(x))=\frac\gamma2(\gamma-1)b^{\gamma-2}\nabla w\cdot \nabla_x b=\frac\gamma2(\gamma-1)b^{\gamma-2} \nabla w\cdot y=\frac\gamma2(\gamma-1)b^{\gamma-2} w,$$ where in the last equality we used that $w$ is $1$-homogeneous and $y=\nabla_x b(x)$.
	Combining the above two identities and using that $\Delta_x v=\frac\gamma2(\gamma-1)b^{\gamma-2}v$, we infer that $$\mathrm{tr}\left((D_x^2b)^2 D^2 w(y)\right)=0\quad\text{in }\R^3\setminus\{0\}.$$
    Since $D^2b$ is $(\beta-2)$-homogeneous, if we set $\widetilde A(y):=(D^2b(x))^2$, then $\widetilde A(y)$ is $\lambda$-homogeneous, with $\lambda:=\frac{2(\beta-2)}{\beta-1}$. 
Moreover, since $\widetilde A$ is continuous and positive on the sphere $\partial B_1$, there exists $\Lambda>0$ such that $\widetilde A$ is uniformly elliptic on $\partial B_1$.
    Finally, \eqref{eq-for-w-a} follows by setting $A(y):=|y|^{-\lambda}\widetilde A(y)$.
\end{proof}

\begin{lemma}\label{lemma:trans2}
	Let $b\in\mathcal{B}_0$ and suppose that $\gamma\in\left(1,\frac32+\frac{1}{2d}\right)$, then $b$ is translational.
\end{lemma}
\begin{proof}
	We first recall the following Weyl's inequality: if $L_1$ and $L_2$ are two self-adjoint elliptic operators on $\partial B_1$, then $$\mu_k(L_1+L_2)\ge \mu_k(L_1)+\mu_1(L_2),$$ where $\mu_j(L)$ is the $j$-th eigenvalue of $L$. Applying this inequality with $L_1:=(2-\gamma)(-\Delta_\theta)$, $L_2:= (\gamma-1)(-\Delta_\theta+\frac\gamma2b^{\gamma-2})$ and $k:=d+2$, we obtain that $$\mu_{d+2}\left(-\Delta_\theta+\frac\gamma2(\gamma-1)b^{\gamma-2}\right)\ge (2-\gamma)\mu_{d+2}(-\Delta_\theta)+\mu_1(L_2)=(2-\gamma)2d+(\gamma-1)\lambda(\beta),$$ where we used that $(\gamma-1)b$ is the first eigenfunction of $L_2$, since $b>0$ on $\partial B_1$. By using that $\lambda(\beta-1)-(\gamma-1)\lambda(\beta)=d-1,$  $$\mu_{d+2}\left(-\Delta_\theta+\frac\gamma2(\gamma-1)b^{\gamma-2}\right)-\lambda(\beta-1)\ge (2-\gamma)2d-(d-1)>0$$ for $\gamma<3/2+1/(2d)$.
	
	On the other hand, the operator $-\Delta_\theta+\frac\gamma2(\gamma-1)b^{\gamma-2}$ has at least $d$ eigenfunctions corresponding to the eigenvalue $\lambda(\beta-1)$, which are exactly the partial derivatives of $b$. Adding also the first eigenfunction, the previous inequality implies that there are exactly $d$ eigenfunctions corresponding to the eigenvalue $\lambda(\beta-1)$, concluding the proof. 
 \end{proof}

\begin{proof}[Proof of \cref{prop:trans0}]
	By \cref{corollary:corollary1}, the translationality of a cylindrical extension is equivalent to proving the translationality for the base cone in $\mathcal{B}_0$.
	Then the result follows by combining \cref{prop:dim2-integrability}, \cref{lemma:trans1} and \cref{lemma:trans2}.
\end{proof}
\begin{remark}[Local bifurcations of radial cone are translational ]\label{rem:translational-bif}
By the parameter-dependent Lyapunov-Schmidt reduction \cref{lemma:Lyapunov-Schmidt-bifurcation}, the bifurcating cones $b_{\text{rad},\gamma}$ of \cref{thm:bifurcation-radial} depend continuously on $\gamma$. Hence, by standard perturbation
theory for self-adjoint elliptic operators, the eigenvalues of the linearized operator $L_{b_{\text{rad},\gamma}}$ vary continuously along the bifurcating branches $\gamma\in I_\varepsilon$. 
Since the radial cone $b_{\text{rad},\gamma}$ is
translational for every $\gamma\in(0,2)$, by
\cref{prop:radial-cone-is-int-traslations}, the continuity of the spectrum
implies that, after possibly shrinking the neighborhood of the resonant value $\gamma_{k,d-\ell}$, the local bifurcating cones in \cref{thm:bifurcation-radial} are translational as well.
\end{remark}
\subsection{Integrability and translationality in low dimensions.} We prove the following result concerning the integrability through rotations and the translationality in dimension $d=2$. 
			\begin{proposition}\label{prop:dim2-integrability}
				Let $d=2$ and $\gamma\in (1,2)$. Then, every cone $b\in\mathcal{B}_0$ is translational. Moreover, every non-radial cone $b\in\mathcal{B}_0$ is integrable through rotations.
			\end{proposition}
			\begin{proof}
				We divide the proof in two steps.
				\\ \vspace{-0.3cm}
				
				\noindent\textit{Step 1.} Let us show that $b$ is translational. By translation invariance of the equation, we have $$\partial_{x_1}b,\partial_{x_2}b\in E_b^{\beta-1}.$$ On the other hand, since $b\in\mathcal{B}_0$, the corresponding angular equation is a second-order linear ODE on $\mathbb{S}^1$, and thus its space of solutions is at most two-dimensional. Therefore, the claim follows once we show that $\partial_{x_1}b$, $\partial_{x_2}b$ are linearly independent. If this is not the case, then there would exist $a\in\R^2\setminus\{0\}$, such that $\nabla b\cdot a\equiv0$. Thus $b$ would be one-dimensional, contradicting $b\in\mathcal{B}_0$. This concludes the proof of Step 1.
				\\ \vspace{-0.3cm}
				
				\noindent\textit{Step 2.} Assume that $b\not=b_{\text{rad}}$, then we prove that $b$ is integrable through rotations. 
				Thus, after relabeling as $b$ the trace of the blow-up on the sphere $\mathbb{S}^1$, we have $b>0$, with $b$ non-constant and $b\not\equiv b_*$, where $b_*:=c_{\mathrm{rad}}$. Given $b'(\theta):=\partial_\theta b(\theta)$, the result follows once we prove that $K_b=\text{span}\{b'\}$.
				Since $b'\in K_b$ by rotational invariance of the equation, it is enough to show that there are no other $2\pi$-periodic solutions of $L_b\phi=0$. By definition, we have that $b$ solves the ODE $$b''+f(b)=0\text{ on } \mathbb{S}^1,\quad\text{where}\quad f(s)=\beta^2 s-\frac\gamma 2 s^{\gamma-1}.$$ 
				Set $F(\tau):=\int_0^\tau f(s)\,ds$, and  observe that $b_\ast$ is the unique critical point of $F$ on $(0,\infty)$, i.e., $F'(b_\ast)=0$. Moreover, $F$ is strictly decreasing on $(0,b_\ast)$ and strictly increasing on $(b_\ast,\infty)$. By the conservation of the energy, we have that $E_0=\frac12 (b'(\theta))^2+F(b(\theta))$, for every $\theta \in [0,2\pi]$.
				
				Since $b$ is positive and non-constant, if
				$m:=\min_{\mathbb{S}^1} b$ and $M:=\max_{\mathbb{S}^1} b$, then \be\label{eq:mbastM}0<m<b_\ast<M.\ee
                Indeed, if $b(\theta_0)=m$, evaluating the ODE at $\theta_0$ and using that $b''(\theta_0)\ge0$, we have $f(m)\le0$ and thus $0<m\le b_\ast$. Moreover, if by contradiction $m=b_{\ast}$, then $m=b(\theta_0)=b_{\ast}$ and consequently $b'(\theta_0)=0$. Thus, by uniqueness of the Cauchy problem, $b\equiv b_{\ast}$, which is a contradiction. Applying a similar argument to $M$, we deduce \eqref{eq:mbastM}.
                
				Since at those points where $b=m$ and $b=M$, we have $b'=0$, we infer that $E_0=F(m)=F(M)$. On one hand, since $F(0)=0$, $F$ is strictly decreasing on $(0,b_\ast)$, and $0<m<b_\ast$, we have $$E_0\in (F(b_\ast),0).$$

				On the other hand, since $F(b_\ast)<0$, $F(s)\to+\infty$ as $s\to+\infty$, and $F$ is strictly increasing on $(b_\ast,+\infty)$, there exists a unique $q\in(b_\ast,+\infty)$ such that $F(q)=0$.
				Now, for every $E\in (F(b_\ast),0)$, there exist unique turning values $m(E)\in (0,b_\ast)$, and $M(E)\in (b_\ast,q)$ such that $F(m(E))=F(M(E))=E.$ 
				The minimal period of the corresponding positive periodic orbit is
				$$T(E):={\sqrt{2}}\int_{m(E)}^{M(E)}\frac{1}{\sqrt{E-F(s)}}\,ds.$$
				By \cite{bonorinoetall}, the period function $T(E)$ satisfies $T'(E)>0$ in the whole interval $(F(b_\ast),0)$. In particular, $T'(E_0)\neq 0$.
				
				After a rotation, we may assume that $b(0)=m$, and thus $b'(0)=0$. For $E$ close to $E_0$, let $B:=B(\theta,E)$ be the solution of
				$$\partial_{\theta\theta} B+f(B)=0\text{ on } \R,\quad B(0,E)=m(E),\quad \partial_\theta B(0,E)=0,$$ so that $b(\theta)=B(\theta,E_0).$
				Defining $\psi(\theta):=\partial_E B(\theta,E)\big|_{E=E_0}$, we have $L_b\psi=0$. 
                
                We also observe that, by differentiating $F(m(E))=E$, we have $m'(E_0)\not=0$, and thus $b'$ and $\psi$ are linearly independent, since $b'(0)=0$ and $\psi(0)\not=0$.
				Moreover, as we observed in Step 1, the solution space of the ODE $L_b\phi=0$ is two-dimensional and so every solution must be a linear combination of $b'$ and $\psi$. The elements of $K_b$ are precisely the $2\pi$-periodic ones.
                
				Differentiating the identity $B(\theta+T(E),E)=B(\theta,E)$ with respect to $E$, and using that $\partial_\theta B(\cdot,E_0)=b'$, we obtain that for $\theta \in [0,2\pi]$
				$$\psi(\theta+T_0)+T'(E_0)\,b'(\theta+T_0)=\psi(\theta),\quad\text{where }T_0:=T(E_0).$$
				Since $b'$ is $T_0$-periodic, $b'\not\equiv 0$, and $T'(E_0)\not=0$, we obtain that $\psi$ is not $T_0$-periodic. Since $b$ is $2\pi$-periodic and $T_0$ is its minimal period, then $T_0$ divides $2\pi$, namely there exists $k\in\N$ such that $2\pi=kT_0$. Iterating the previous identity, we get $$\psi(\theta+2\pi)-\psi(\theta)=-kT'(E_0)\,b'(\theta)\not=0,$$ and thus $\psi$ is not $2\pi$-periodic. This means that $K_b=\text{span}\{b'\}$, concluding the proof.
			\end{proof}
			The next corollary shows that any cylindrical extension of a non-radial two-dimensional positive cone is integrable through rotations.
            \begin{corollary}\label{corollary:integrability1}
				Let $\gamma\in (1,2)$ and $d\geq 3$. Let $b\in\mathcal{B}_{d-2}$ and $b\not=b_{d-2}$, then it is integrable through rotations.
			\end{corollary}
			\begin{proof}
				If $b\in\mathcal{B}_{d-2}$, then $b$ is the cylindrical extension of a two-dimensional positive cone. Then the result follows by combining \cref{corollary:corollary3} and \cref{prop:dim2-integrability}.
			\end{proof}
            In the next corollary we summarize the results for cones in dimension $d=2$ and $d=3$.
			\begin{corollary}\label{corollary:integrability2}
				Let $\gamma\in (1,2)$, then:
                \begin{itemize} 
                \item if $d=2$, every cone $b$ is either integrable through rotations, or is a cylindrical cone $b=b_{\ell}$ for $\ell=0,1$;
                \item if $d=3$, every cone $b$ is either integrable through rotations, or is a cylindrical cone $b=b_{\ell}$ for $\ell=0,1,2$ or $b\in\mathcal{B}_0$.
                \end{itemize}
			\end{corollary}
			\begin{proof}
				The case $d=2$ follows by \cref{prop:dim2-integrability}, while $d=3$ follows by \cref{corollary:integrability1}.
			\end{proof}

        We also have the following rate characterization for minimizing cones in dimension $d=2$.
            \begin{remark}\label{rem:caso2special}
		Let $d=2$ and $\gamma\in (1,2)$, for a blow-up $b\in\mathcal{B}$ we have the following dichotomy.
        \begin{itemize}
            \item If $b=b_{\text{rad}}$ and $\gamma=\gamma_{k,2}=2-\frac{4}{k^2}$, for some $k\in \N_{\ge3}$, then the logarithmic convergence in \eqref{eq:log-conv} to $b_{\text{rad}}$ is sharp;
            \item otherwise, the logarithmic rate in the convergence \eqref{eq:log-conv} can be improved to $r^\alpha$. 
        \end{itemize}
		
        Indeed, we can combine the two epiperimetric inequalities \cref{t:epi-rad} and \cref{prop:epi-one-dim-cone}, together with the integrability through rotations for cones $b\not\in \{b_{\text{rad}}, b_{\text{one}}\}$ in \cref{prop:dim2-integrability} and the results on the radial cone \cref{prop:optimality-log}. Then the conclusion follows from the result in \cref{section:final}.
	\end{remark}			
			\section{Epiperimetric inequality for translational cones}\label{section:intermediate-epiperimetric}
			
			In this section we prove a logarithmic epiperimetric inequality for cylindrical extensions of cones, when $\gamma\in (1,2)$, under the additional assumption that $b$ is translational. Let $b\in\mathcal{B}_\ell$, for some $\ell=1,\ldots,d-2$, then we write \be\label{eq:bB2}b(y,z)=B(y),\quad (y,z)\in\R^{d-\ell}\times \R^\ell,\ee where $B>0$ in $\R^{d-\ell}\setminus\{0\}$.             
            Recalling the decomposition of the kernel           $K_b=K_B\oplus K_b^z$ given in \eqref{eq:Kbz-decomposition}, by \cref{prop:kernel-cylindric}, the assumption that $b$ is a translational cone is equivalent to $K_b^z$ being generated by rotations. We emphasize that $K_{b}^z$ represents only the part of the rotational kernel which is not of the form $|y|^\beta f(y/|y|)$, and not the full rotational kernel. Indeed, the rotational directions depending only on the $y$-variables are contained in $K_B$, and hence are excluded from $K_{b}^z$.

			The following is the main result of this section.
			\begin{proposition}[Logarithmic epiperimetric inequality]\label{t:epi-cylindrical}
			    Let $\gamma\in (1,2)$, then
            the epiperimetric inequality in \cref{thm:epiperimetric-combined} holds for cones $b\in\mathcal{B}_\ell$, $\ell=1,\ldots,d-2$ that are translational, under the closeness assumptions
            \be\label{eq:closness-assumption-cylindrical}\|z-b\|_{H^1(B_1)}\le \delta,\quad \|c-b\|_{L^\infty(\partial B_1)}\le \delta\quad\text{and}\quad |W(z)-W(b)|\le \delta.\ee 
				Moreover, if $b$ is sub-integrable (see \cref{def:sub-integrability}), then we can take $\sigma=0$.
			\end{proposition}
			\begin{remark}\label{rem:cyl1}
				By \cref{corollary:kernel-cylindrical}, if $b_\ell$ is a cylindrical cone as in \eqref{eq:parabolas}, for some $\ell=1,\ldots,d-2$, then $b_\ell$ is translational. Therefore, the epiperimetric inequality in \cref{t:epi-cylindrical} applies.
			\end{remark}
            \begin{remark}\label{rem:cyl2}
            By \cref{corollary:corollary1}, if $b$ is integrable through rotations, then $b$ is translational. Therefore, under the assumption of integrability through rotations, the epiperimetric inequality of \cref{t:epi-cylindrical} applies with $\sigma=0$.
			\end{remark}
			Let $b\in\mathcal{B}_\ell$ be as in \eqref{eq:bB2}, throughout the section, we will use the following coordinates. Given $\theta \in \mathbb{S}^{d-1}$, we write \be\label{eq:coordinates-theta}\theta=(\rho\,\omega,\sqrt{1-\rho^2}\,\eta)\quad\text{where}\quad \rho\in(0,1),\ \omega\in \mathbb{S}^{d-\ell-1},\ \eta\in \mathbb{S}^{\ell-1}.\ee
		      We also set $$X_B:=\{\phi\in H^1(\partial B_1):\ \phi=|y|^\beta f(y/|y|)=\rho^\beta f(\omega),\ f\in H^1(\mathbb{S}^{d-\ell-1}) \}.$$
              We notice that $b\in X_B$, since $b=b(\rho,\omega)=\rho^\beta B(\omega)$.
			\subsection{Decomposition of $H^1(\partial B_1)$} 
            In the next proposition, we prove a key oblique decomposition of $H^1(\partial B_1)$.
            \begin{proposition}\label{prop:decompositionH1}
			    Let $\gamma\in (1,2)$ and $b\in\mathcal{B}_\ell$ be a translational cone as in \eqref{eq:bB2}. Then
                \bea H^1(\partial B_1)=K_b^z\oplus (K_b^z)^\perp\quad\text{and} \quad (K_b^z)^\perp=K_B\oplus \mathcal{M}_b=K_B\oplus N_B\oplus \mathcal{O}_b.\eea
                Here $\mathcal M_b:=N_B\oplus\mathcal O_b$, $N_B$ denotes the orthogonal complement of $K_B$ in $X_B$, so that $X_B=K_B\oplus N_B$, and $\mathcal{O}_b$ is defined in \eqref{def:Ob} below. 
                
                Moreover, $\mathcal{O}_b$ is tangent to the level sets of the functional $\mathcal G$ along $X_B$, namely, for every $\phi_0\in X_B$ such that $b+\phi_0\ge0$, we have
			\be\label{eq:propertyRb}\delta\mathcal{G}(\phi_0)[\psi]=0\quad\text{for every }\psi\in \mathcal{O}_b.
			\ee
            \end{proposition}
            
            The main idea of the decomposition in \cref{prop:decompositionH1} is to divide the complement $\mathcal{M}_b$ of the whole kernel $K_b$ into two sets $N_B$ and $\mathcal{O}_b$. On the first subspace, $N_B$, we can prove a partial Lyapunov-Schmidt reduction (see \cref{lemma:Lyapunov-Schmidt-partial}), which is the analogue of \cref{lemma:Lyapunov-Schmidt} in the $y$-variables, producing a map $Y$. In the second subspace, $\mathcal{O}_b$, the first variation of $\mathcal{G}$ computed at $\phi+Y(\phi)\in X_B$, with $\phi\in K_B$, vanishes by \eqref{eq:propertyRb}. In this way, the first variation of $\mathcal{G}$ computed at $\phi+Y(\phi)$ vanishes on $\mathcal{M}_b$, which is a complement of the kernel $K_b$.
            
            We point out that this decomposition has the following two key advantages. First, 
            the elements in $X_B$ have the same order of vanishing as $b$, and thus small perturbations of $b$ in $X_B$ are non-negative. Secondly, the finite-dimensional reduction in $y$ allows us to prove that the Lyapunov-Schmidt map $Y$ is analytic, and thus we can apply the Łojasiewicz inequality.

            In order to prove \cref{prop:decompositionH1}, we introduce a suitable projection $\Pi$, which will be used to define $\mathcal{O}_b$.
			\begin{definition}\label{def:map-Pi}
            We define the linear map $\Pi:H^1(\partial B_1)\to H^1(\mathbb{S}^{d-\ell-1})$ as follows.
            Using the coordinates in \eqref{eq:coordinates-theta}, for every $\psi=\psi(\rho,\omega,\eta)\in H^1(\partial B_1)$, we define $$\Pi\psi(\omega):=\int_{0}^1\rho^{d-\ell+\beta-3}(1-\rho^2)^{\frac{\ell-2}2}\int_{\mathbb{S}^{\ell-1}} \psi(\rho,\omega,\eta)\,d\HH^{\ell-1}(\eta)\,d\rho.$$ 
            \end{definition}
            The key point of this definition is that, when the first variation of $\mathcal{G}$ is computed at an element of $X_B$ in the direction $\psi$, it depends only on the projected component $\Pi\psi$, as shown in the following lemma.
			\begin{lemma}\label{lemma:variation-Pi}
				Let $\phi_0\in X_B$ be such that $b+\phi_0\ge0$. Then, if we set $q:=b+\phi_0\in X_B$ and write $q(\rho,\omega)=\rho^\beta Q(\omega)$, we get \bea\frac12\delta\mathcal{G}(\phi_0)[\psi]=\int_{\mathbb{S}^{d-\ell-1}} \Big(\nabla_\omega Q\cdot\nabla_\omega (\Pi\psi)-\lambda_{d-\ell}(\beta)Q\Pi\psi+\frac\gamma2 Q^{\gamma-1}\Pi\psi\Big)\,d\HH^{d-\ell-1}(\omega).\eea 
			\end{lemma}
			\begin{proof}
				We start by computing the three terms in the variation \be\label{eq:inlemmapi}\frac12\delta\mathcal{G}(\phi_0)[\psi]=\frac12\delta\mathcal{F}(q)[\psi]=\int_{\partial B_1}\Big(\nabla_\theta q\cdot\nabla_\theta \psi-\lambda_d(\beta)q\psi+\frac\gamma2 q^{\gamma-1}\psi\Big)\,d\HH^{d-1}.\ee         
                Regarding the first term in \eqref{eq:inlemmapi}, we notice that, using the coordinates $\theta=(\rho\,\omega,\sqrt{1-\rho^2}\,\eta)$ in \eqref{eq:coordinates-theta}, the metric $g_{\mathbb{S}^{d-1}}$ is given by $$g_{\mathbb{S}^{d-1}}=\frac{1}{1-\rho^2}\,d\rho^2+\rho^2g_{\mathbb{S}^{d-\ell-1}}+(1-\rho^2)g_{\mathbb{S}^{\ell-1}}.$$ Then, setting $\alpha(\rho):=\rho^{d-\ell-1}(1-\rho^2)^{\frac{\ell-2}{2}}$, we have $$\nabla_\theta q\cdot\nabla_\theta\psi=(1-\rho^2)\partial_\rho q\partial_\rho \psi+\frac{1}{\rho^2}\nabla_\omega q\cdot\nabla_\omega \psi,\quad d\mathcal{H}^{d-1}=\alpha(\rho)\,d\rho \,d\mathcal{H}^{d-\ell-1}(\omega)\,d\mathcal{H}^{\ell-1}(\eta),$$ 
                where we used that $q$ is $\eta$-independent.
                Therefore, denoting by $$\widetilde \psi(\rho,\omega):=\int_{\mathbb{S}^{\ell-1}} \psi(\rho,\omega,\eta)\,d\HH^{\ell-1}(\eta),$$ we have
				$$\int_{\partial B_1}\nabla_\theta q\cdot\nabla_\theta \psi\,d\HH^{d-1}=\int_{\mathbb{S}^{d-\ell-1}}\int_0^1\alpha(\rho)\Big((1-\rho^2)\partial_\rho q \partial_\rho\widetilde\psi+\frac{1}{\rho^2}\nabla_\omega q\cdot\nabla_\omega \widetilde\psi\Big)\,d\rho\,d\HH^{d-\ell-1}(\omega).
				$$
				By definition of $\alpha(\rho)$, we also have
                $$ \partial_\rho(\alpha(\rho)(1-\rho^2)\beta\rho^{\beta-1})=\alpha(\rho)\rho^{\beta-2}(\lambda_{d-\ell}(\beta)-\lambda_{d}(\beta)\rho^2).$$
                Then, using that $\partial_\rho q=\beta \rho^{\beta-1}Q(\omega)$ and $\nabla_\omega q=\rho^\beta \nabla_\omega Q$,
                we can integrate by parts to get 
                \bea \int_{\partial B_1}\nabla_\theta q\cdot\nabla_\theta \psi\,d\HH^{d-1}=\int_{\mathbb{S}^{d-\ell-1}}\int_0^1\Big(&\alpha(\rho)\lambda_d(\beta)\rho^{\beta}Q\widetilde \psi-\alpha(\rho)\rho^{\beta-2}\lambda_{d-\ell}(\beta)Q\widetilde\psi
                \\&\qquad+\alpha(\rho)\rho^{\beta-2}\nabla_\omega Q\cdot\nabla_\omega\widetilde \psi\Big)\,d\rho\,d\HH^{d-\ell-1}(\omega).\eea
                We also observe that the second and the third term in \eqref{eq:inlemmapi} are given by $$\int_{\partial B_1} \lambda_d(\beta) q\psi\,d\HH^{d-1}=\int_{\mathbb{S}^{d-\ell-1}}\int_0^1\alpha(\rho)\lambda_d(\beta)\rho^\beta Q\widetilde \psi\,d\rho\,d\HH^{d-\ell-1}(\omega)$$ and $$\int_{\partial B_1}\frac\gamma2q^{\gamma-1}\psi\,d\HH^{d-1}=\int_{\mathbb{S}^{d-\ell-1}}\int_0^1\alpha(\rho)\rho^{\beta-2} \frac\gamma2 Q^{\gamma-1}\widetilde \psi\,d\rho\,d\HH^{d-\ell-1}(\omega).$$
                Using \eqref{eq:inlemmapi} with the last three identities, we conclude the proof, by definition of $\Pi$.
			\end{proof}
			\begin{proof}[Proof of \cref{prop:decompositionH1}]
			    We first decompose $H^1(\partial B_1)=K_b^z\oplus (K_b^z)^\perp$. By \eqref{eq:Krot}, we observe that the elements of $K_b^z$ are odd in $z$ and so $$X_B\subset (K_b^z)^\perp.$$
            We also set \be\label{def:Ob}\mathcal{O}_b:=\ker(\Pi)\cap (K_b^z)^\perp\subset (K_b^z)^\perp,\ee where $\Pi$ is the projection of \cref{def:map-Pi}. In particular, for every $\psi\in \mathcal{O}_b$, we have $\Pi\psi=0$. Therefore, by \cref{lemma:variation-Pi}, for every $\phi_0\in X_B$ such that $b+\phi_0\ge0$, we have $\delta\mathcal{G}(\phi_0)[\psi]=0$, which is exactly \eqref{eq:propertyRb}. 

           Since $X_B=K_B\oplus N_B$, in order to conclude the proof, we only need to show the validity of the decomposition $(K_b^z)^\perp=X_B\oplus \mathcal{O}_b$. We observe that if $\phi_0\in X_B$, then we can write $\phi_0(\rho,\omega)=\rho^\beta f(\omega)$, and thus $\Pi \phi_0(\omega)=c_\ast f(\omega)$, for some $c_\ast=c_\ast(d,\ell,\beta)>0$.
			Then, if $\phi\in (K_b^z)^\perp$, we set $$\phi_0(\rho,\omega):=\frac{1}{c_\ast}\rho^\beta \Pi \phi(\omega)\in X_B\subset (K_b^z)^\perp,$$ so that $\Pi(\phi-\phi_0)=0$, and thus $\phi-\phi_0\in \mathcal{O}_b$. Thus we proved $(K_b^z)^\perp=X_B + \mathcal{O}_b$. Finally, if $\phi=\rho^\beta f(\omega)\in X_B\cap \mathcal{O}_b$, then $0=\Pi\phi=c_\ast f(\omega)$, and so $f=0$. Therefore $X_B\cap \mathcal{O}_b=\{ 0\}$, and thus we have the desired direct sum $(K_b^z)^\perp=X_B\oplus \mathcal{O}_b$.
			\end{proof}

			\subsection{A partial Lyapunov-Schmidt reduction} 
			We set $N:={\rm dim}\,K_B$ and we recall the subspace $N_B$, which is the orthogonal complement of $K_B$ in $X_B$, i.e., $X_B=K_B\oplus N_B$.
			
			The following proposition gives a Lyapunov-Schmidt reduction only in the $y$-variables (see e.g.~\cite[Section 2]{Simon} or \cite[Lemma B.1]{esv}). 

            \begin{proposition}\label{lemma:Lyapunov-Schmidt-partial}
            There is a neighborhood $U\subset K_B$ of $0$ in $C^{1,\alpha}(\partial B_1)$ and an analytic map $$Y:K_B\cap U\to N_B\subset X_B$$ such that the following holds.
            \begin{itemize}
            \item $Y(0)=0$, $\delta Y(0)=0$. Moreover,
            \bea\label{eq:cond-on-perp-partial}
			P_{N_B}\big(\delta\mathcal{G}(\phi+Y(\phi))\big)=0,\quad \mbox{for every }\phi \in K_B\cap U.
            \eea
    \item Let $\Phi_1,\dots,\Phi_N$ be an orthonormal basis of $K_B$. Then, there exists $\rho>0$ such that, for every $\mu\in B_\rho\subset\R^N$, the reduced functional $G:B_\rho\to\R$ defined by
\be\label{def:G-finito-dim-partial}
G(\mu):=\mathcal{G}(\Phi_\mu +Y\left(\Phi_\mu\right)),
\quad \mbox{with}\quad
\Phi_\mu:=\sum_{j=1}^N\mu_j\Phi_{j},
\ee
satisfies $P_{K_B}\big(\delta\mathcal{G}(\Phi_\mu+Y(\Phi_\mu))\big)=\nabla_\mu G(\mu)$, for every $\mu\in B_{\rho}$.
			\end{itemize}
            \end{proposition}
			We point out that the functional $\mathcal{G}$ in \cref{lemma:Lyapunov-Schmidt-partial} must be understood as its restriction to $X_B$. Then, the analyticity of the Lyapunov-Schmidt map $Y$ in \cref{lemma:Lyapunov-Schmidt-partial} follows by the fact that the restriction of $\mathcal{G}$ to ${X_B}$ is analytic in a neighborhood of $0$.
            Notice that the analyticity of $Y$ is the key ingredient to apply the Łojasiewicz inequality in \cref{lemma:h0} below.

			\begin{remark}\label{remark:sub-int-ifandonlyif-sub-int} We observe that the partial Lyapunov-Schmidt reduction in \cref{lemma:Lyapunov-Schmidt-partial} is, in essence, the Lyapunov-Schmidt reduction of \cref{lemma:Lyapunov-Schmidt} applied to the cone $B$ on $\mathbb{S}^{d-\ell-1}$. In particular, the integrability and sub-integrability of $b$ are equivalent to the corresponding properties for $B$.

            More precisely, let $\mathcal{G}:H^1(\mathbb{S}^{d-1})\to \R$ and $\widetilde{\mathcal{G}}:H^1(\mathbb{S}^{d-\ell-1})\to \R$ be the functionals in \eqref{def:G} associated respectively to $b$ and $B$. Arguing as in \cref{lemma:variation-Pi}, for every $\phi_0=\rho^\beta f(\omega)\in X_B$ and $\psi=\rho^\beta \Psi(\omega)\in X_B$, we have
                     $$\mathcal{G}(\phi_0)=c_\ast\widetilde{\mathcal{G}}(f)\quad\text{and}\quad\delta\mathcal{G}(\phi_0)[\psi]=c_\ast\delta \widetilde{\mathcal{G}}(f)[\Psi],$$ for some $c_\ast=c_\ast(d,\ell,\beta)>0$. Let $\widetilde Y$ be the Lyapunov-Schmidt map associated with $B$ in \cref{lemma:Lyapunov-Schmidt}. Then, for $\phi_0=\rho^\beta f(\omega)\in K_B$ the correction map $Y$ associated with $b$ in \cref{lemma:Lyapunov-Schmidt-partial} satisfies
                     $$Y(\phi_0)(\rho,\omega)=\rho^\beta \widetilde Y(f)(\omega).$$
                     Therefore, if $\Phi_\mu=\rho^\beta \widetilde\Phi_\mu(\omega)\in K_B$, with $\widetilde \Phi_\mu$ as in \eqref{def:G-finito-dim}, we obtain $$G(\mu)=\mathcal{G}(\Phi_\mu+Y(\Phi_\mu))=\mathcal{G}(\rho^\beta \widetilde\Phi_\mu+\rho^\beta\widetilde Y(\widetilde\Phi_\mu))=c_{\ast}\widetilde{\mathcal{G}}(\widetilde\Phi_\mu+\widetilde Y(\widetilde\Phi_\mu))=c_{\ast} \widetilde{G}(\mu),$$ 
                        where $G$ and $\widetilde{G}$ are the reduced functionals associated with $\mathcal{G}$ and $\widetilde{\mathcal{G}}$, respectively. Since $c_\ast>0$, the integrability, respectively sub-integrability, of $b$ is equivalent to the integrability, respectively sub-integrability, of $B$.
			\end{remark}
			\subsection{Decomposition of the trace} 
            In order to decompose the trace $c\in H^1(\partial B_1)$, we use the oblique decomposition of \cref{prop:decompositionH1}. We first kill the rotational modes in $K_b^z$ using the following lemma.
            \begin{lemma}\label{lemma:kill-rotation}
                Let $c\in H^1(\partial B_1)$ be a non-negative trace satisfying the closeness assumptions \eqref{eq:closness-assumption-cylindrical}. Then, there exists a rotation $R\in SO(d)$ such that $$c-Rb \in (K_b^z)^\perp,$$ where $Rb(\theta):=b(R^{-1}\theta)$.
            \end{lemma}
            \begin{proof}
    Given an orthonormal basis $\psi_1,\ldots,\psi_m$ of $K_b^z$, we take $\nu=(\nu_{1},\ldots,\nu_{m})\in \R^m$ and we consider the rotation 
    $R_\nu:=R^{1}_{\nu_{1}}\circ\ldots\circ R^{m}_{\nu_m},$
    where $R_{t}^{j}$ is a one-parameter family of rotations whose infinitesimal action on $b$ generates $\psi_j$, i.e., $\partial_{t}(R^j_{t} b) \big|_{t=0}=\psi_j$.
    Now, consider the function $F:\R^{m}\to \R^{m}$ given by $$F(\nu):=\left(\int_{\partial B_1}(R_{\nu} b)\psi_{1}\,d\HH^{d-1},\ldots,\int_{\partial B_1}(R_{\nu} b)\psi_{m}\,d\HH^{d-1}\right).$$ Since $\partial_{\nu_j}(R_{\nu} b) \big|_{\nu=0}=\psi_j,$ we have that $ DF(0)=\mathrm{Id}$. Then, we
can apply the implicit function theorem to $F$ in a neighborhood of $0$. 
Therefore, if $c$ is sufficiently close to $b$, we can define $$\nu:=F^{-1}\left(\int_{\partial B_1}c\psi_{1}\,d\HH^{d-1},\ldots,\int_{\partial B_1}c\psi_{m}\,d\HH^{d-1}\right),$$ and $R:=R_{\nu}$ is the desired rotation.
\end{proof}

            Let $c\in H^1(\partial B_1)$ be a non-negative trace satisfying the closeness assumptions \eqref{eq:closness-assumption-cylindrical}. By \cref{lemma:kill-rotation}, up to a rotation, we can suppose that $R=\mathrm{Id}$, so that $c-b\in(K_b^z)^\perp$. 
            
            By the decomposition of $H^1(\partial B_1)$ in \cref{prop:decompositionH1}, we can write 
            $$
            c=b+\phi+\widetilde\varphi,\qquad \text{where}\quad \phi \in K_B,\, \widetilde\varphi \in \mathcal{M}_b.
            $$
            Now, set $\phi_0:=\phi+Y(\phi)$, where $Y$ is the Lyapunov-Schmidt map in \cref{lemma:Lyapunov-Schmidt-partial}. Since $Y(\phi)\in N_B\subset \mathcal{M}_b$, we can also set $\varphi:=\widetilde \varphi-Y(\phi)\in \mathcal{M}_b$. Then we have the decomposition of the trace \be\label{eq:dec-trace-cylindrical}c=b+\left(\phi+Y(\phi)\right)+\left(\widetilde \varphi -Y(\phi)\right)=b+\phi_0+\varphi.\ee
			Recalling the orthonormal basis $\Phi_1,\ldots,\Phi_{N}$ of $K_B$ and the definition of $\Phi_\mu$ in \eqref{def:G-finito-dim-partial}, we can write $$\phi_0=\Phi_{\mu^0}+Y(\Phi_{\mu^0})=\sum_{j=1}^N \mu_j^0\Phi_{j}+Y\left(\sum_{j=1}^N \mu_j^0\Phi_{j}\right)$$ for some $\mu^0=(\mu_1^0,\ldots,\mu_N^0)\in\R^N$.
			\begin{remark}\label{remark:positivity-parabola}
			Arguing as in \cref{r:h-admissible}, we observe that $q:=b+\phi_0\ge0$ on $\mathbb{S}^{d-1}$ and, using the coordinates in \eqref{eq:coordinates-theta}, $q(\rho,\omega)=\rho^\beta Q(\omega)$, with $Q>0$ on $\mathbb{S}^{d-\ell-1}$. Indeed, we first observe that $B\ge \overline c>0$ on $\mathbb{S}^{d-\ell-1}$, for some constant $\overline c>0$. Moreover, since $\phi$ and $Y(\phi)$ are small for $\delta$ small, if we write $\phi_0(\rho,\omega)=\rho^\beta f(\omega)$, then $\|f\|_{L^\infty(\mathbb{S}^{d-\ell-1})}=\|\phi_0\|_{L^\infty(\partial B_1)}\le C\delta$. Then, recalling that $b(\rho,\omega)=\rho^\beta B(\omega)$, we have $$q(\rho,\omega)=b(\rho,\omega)+\phi_0(\rho,\omega)=\rho^{\beta}B(\omega)+\rho^{\beta}f(\omega)\ge \rho^{\beta}(\overline c-C\delta)\ge0,$$ by choosing $\delta$ small enough. 
            \end{remark}
			\subsection{The competitor by the gradient flow of $G$}
			We recall the function $G$ defined in \cref{lemma:Lyapunov-Schmidt-partial},
			and we consider the gradient flow $\mu(t)$ starting from $\mu^0$, which is defined as in \eqref{eq:flow} through the gradient $\nabla_\mu G$. We also consider the functions $\eta(r)$ and $\eta_1(r)$ as in \eqref{def:etaeta1}, and we set \be\label{eq:aseta}\overline \psi(r,\theta):=b(\theta)+\Phi_{\mu(\eta(r))}+Y(\Phi_{\mu(\eta(r))})=b(\theta)+\sum_{j=1}^N \mu_j(\eta(r))\Phi_j(\theta)+Y\left(\sum_{j=1}^N \mu_j(\eta(r))\Phi_j(\theta)\right).\ee 
			We also define $q(r,\theta):=r^\beta b(\theta)+r^\beta \phi_0(\theta)$ and $\overline q(r,\theta):=r^\beta \overline \psi(r,\theta)$, then we have the following lemma, which essentially follows by the Łojasiewicz inequality \cite{lojasiewicz}.
			\begin{lemma}\label{lemma:h0}
				Let $\mathcal{G}(\phi_0)>0$ and let $\rho\in(0,1)$. 
				Then, there is $\overline\eps=\overline\eps(d,\gamma,b,\rho)>0$, $C_0=C_0(d,\gamma,b,\rho)>0$ and $\sigma=\sigma(d,\gamma,b)\in(0,1/2]$, such that, setting $$\eta_1(r):=C_0\eps \mathcal{G}(\phi_0)^\sigma(\rho-r)\quad\text{and}\quad \eps:=\overline \eps\,\mathcal{G}(\phi_0)^{1-2\sigma},$$ and choosing $\delta=\delta(d,\gamma,b,\rho)$ small enough, we have $$W(\overline q)-W(b)\le(1-\eps) (W(q)-W(b)),$$ 
			\end{lemma}
            \begin{proof}
                The proof follows by the same argument of \cref{lemma:uguale}.
            \end{proof}
			\subsection{Decomposition of the energy}

			We first observe that, by combining the decomposition in \cref{prop:decompositionH1} together with the partial Lyapunov-Schmidt reduction in \cref{lemma:Lyapunov-Schmidt-partial}, we have the following result. Given $\phi_0=\phi+Y(\phi)\in X_B$ as in \eqref{eq:dec-trace-cylindrical}, we have \be\label{eq:kill-first-variation}\delta\mathcal{F}(b+\phi_0)[\varphi]=\delta\mathcal{G}(\phi_0)[\varphi]=0,\quad\text{for every }\varphi\in \mathcal{M}_b.\ee 
			Indeed, if we write $\varphi=\varphi_{N_B}+\varphi_{\mathcal{O}_b}$, where $\varphi_{N_B}\in N_B$, $\varphi_{\mathcal{O}_b}\in \mathcal{O}_b$, then, by \cref{lemma:Lyapunov-Schmidt-partial}, we get $\delta\mathcal{G}(\phi_0)[\varphi_{{N_B}}]=0,$ and, by \eqref{eq:propertyRb}, we infer that 
            $\delta\mathcal{G}(\phi_0)[\varphi_{\mathcal{O}_b}]=0$.
			
			We consider the first variation of $W$ by using the slicing, as in \cref{lemma:slicing}
			\be\label{eq:deltaW-slicing}\delta W(r^\beta \phi_r)[r^\beta \varphi_r]=\int_{0}^1r^{d+2\beta-3}\delta\mathcal{F}(\phi_r)[\varphi_r]\,dr+2\int_0^1r^{d+2\beta-1}\int_{\partial B_1}\partial_r\phi_r\partial_r\varphi_r\,d\HH^{d-1}\,dr.\ee 
			Given $\rho\in(0,1)$, we choose $\eta(r)$ as in \eqref{def:etaeta1}, so that $\eta(r)\equiv0$ in $(\rho,1)$. We take the corresponding $\overline q(r,\theta):=r^\beta \overline \psi(r,\theta)$, where $\overline \psi$ is defined in \eqref{eq:aseta}.           
            Arguing as in \eqref{eq:formula-in-epi}, for every non-negative function $h\in H^1(B_1)$, we have
			\be\label{eq:decomposition-competitor}W(h)-W(\overline q)=W_0(h-\overline q)+\delta W(\overline q)[h-\overline q]+\int_{B_1}Q_{\overline q}(h-\overline q)\,dx,\ee where
			\bea\label{def:R-cylindrical-cone} Q_{\overline q}(w):=(\overline q+w)^\gamma-\overline q^\gamma-\gamma \overline q^{\gamma-1}w=\int_0^1\int_0^t \gamma(\gamma-1)(\overline q+sw)^{\gamma-2}
			w^2\,ds\,dt\ge0\eea is defined for every $w$ such that $\overline q+w\ge0$.
			
			Given a non-negative function $h\in H^1(B_1)$, we write $h(r,\theta)=r^\beta \psi^r(\theta)$, for some $\psi^r(\theta)=\psi(r,\theta)$. Then, if we suppose that $\psi^r-\overline\psi^r\in \mathcal{M}_b$ for every $r\in(\rho,1)$ and $\psi^r-\overline\psi^r\equiv0$ for $r\in(0,\rho)$, then, by \eqref{eq:kill-first-variation} and \eqref{eq:deltaW-slicing}, we have \be\label{eq:utileparabola1}\delta W(\overline q)[h-\overline q]=0,\ee where we used that  $\overline q=q$ for $r\in(\rho,1)$ and consequently $\partial_r\overline \psi^r=0$ for $r\in(\rho,1)$. 
			
			Similarly, if $q=r^\beta b+r^\beta \phi_0$, by \eqref{eq:kill-first-variation} and \eqref{eq:deltaW-slicing} we have $\delta W(q)[z-q]=0$, and so  \be\label{eq:decomposition-non-competitor}W(z)-W(q)=W_0(z-q)+\int_{B_1}Q_{q}(z-q)\,dx=:\mathcal{Q}_q(z-q).\ee

			\subsection{Proof of the epiperimetric inequality}
			Now we have all the ingredients to prove the epiperimetric inequality in \cref{t:epi-cylindrical}.
			\begin{proof}[Proof of \cref{t:epi-cylindrical}]
				By contradiction, let us suppose that there is a sequence of non-negative traces $c_j\in H^1(\partial B_1)$, constants $\widetilde\eps_j\to0^+$, $\delta_j\to0^+$ and $\sigma\in(0,1/2]$ as in \cref{lemma:h0}, such that the following holds. Let $z_j$ be the $\beta$-homogeneous extension of $c_j$, then
				$$\|z_j-b\|_{H^1(B_1)}\le \delta_j,\quad \|c_j-b\|_{L^\infty(\partial B_1)}\le \delta_j\quad\text{and}\quad |W(z_j)-W(b)|\le \delta_j$$
				and \be\label{eq:contr-hyp1-cylindrical} (1-\eps_j)(W(z_j)-W(b))<W(h)-W(b),\quad \eps_j=\widetilde \eps_j|W(z_j)-W(b)|^{1-2\sigma}\ee for every non-negative function $h\in H^1(B_1)$ such that $h=c_j$ on $\partial B_1$.
				
				For every $j$, we decompose $c_j$ as in \eqref{eq:dec-trace-cylindrical}, and we set $$q_j(r,\theta):=r^\beta b(\theta)+r^\beta(\phi_0)_j(\theta)\ge0,$$ where the non-negativity follows by \cref{remark:positivity-parabola}, and  $(\phi_0)_j=\phi_j+Y(\phi_j)$, with $\phi_j\in K_B$.
                
                We observe that $W(z_j)-W(b)>0$, since we can choose $h=z_j$ in \eqref{eq:contr-hyp1-cylindrical}. 
                We also set $\eta_j:=\|z_j-q_{j}\|_{H^1(B_1)}$ and we start by proving that $\eta_j>0$ and $\eta_j\to0^+$ as $j\to+\infty$. First, since the projection onto $K_B$ is bounded, we have $\|r^\beta\phi_j\|_{H^1(B_1)}\le C\delta_j$, then, by the analyticity of $Y$, we also have $\|r^\beta(\phi_0)_j\|_{H^1(B_1)}\le C\delta_j$. Thus $\eta_j\le C\delta_j$, so that $\eta_j\to0^+$ as $j\to+\infty$. Then, suppose by contradiction that $\eta_j=0$ for some $j>0$. Thus, we take $\overline q_j$ and $\overline \eps$ as the constants in \cref{lemma:h0} corresponding to the choice $\rho=1/2$. Since $z_j=q_j$, we have that $$\eps_j=\widetilde \eps_j|W(z_j)-W(b)|^{1-2\sigma}=\widetilde \eps_j|W(q_j)-W(b)|^{1-2\sigma}= C\widetilde\eps_j\mathcal{G}((\phi_0)_j)^{1-2\sigma}\le\overline\eps\mathcal{G}((\phi_0)_j)^{1-2\sigma},$$
                 for $j$ large enough. This is a contradiction, by choosing $h=\overline q_j$ in \eqref{eq:contr-hyp1-cylindrical}.
							
				By using the decompositions \eqref{eq:decomposition-competitor} and \eqref{eq:decomposition-non-competitor}, the estimate \eqref{eq:contr-hyp1-cylindrical} becomes
				\be\label{eq:contrad-hyp3-cylindrical-cones}
				\begin{aligned}
					(1-\eps_j)&\bigg( W_0(z_j-q_j)+\int_{B_1}Q_{q_j}(z_j-q_j)\,dx
					\bigg)\\&\qquad\le W_0(h-\overline q_j)+\delta W(\overline q_j)[h-\overline q_j]+\int_{B_1}Q_{\overline q_j}(h-\overline q_j)\,dx+E_j.
				\end{aligned} 
				\ee where $$E_j:=W(\overline q_j)-W(b)-(1-\eps_j)(W(q_j)-W(b)).$$
				
                Let $\rho\in(0,1)$, we recall the function $\eta_1$ used to build the competitor in \eqref{eq:aseta} corresponding to $\rho$, and the functional $\mathcal{Q}_q$ in \eqref{eq:decomposition-non-competitor}. After passing to a subsequence, at least one of the following alternatives (corresponding to the alternatives in \cref{subsec:ultima}) holds for every $j$:
				\begin{enumerate}[label=(\roman*)]
					\item[(i)] if $\mathcal{G}((\phi_0)_j)\le0$, then we choose $\eta_1\equiv0$;
					\item[(ii)] if $\mathcal{G}((\phi_0)_j)>0$ and $\mathcal{Q}_{q_j}(z_j-q_j)>\mathcal{G} ((\phi_0)_j)$, then we choose $\eta_1\equiv0$;
					\item[(iii)] if $\mathcal{G}((\phi_0)_j)>0$ and $\mathcal{Q}_{q_j}(z_j-q_j)\le \mathcal{G} ((\phi_0)_j)$, then we choose $\eta_1$ corresponding to $\rho$ as in \cref{lemma:h0}.
				\end{enumerate}
				Then, we choose \be\label{eq:itemize}\overline q_j(r,\theta):=r^\beta \overline \psi_j(r,\theta),
                \ee where $\overline \psi_j$ is constructed as in \eqref{eq:aseta} using the choice of $\eta_1$, corresponding to $\rho$, as in the previous cases.
                Then, in the case (i) we have $W(q_j)-W(b)\le0$ and thus $$E_j=\widetilde\eps_j |W(z_j)-W(b)|^{1-2\sigma}(W(q_j)-W(b))\le0.$$ Similarly, in the case (ii), since $|W(z_j)-W(b)|\le \delta_j$, we have by \eqref{eq:decomposition-non-competitor} that $$E_j=\widetilde\eps_j |W(z_j)-W(b)|^{1-2\sigma}(W(q_j)-W(b))\le  \widetilde \eps_j\mathcal{G}((\phi_0)_j)\le \widetilde\eps_j\mathcal{Q}_{q_j}(z_j-q_j).$$
				In the case (iii), we have again by \eqref{eq:decomposition-non-competitor}, $W(z_j)-W(b)\le C\mathcal{G}((\phi_0)_j)$, and thus, by \cref{lemma:h0} $$E_j\le W(\overline q_j)-W(b)-(1-\widetilde\eps_jC^{1-2\sigma}\mathcal{G}((\phi_0)_j)^{1-2\sigma})(W(q_j)-W(b))\le0.$$  
				Putting all together and using that $\mathcal{Q}_{q_j}(z_j-q_j)\ge0$, we have that, given $\rho\in(0,1)$, \be\label{eq:est-Ej} E_j\le \widetilde \eps_j \mathcal{Q}_{q_j}(z_j-q_j),\ee for $j$ large enough depending on $\rho$.
				
				Finally, we introduce the normalized functions $w_j:=(z_j-q_j)/\eta_j$, and we choose a subsequence $j\to+\infty$ such that $ w_j\to w$ weakly in $H^1(B_1)$ and strongly in $L^2(B_1)$.
				By the decomposition of the trace \eqref{eq:dec-trace-cylindrical}, we observe that $w_j(r,\theta)=r^\beta w_j(\theta)$, with $w_j(\theta)\in \mathcal{M}_b$.
                
                We divide the rest of the proof in several steps.
				\\ \vspace{-0.3cm}
				
				\noindent\textit{Step 1.} In the first step, we prove that for every $j$ large enough \be\label{eq:proof-step-1-cylindrical-cones}\int_{B_1}\frac{Q_{q_j}(z_j-q_j)}{\eta_j^2}\,dx\le C,\quad \int_{B_1}b^{\gamma-2}w^2\,dx\le C \quad\text{and}\quad \frac{E_j}{\eta_j^2}\le C\widetilde \eps_j,\ee where $C$ does not depend on $j$.
				To prove these inequalities, let $\zeta$ be a radial function such that $0\leq \zeta\leq 1$ and $\zeta \equiv 1$ in $B_{1/2}$, and  
                $\overline q_j$ be the competitor in \eqref{eq:itemize} corresponding to $\rho=1/2$. Then, we consider the competitor  $h:=(1-\zeta)z_j+\zeta \overline q_j$ in \eqref{eq:contrad-hyp3-cylindrical-cones}. 
                
                Since $h-\overline q_j=(1-\zeta)(z_j-\overline q_j)$, by \eqref{eq:utileparabola1} we have that $\delta W(\overline q_j)[h-\overline q_j]=0$. Moreover, by the convexity of $w\mapsto Q_{\overline q_j}(w)$ and $Q_{\overline q_j}(0)=0$, we have \bea \int_{B_1}{Q_{\overline q_j}(h-\overline q_j)}\,dx&\le \int_{B_1}(1-\zeta){Q_{\overline q_j}(z_j-\overline q_j)}\,dx=\int_{B_1}(1-\zeta){Q_{q_j}(z_j-q_j)}\,dx
				,\eea where we use the fact that $\overline q_j=q_j$ in $B_1\setminus B_{1/2}$ and $\zeta\equiv1$ in $B_{1/2}$.
				Then, since $\zeta\equiv 1$ in $B_{1/2}$, \eqref{eq:contrad-hyp3-cylindrical-cones} becomes
				\be\label{eq:eq:ultimastep4}\begin{aligned}
					(1-\eps_j)&\bigg( W_0(z_j-q_j)+\int_{B_1}Q_{q_j}(z_j-q_j)\,dx
					\bigg)\\&\qquad\le W_0((1-\zeta)(z_j-q_j))+\int_{B_1}(1-\zeta)Q_{q_j}(z_j-q_j)\,dx+E_j.
				\end{aligned}\ee where we used $W_0((1-\zeta)(z_j-\overline q_j))=W_0((1-\zeta)(z_j-q_j))$. Now we can proceed as in Step 1 of \cref{prop:epi-one-dim-cone}, and we get
				\bea \int_{B_1}\frac{Q_{q_j}(z_j-q_j)}{\eta_j^2}\,dx\le C+\frac{CE_j}{\eta_j^2}.
				\eea The first estimate of \eqref{eq:proof-step-1-cylindrical-cones} follows by the fact that, by \eqref{eq:est-Ej}, \be\label{e:manca}\frac{E_j}
				{\eta_j^2}\le C\widetilde \eps_j+\widetilde\eps_j\int_{B_1}\frac{Q_{q_j}(z_j-q_j)}{\eta_j^2}\,dx.\ee The third estimate of \eqref{eq:proof-step-1-cylindrical-cones} follows by substituting the first one in \eqref{e:manca}.
                The second estimate of \eqref{eq:proof-step-1-cylindrical-cones} follows by the first one, together with the bound $\eta_j^2w_j^2(z_j+q_j)^{\gamma-2}\le CQ_{q_j}(z_j-q_j)$. Indeed, since $z_j\to b, q_j\to b$, and $w_j\to w$ a.e. in $B_1$, up to passing to a subsequence, the conclusion follows from Fatou's lemma.
				\\ \vspace{-0.3cm}
				
				\noindent\textit{Step 2.}
				In the second step, we prove that if $w(r,\theta)=r^\beta w(\theta)$, then its trace solves the linearized equation \be\label{eq:proof-step-2-cylindrical-cones}
				L_bw=0\quad\text{in }\partial B_1\cap\{b>0\}.
				\ee
				To this end, we first take an open domain $D\Subset B_1\cap \{b>0\}$ and $\rho>0$ such that $D\cap \overline B_{2\rho}=\emptyset$. With this choice of $\rho$, let $\overline q_j$ be the corresponding function as in \eqref{eq:itemize} and consider the competitor $$h:=\zeta (\overline q_j+\eta_j\varphi_j)+(1-\zeta)((1-\chi)z_j+\chi \overline q_j),$$ where: 
                \begin{itemize}
                    \item $\zeta\in C^\infty_c(B_1\cap\{b>0\})$, with $0\le\zeta\le1$, $\zeta\equiv0$ in $B_{2\rho}$ and $\zeta\equiv1$ in $D$;
                    \item $\chi\in C^\infty_c(B_{2\rho})$ is radial and satisfies $0\le \chi\le 1$, $\chi\equiv1$ in $B_\rho$, $|\nabla \chi|\le C\rho^{-1}$;
                    \item $\varphi_j:=w_j+\xi$, where $\xi\in H^1(B_1)\cap L^\infty(B_1)$ does not depend on $j$; moreover, we assume that $\xi\equiv0$ in $B_1\setminus D$, and it takes the form $\xi(r,\theta)=r^\beta \xi_r(\theta)$, where $\xi_r\in \mathcal{M}_b$ for every $r\in(0,1)$.
                \end{itemize}
                Since the support of $\zeta$ is away from $\{b=0\}$ and $\overline q_j=q_j$ in $\text{supp}(\zeta)$, for every $j$ large enough, $\overline q_j+\eta_j\varphi_j=z_j+\eta_j\xi\ge0$ in $\text{supp}(\zeta)$. We stress that the latter lower bound follows from the fact that $\xi$ is bounded and $z_j$ is bounded from below away from zero in the support of $\zeta$, since $z_j\to b$.
				
				We observe that, since $\overline q_j=q_j$ in $B_1\setminus B_\rho$ and $\chi\equiv1$ in $B_\rho$, then $(1-\chi)(\overline q_j-q_j)=0$, and thus \bea h-\overline q_j&=\zeta\eta_j\varphi_j+(1-\zeta)(1-\chi)(z_j-\overline q_j)=\zeta\eta_j\varphi_j+(1-\zeta)(1-\chi)\eta_jw_j\\&=\zeta \eta_j(\varphi_j-w_j)+(1-\chi)\eta_jw_j,\eea 
                where in the last equality we used that $\chi\equiv0$ in $\text{supp}(\zeta)$.
                Moreover, by using that $\zeta\equiv1$ in $D$ and $\varphi_j=w_j$ outside $D$, we have $\zeta (\varphi_j-w_j)=\varphi_j-w_j=\xi$, and thus 
                \be\label{eq:diffh} h-\overline q_j= \eta_j\xi+(1-\chi)\eta_jw_j.\ee 
                Then, by combining \eqref{eq:kill-first-variation} and \eqref{eq:deltaW-slicing}, we have $$\delta W(\overline q_j)[h-\overline q_j]=\delta W(\overline q_j)[\eta_j\xi]+\delta W(\overline q_j)[(1-\chi)\eta_jw_j]=0.$$ Indeed, for the first term we use $\xi_r\in \mathcal{M}_b$ for every $r\in (0,1)$, $\overline q_j$ is $\beta$-homogeneous in $B_1\setminus B_\rho$, and $\xi\equiv0$ in $B_\rho$.
				For the second term, we use that $\chi\equiv1$ in $B_\rho$, $\overline q_j$ is $\beta$-homogeneous in $B_1\setminus B_\rho$, and, since $\chi$ is radial, then $(1-\chi)w_j(\theta)\in \mathcal{M}_b$.

				Therefore, by \eqref{eq:proof-step-1-cylindrical-cones}, the inequality \eqref{eq:contrad-hyp3-cylindrical-cones} implies that
				$$\begin{aligned}
					& \frac{W_0(z_j-q_j)}{\eta_j^2}+\int_{B_1}\frac{Q_{q_j}(z_j-q_j)}{\eta_j^2}\,dx
					\le C\widetilde \eps_j+\frac{W_0(h-\overline q_j)}{\eta_j^2}+\int_{B_1}\frac{Q_{\overline q_j}(h-\overline q_j)}{\eta_j^2}\,dx.
				\end{aligned}$$
				Using that $(h-\overline q_j)/\eta_j=\xi+(1-\chi)w_j$, see \eqref{eq:diffh}, we can proceed as in Step 2 of \cref{prop:epi-one-dim-cone}, and we get
				\be\label{eq:passj}\begin{aligned}
				\int_{B_1}(1-&(1-\chi)^2)|\nabla w_j|^2\,dx+\int_{D\cup B_{2\rho}}\frac{Q_{q_j}(z_j-q_j)}{\eta_j^2}\,dx\\&\le C\widetilde \eps_j+\int_{B_1}\Big(|\nabla\xi|^2+|\nabla\chi|^2 w_j^2-2w_j(1-\chi)\nabla \chi\cdot\nabla w_j+2\nabla\xi\cdot \nabla((1-\chi)w_j)\Big)\,dx\\&\qquad+\int_{D\cup B_{2\rho}}\frac{Q_{\overline q_j}(h-\overline q_j)}{\eta_j^2}\,dx,
				\end{aligned}\ee where we used that, in $B_1\setminus (D\cup B_{2\rho})$, we have $\overline q_j=q_j$, $\xi=0$, and $\chi=0$, and therefore \eqref{eq:diffh} gives $h-\overline q_j=z_j-q_j$ in $B_1\setminus (D\cup B_{2\rho})$.

                We also observe that, by using both \eqref{eq:diffh} and the convexity of  $w\mapsto Q_{q_j}(w)$, with $Q_{q_j}(0)=0$, then $$Q_{\overline q_j}(h-\overline q_j)=Q_{q_j}((1-\chi)(z_j-q_j))\le (1-\chi)Q_{q_j}(z_j-q_j)\quad\text{in }B_{2\rho}.$$
                More precisely, we exploit that  $\xi\equiv0$ in $B_{2\rho}$, $\chi\equiv1$ in $B_\rho$ and $\overline q_j=q_j$ in $B_1\setminus B_\rho$.
                Then, by \eqref{eq:proof-step-1-cylindrical-cones}
                \bea \int_{B_{2\rho}}\frac{Q_{q_j}(z_j-q_j)}{\eta_j^2}\,dx+\int_{B_{2\rho}}\frac{Q_{\overline q_j}(h-\overline q_j)}{\eta_j^2}\,dx&\le C\int_0^{2\rho}r^{d+2\beta-3}\int_{\partial B_1}\frac{Q_{q_j}(z_j-q_j)}{\eta_j^2}\,d\HH^{d-1}\\&\le C\rho^{d+2\beta-2}
                \eea
                for some $C$ which does not depend on $j$ and on $\rho$.
                
                We can pass to the limit as $j\to+\infty$ in \eqref{eq:passj}. By using that $\{b=0\}$ is away from $D$ and that
				$\overline q_j\to b$, $h-\overline q_j\to0$ in $D$, we get
				\bea\label{eq:daderivare}
				\begin{aligned}
					\int_{B_1}&|\nabla w|^2\,dx+\int_{D}\frac\gamma2(\gamma-1)b^{\gamma-2}w^2\,dx\\&\qquad\le \int_{B_1}|\nabla(\xi+(1-\chi)w)|^2\,dx+\int_{D}\frac\gamma2(\gamma-1)b^{\gamma-2}(\xi+(1-\chi)w)^2\,dx+C\rho^{d+2\beta-2}.
				\end{aligned}
				\eea
                Moreover, since $\xi+(1-\chi)w=w$ in $B_1\setminus (D\cup B_{2\rho})$, the contributions of the two integrals over $B_1\setminus (D\cup B_{2\rho})$ cancel out. Hence, the two integrals over $B_1$ above may be replaced by the corresponding integrals over $D\cup B_{2\rho}$. Now, by combining $$\int_{B_{2\rho}}|\nabla ((1-\chi)w)|^2\,dx+\int_{B_{2\rho}}|\nabla w|^2\,dx\le C\int_{B_{2\rho}}|\nabla w|^2\,dx+\frac{C}{\rho^2}\int_{B_{2\rho}}w^2\,dx\le C\rho^{d+2\beta-2},$$ with the facts that $\xi=0$ in $B_{2\rho}$ and $\chi=0$ in $D$, we can pass to the limit as $\rho\to0^+$ to get 
                \bea
					\int_{D}&|\nabla w|^2\,dx+\int_{D}\frac\gamma2(\gamma-1)b^{\gamma-2}w^2\,dx\le \int_{D}|\nabla(w+\xi)|^2\,dx+\int_{D}\frac\gamma2(\gamma-1)b^{\gamma-2}(w+\xi)^2\,dx.
				\eea
                Notice that this minimality condition holds only for test functions $\xi\in H^1(B_1)\cap L^\infty(B_1)$ with $\xi\equiv0$ in $B_1\setminus D$ and $\xi_r\in \mathcal{M}_b$. We can drop the condition $\xi\in L^\infty(B_1)$ by approximation. Then, replacing $\xi$ with $t\xi$ and differentiating the resulting identity with respect to $t$, we obtain
                $$\int_{D}\Big( \nabla w\cdot\nabla \xi+\frac{\gamma}{2}(\gamma-1)b^{\gamma-2} w \xi\Big)\,dx=0,$$ for every $\xi\in H^1(B_1)$ with $\xi\equiv0$ in $B_1\setminus D$ and $\xi_r\in \mathcal{M}_b$.

                Now we proceed by showing the spherical analogue of the previous identity. Let $\tau\in C^\infty_c(\partial B_1\cap\{b>0\})\cap\mathcal{M}_b$, we take $\kappa(r)\in C^\infty_c((0,1))$ and $D$ such that $\xi(r,\theta):=r^\beta\kappa(r)\tau(\theta)\in H^1_0(D)$, with $\kappa\not\equiv0$. Then, by integration by parts, we get \bea\int_0^1 r^{d+2\beta-3}\kappa(r)\,dr\int_{\partial B_1}\Big( \nabla_\theta w(\theta)\cdot\nabla_\theta \tau-\lambda(\beta)w(\theta)\tau+\frac{\gamma}{2}(\gamma-1)b^{\gamma-2} w(\theta)\tau\Big)\,d\HH^{d-1}=0,\eea
                where we follow the notation $w(r,\theta)=r^\beta w(\theta)$. Since $\kappa$ is not trivial, we have \be\label{eq:quasifine}\frac12\delta^2\mathcal{G}(0)[w,\tau]=\int_{\partial B_1}\Big( \nabla_\theta w\cdot\nabla_\theta \tau-\lambda(\beta)w\tau+\frac{\gamma}{2}(\gamma-1)b^{\gamma-2} w\tau\Big)\,d\HH^{d-1}=0,\ee for every $\tau\in C^\infty_c(\partial B_1\cap\{b>0\})\cap \mathcal{M}_b$.

                Now we drop the condition $\tau\in \mathcal{M}_b$ as follows. Define the normed space $(Z,\|\cdot\|_Z)$ by $$Z:=\{\tau\in C^\infty_c(\partial B_1\cap\{b>0\}): \|\tau\|_Z<+\infty\}\cap\mathcal{M}_b\quad\text{and}\quad \|\tau\|_Z:=\|\tau\|_{H^1}+\|b^{\frac{\gamma-2}2}\tau\|_{L^2}.$$
Since $C^\infty_c(\partial B_1\cap\{b>0\})$ is dense, with respect to $\|\cdot\|_Z$, in the corresponding weighted Sobolev space, and since the projection $P_{K_b}$ is continuous, the subspace $C^\infty_c(\partial B_1\cap\{b>0\})\cap\mathcal{M}_b$ is dense in $(Z,\|\cdot\|_Z)$. Indeed, starting from any compactly supported approximation, one can eliminate its $K_b$-component by subtracting a finite linear combination of fixed compactly supported functions whose $P_{K_b}$-projections form a basis of $K_b$.  
By the second estimate in \eqref{eq:proof-step-1-cylindrical-cones} and the homogeneity of the terms involved, we have $b^{\frac{\gamma-2}2}w\in L^2(\partial B_1)$, and thus, using the density in \eqref{eq:quasifine}, we have that $$\delta^2\mathcal{G}(0)[w,\tau]=0\quad\text{for every }\tau\in Z.$$
                Therefore, given $\tau\in C^\infty_c(\partial B_1\cap\{b>0\})$, we write $\tau=\tau_{K_b}+\tau_{\mathcal{M}_b}$, with $\tau_{K_b}\in K_b$ and $\tau_{\mathcal{M}_b}\in \mathcal{M}_b$. By the classification of $K_b$ in \eqref{eq:Kbz-decomposition} and \cref{prop:kernel-cylindric}, the worst case scenario is that $\tau_{K_b}$ behaves like $|y|^{\beta-1}$ at $\{b=0\}$, which still implies that $\|\tau_{K_b}\|_Z <+\infty$. Moreover, since $\tau=0$ in a neighborhood of $\{b=0\}$, then $\|\tau_{\mathcal{M}_b}\|_Z<+\infty$ as well. Thus, we have shown that $\tau_{\mathcal{M}_b}\in Z$ and since $\tau_{K_b}\in K_b$, $\tau_{\mathcal{M}_b}\in Z$, we get $$\delta^2\mathcal{G}(0)[w,\tau]=\delta^2\mathcal{G}(0)[w,\tau_{K_b}]+\delta^2\mathcal{G}(0)[w,\tau_{\mathcal{M}_b}]=0,$$ which is exactly the weak formulation of \eqref{eq:proof-step-2-cylindrical-cones}.
				\\ \vspace{-0.3cm}
				
				\noindent\textit{Step 3.} Now we prove that $w\equiv0$. As already observed in Step 2, we have $w\in \mathcal{M}_b$. Moreover, it follows from \eqref{eq:general-kernel-only-positive} that \eqref{eq:proof-step-2-cylindrical-cones} is equivalent to requiring that $w\in K_b$. Then, since $H^1(\partial B_1)=K_b\oplus \mathcal{M}_b$, necessarily $w\equiv0$.
				\\ \vspace{-0.3cm}
				
				\noindent\textit{Step 4.}
				Finally we prove that $w_j$ converges to $w\equiv0$ strongly in $H^1(B_1)$, which is a contradiction with $\|w_j\|_{H^1(B_1)}=1$.
				We choose the same competitor of Step 1, i.e., $h:=(1-\zeta)z_j+\zeta \overline q_j$. Then, using \eqref{eq:eq:ultimastep4} with \eqref{eq:proof-step-1-cylindrical-cones}, the same computation of Step 7 in \cref{prop:epi-one-dim-cone} gives
				$$\int_{B_{1/2}}|\nabla w_j|^2\,dx+ \int_{B_1}\frac{Q_{q_j}(z_j-q_j)}{\eta_j^2}\,dx \le C\widetilde \eps_j.$$
				The conclusion follows by combining the homogeneity of $w_j$ with  $Q_{q_j}(z_j-q_j)\ge0$.
			\end{proof}
            \begin{remark}\label{remark:sigma=0-cylindrical}
                If $b\in\mathcal{B}_\ell$, for $\ell=1,\ldots,d-2$, is sub-integrable, by \cref{remark:sub-int-ifandonlyif-sub-int} only case (i) in the proof of \cref{t:epi-cylindrical} can happen. Then the epiperimetric inequality holds with $\sigma=0$.
            \end{remark}
            \begin{remark}\label{remark:intermediate-cones}
                If $b\in\mathcal{B}_\ell$, for some $\ell=1,\ldots,d-2$, is not sub-integrable, then we can apply \cref{lemma:adamsimon} to $B$, obtaining that there exists a solution $u$ of the Alt-Phillips problem satisfying the logarithmic convergence \eqref{e:log-decay}. Moreover, since $\gamma\in(1,2)$, such a solution is also a minimizer (see \cref{prop:convex-functional}).
            \end{remark}

			\section{Proof of the main results}\label{section:final}
			In this section we conclude the proof of our main results. We first prove the epiperimetric inequality in \cref{thm:epiperimetric-combined}.
             \begin{proof}[Proof of \cref{thm:epiperimetric-combined}]
			   The result follows by combining \cref{t:epi-rad}, \cref{prop:epi-one-dim-cone} and \cref{t:epi-cylindrical}.
			\end{proof}
            The next step is to prove that the epiperimetric inequality in \cref{thm:epiperimetric-combined} can be applied to every scale.
            \begin{proposition}\label{prop:application-epi}
				Let $u$ be a minimizer of the Alt-Phillips problem in $B_1$, and suppose that $b$ is a blow-up of $u$ at $0\in\partial\Omega_u$ satisfying \cref{condition:assumption-uniqueness}. Then, for every $r$ small enough, the traces $u_r\big|_{\partial B_1}$ satisfy the hypotheses of the epiperimetric inequality in \cref{thm:epiperimetric-combined}.
			\end{proposition}
            We use the following lemma.
\begin{lemma}\label{lemma:estimate-application}
    Let $d\geq 2$, $\gamma \in (0,2)$, then, there exists $\tau>0$ depending only on $d$ and $\gamma$ such that the following holds.
    Let $u$ be a solution of the Alt-Phillips problem in $B_1$, and suppose that $b$ is a blow-up of $u$ at $0\in\partial\Omega_u$ satisfying \cref{condition:assumption-uniqueness}. We denote by $z_r$ the $\beta$-homogeneous extension of the trace $u_r|_{\partial B_1}$. Then $$\|u_r-b\|_{L^\infty(\partial B_{1})}+\|z_r-b\|_{H^1(B_{1})}+|W(z_r)-W(b)|\le C \|u_r-b\|^\tau_{L^2(\partial B_1)}.$$
\end{lemma}
\begin{proof}
For a function $\phi\in C^{1,\alpha}(\partial B_1)$, with $\alpha\in(0,1)$, we use the following interpolation inequalities 
\be\label{eq:interpolation-inequalities}
\|\phi\|_{L^\infty(\partial B_1)}\le C\|\phi\|_{C^{0,1}(\partial B_1)}^{\frac{d-1}{d+1}}\|\phi\|_{L^2(\partial B_1)}^\frac{2}{d+1}\quad\text{and} \quad\|\nabla_\theta \phi\|_{L^2(\partial B_1)}\le C\| \phi\|_{C^{1,\alpha}(\partial B_1)}^\frac{1}{1+\alpha}\| \phi\|_{L^2(\partial B_1)}^{\frac{\alpha}{1+\alpha}}.\ee
Indeed, the first can be proved as in \cite[Lemma 3.2]{sv19}, while the second follows by an elementary finite-difference argument. By the $C^\beta$ regularity of solutions of the Alt-Phillips problem in \cref{prop:regularity-of-solutions-alt-phillips}, we have $\|u_r\|_{C^\beta(B_1)}\le C$ and $\|b\|_{C^\beta(B_1)}\le C$. Then, applying the interpolation inequalities in \eqref{eq:interpolation-inequalities} to $\phi=u_r-b$, we have
$$
\|u_r-b\|_{L^\infty(\partial B_{1})}\le C \|u_r-b\|^\frac{2}{d+1}_{L^2(\partial B_1)}
$$ and, by homogeneity, 
\bea
\|z_r-b\|_{H^1(B_{1})}\le C\|\nabla_\theta(z_r-b)\|_{L^2(\partial B_{1})}+C\|z_r-b\|_{L^2(\partial B_{1})}\le C\|z_r-b\|^{\frac{\alpha}{1+\alpha}}_{L^2(\partial B_{1})}.
\eea
Now let $\varphi\in C^\beta(\partial B_1)$ be the trace of $z_r-b$ on $\partial B_1$. Then, by slicing \cref{lemma:slicing} and using the definition of $\mathcal{G}$ in \eqref{def:G}, we have 
$$|W(z_r)-W(b)|=C|\mathcal{G}(\varphi)|\le C\|\nabla_\theta \varphi\|^2_{L^2(\partial B_1)}+C\| \varphi\|^2_{L^2(\partial B_1)}+ C\int_{\partial B_1}|Q_b(\varphi)|\,d\HH^{d-1},$$ where $Q_b(\varphi):=(b+\varphi)^\gamma-b^\gamma-\gamma b^{\gamma-1}\varphi$. At this point, we have two possibilities:
\begin{itemize}
    \item if $b \in \mathcal{B}_0$, then $\gamma \in (0,2)$ and we apply $|(1+s)^\gamma-1-\gamma s|\le Cs^2$ with $s=\varphi/b$;
    \item if $b \in \mathcal{B}\setminus \mathcal{B}_0$, then $\gamma \in (1,2)$ and we apply $|(1+s)^\gamma-1-\gamma s|\le C|s|^\gamma$ with $s=\varphi/b$.
\end{itemize}
Thus, we have that  
\be\label{eq:stimaRb}|Q_b(\varphi)|\le C\varphi^2\quad\text{if }\gamma\in(0,2)\quad\text{and}\quad |Q_b(\varphi)|\le C|\varphi|^\gamma\quad\text{if }\gamma>1.\ee
Combining the second interpolation inequality in \eqref{eq:interpolation-inequalities} with \eqref{eq:stimaRb}, and recalling that $\varphi=z_r-b$ on $\partial B_1$, we infer the existence of $\tau>0$ so that $$|W(z_r)-W(b)|\le C\|z_r-b\|^\tau_{L^2(\partial B_1)},$$
concluding the proof, since $z_r=u_r$ on $\partial B_1$.
\end{proof}
            
            \begin{proof}[Proof of \cref{prop:application-epi}]
Let $\rho>0$ and $\delta_1=\delta_1(\rho)>0$ to be chosen later.
Since $b$ is a blow-up of $u$ at $0$, then, up to rescaling $u$, we can assume that \be\label{eq:eq2final}\|u-b\|_{L^2(\partial B_1)}\le \delta_1 \quad\text{and}\quad |W(u)-W(b)|\le \delta_1
.\ee 
By \cref{lemma:estimate-application}, for every $r\in [\rho/2,1]$, we have
\be\label{eq:fundamental-application}\|u_r-b\|_{L^\infty(\partial B_{1})}+\|z_r-b\|_{H^1(B_{1})}+|W(z_r)-W(b)|\le C \|u_r-b\|^\tau_{L^2(\partial B_1)},\ee where $z_r$ is the $\beta$-homogeneous extension of the trace $u_r|_{\partial B_1}$.

Integrating the monotonicity formula of the Weiss' energy \eqref{eq:weiss-formula}, we obtain \be\label{eq:eq1final}\|u_{r_2}-u_{r_1}\|_{L^2(\partial B_1)}\le C \log\left(\frac{r_2}{r_1}\right)^{1/2}(W(u_{r_2})-W(b))^{1/2}\quad\text{for every } 0<r_1\le r_2\le 1,\ee where we used that $W(u_{r_1})-W(b)\ge0$.
Combining \eqref{eq:eq2final} with \eqref{eq:eq1final} applied with $r_1=r$ and $r_2=1$, we obtain, for every $r\in[\rho/2,1]$,
\be\label{utiledopo-application}\|u_r-b\|_{L^2(\partial B_1)}\le \|u-u_r\|_{L^2(\partial B_1)}+\|u-b\|_{L^2(\partial B_1)}\le C\left( \log\frac2\rho\right)^{1/2}\delta_1^{1/2}+\delta_1\le C_\rho\delta_1^{1/2},\ee 
for some $C_\rho>0$ depending on $\rho$.
Therefore, by \eqref{eq:fundamental-application}, for every $r\in [\rho/2,1]$, we have 
$$\|u_r-b\|_{L^\infty(\partial B_{1})}+\|z_r-b\|_{H^1(B_{1})}+|W(z_r)-W(b)|\le CC_\rho^\tau\delta_1^{\tau/2}.$$
Let $\delta$ be the constant appearing in  \cref{thm:epiperimetric-combined}, and we choose $\delta_1=\delta_1(\rho)>0$ sufficiently small so that $CC_\rho^\tau\delta_1^{\tau/2}\le \delta$. Then, we can apply the epiperimetric inequality \cref{thm:epiperimetric-combined} to the rescaling $u_r$, for every $r\in[\rho/2,1]$. 

Let $r_0\in[0,\rho/2)$ be the minimum of the radii such that the epiperimetric inequality applies to $u_r$, for every $r\in(r_0,1)$. We claim that $r_0=0$. Suppose, by contradiction, that $r_0>0$. Then, by applying the epiperimetric inequality
$$\frac{d}{dr}\Big(W(u_r)-W(b)\Big)\ge \frac cr\Big((W(z_r)-W(b))-(W(u_r)-W(b))\Big)\ge \frac{c\eps}{r}(W(u_r)-W(b))^{1+\sigma}$$
for every $r \in (r_0,1)$. Then, by direct integration, we obtain the following decay of the Weiss' energy $$0\le W(u_r)-W(b)\le \frac{C}{|\log r|^\alpha},\quad\text{for every }r\in(r_0,1),$$ for some $\alpha>0$. Combining this decay with a standard dyadic argument (see e.g.~\cite{csv18}), we get \be\label{eq:application-epi}\|u_r-u_\rho\|_{L^2(\partial B_1)}\le \frac{C}{|\log \rho|^\alpha}\quad\text{for every }r\in [r_{0},\rho).\ee
Thus, by \eqref{eq:eq1final}, \eqref{utiledopo-application} and  \eqref{eq:application-epi}, for every $r'\in(r_{0}/2,r_{0})$, we have \bea \|u_{r'}-b\|_{L^2(\partial B_1)}&\le \|u_{r'}-u_{r_{0}}\|_{L^2(\partial B_1)}+\|u_{r_{0}}-u_\rho\|_{L^2(\partial B_1)}+\|u_\rho-b\|_{L^2(\partial B_1)}\\&\le C\log(2)^\frac12\delta_1^{1/2}+\frac{C}{|\log\rho|^\alpha}+C_\rho\delta_1^{1/2}.\eea 
Therefore, by \eqref{eq:fundamental-application} $$\|u_r-b\|_{L^\infty(\partial B_{1})}+\|z_r-b\|_{H^1(B_{1})}+|W(z_r)-W(b)|\le C\left(C\log(2)^\frac12\delta_1^{1/2}+\frac{C}{|\log\rho|^\alpha}+C_\rho\delta_1^{1/2}\right)^\tau. $$
Choosing first $\rho>0$ sufficiently small and then $\delta_1=\delta_1(\rho)>0$ sufficiently small, we can ensure that the right-hand side above is bounded by $\delta$. Therefore, $u_{r'}$ satisfies the hypotheses of \cref{thm:epiperimetric-combined}, for every $r'\in(r_0/2,r_0)$, contradicting the definition of $r_0$.
\end{proof}           
            \begin{proof}[Proof of \cref{t:uniqueness}]
				By \cref{prop:application-epi}, the epiperimetric inequality in \cref{thm:epiperimetric-combined} can be applied at every scale. The uniqueness of the blow-up limit and the corresponding rate of convergence then follow as in \eqref{eq:application-epi}.          
			\end{proof}
            \begin{proof}[Proof of \cref{t:ap-low-dimension}] 
            By combining \cref{prop:trans0} with \cref{t:uniqueness}, we get uniqueness of the blow-up limit for $d$ and $\gamma$ satisfying \eqref{eq:assumption-dimension}. Then, the stratification of the free boundary follows by the implicit function theorem and Whitney’s extension theorem in \cite{fefferman} (see e.g.~\cite{gp09,fs19,csv20}). Finally, the regularity result for $\Sigma_{d-1}(u)$ follows by \cref{thm:parabola-cones}.
			\end{proof}
            \begin{proof}[Proof of \cref{t:rate-convergence}]
            If $b$ is sub-integrable, by \cref{remark:sigma=0} and \cref{remark:sigma=0-cylindrical}, the epiperimetric inequalities hold with $\sigma=0$, and hence yield a polynomial rate of convergence. On the other hand, if $b$ is not sub-integrable, the conclusion follows from \cref{lemma:adamsimon} and \cref{remark:intermediate-cones}.
            \end{proof}
            \begin{proof}[Proof of \cref{thm:parabola-cones}]
            The result follows from the epiperimetric inequality in \cref{thm:epiperimetric-combined} applied at every scale (see \cref{prop:application-epi}) and by \cref{t:rate-convergence}, combined with \cref{prop:optimality-log}, \cref{remark:non-int-one-dim-cone} and
            \cref{corollary:kernel-cylindrical}.
            \end{proof}
			\bibliographystyle{plain}
			\bibliography{FreeBoundary_bib}
		\end{document}